\documentclass[reqno]{amsart}
\usepackage{fullpage, verbatim, xy, amssymb}
\usepackage{amsmath}
\usepackage{amssymb}
\usepackage{amsfonts}
\usepackage{graphicx}

\xyoption{all}
\newtheorem{theorem}{Theorem}[section]

\newtheorem{corollary}[theorem]{Corollary}

\newtheorem{definition}[theorem]{Definition}

\newtheorem{lemma}[theorem]{Lemma}

\newtheorem{proposition}[theorem]{Proposition}
\newtheorem{remark}[theorem]{Remark}

\input{tcilatex}
\begin{document}
\title{Dirac operators in tensor categories and the motive of quaternionic modular forms}
\author{Marc Masdeu}
\email{masdeu@mat.uab.cat}
\address{Department of Mathematics, University of Warwick, Coventry, United Kingdom}
\author{Marco Adamo Seveso}
\email{seveso.marco@gmail.com}
\address{Dipartimento di Matematica "Federigo Enriques", Universit\`a di Milano, Via Cesare Saldini 50, 20133 Milano, Italy}

\begin{abstract}
We define a motive whose realizations afford modular forms (of arbitrary weight)
on an indefinite division quaternion algebra. This generalizes work of Iovita--Spiess to odd weights
in the spirit of Jordan--Livn\'e. It also generalizes a construction of Scholl to indefinite division quaternion algebras, and provides the first motivic construction of new-subspaces of modular forms.
\end{abstract}
\maketitle
\tableofcontents

\section{Introduction}

The paper \cite{Sc} offers the construction of a motive whose realizations
afford modular forms of even or odd weight on the indefinite split
quaternion algebra over $\mathbb{Q}$. In~\cite[\S 10]{IS} the authors
construct a motive of even weight modular forms on a quaternion division
algebra (see also~\cite{Wo}). Based on ideas of Jordan and Livn\'{e} (see 
\cite{JL}), this motive is constructed as the kernel of an appropriate
Laplace operator. More precisely, let $h(A)$ be the motive of an abelian
scheme $A$ of relative dimension $d$ over a smooth base scheme $S$ (see \cite{DM} and \cite{Ku}). It decomposes as
the direct sum%
\begin{equation*}
h\left( A\right) =h^{0}\left( A\right) \oplus h^{1}\left( A\right) \oplus
...\oplus h^{g}\left( A\right)
\end{equation*}%
where $g=2d$ and there are canonical identifications%
\begin{equation*}
h^{i}\left( A\right) =\vee ^{i}h^{1}\left( A\right) \text{, }h^{i}\left(
A\right) \simeq h^{2d-i}\left( A\right) ^{\vee }\left( -d\right) \text{ and }%
h^{2d}\left( A\right) \simeq \mathbb{I}\left( -d\right),
\end{equation*}%
where $\vee ^{\cdot }V$ denotes the symmetric algebra of the object $V$. It
follows that the multiplication morphisms%
\begin{equation*}
\varphi _{i,2d-i}:\vee ^{i}h^{1}\left( A\right) \otimes \vee
^{2d-i}h^{1}\left( A\right) \rightarrow \mathbb{I}\left( -d\right)
\end{equation*}%
are perfect. In particular, taking $i=d$, one gets an associated Laplace
operator\footnote{%
For a symmetric or alternating power $M$ we will write $\mathrm{Sym}%
^{n}\left( M\right) $ and $\mathrm{Alt}^{n}\left( M\right) $ when
considering its symmetric or alternating powers once again.}%
\begin{equation*}
\Delta ^{n}:\mathrm{Sym}^{n}\left( \vee ^{d}h^{1}\left( A\right) \right)
\rightarrow \mathrm{Sym}^{n-2}\left( \vee ^{d}h^{1}\left( A\right) \right)
\left( -d\right) \text{, }n\geq 2
\end{equation*}%
and it is possible to show that the kernel exists. The following remark has
been employed in \cite[\S 10]{IS}. When $A$ is an abelian scheme of
dimension $d=2$ with multiplication by the quaternion algebra $B$, we have
that $B\otimes B$ acts on $\vee ^{2}h^{1}\left( A\right) $ and there is a
canonical direct sum decomposition 
\begin{equation*}
\vee ^{2}h^{1}\left( A\right) =\left( \vee ^{2}h^{1}\left( A\right) \right)
_{+}\oplus \left( \vee ^{2}h^{1}\left( A\right) \right) _{-}
\end{equation*}%
is such a way that $B^{\times }\subset B\otimes B$ (diagonally)\ acts via
the reduced norm on $\left( \vee ^{2}h^{1}\left( A\right) \right) _{-}$.
Furthermore, since the idempotents giving rise to this decomposition are
self-adjoint with respect to $\varphi _{2,2}$, it follows that the induced
pairing%
\begin{equation*}
\left( \vee ^{2}h^{1}\left( A\right) \right) _{-}\otimes \left( \vee
^{2}h^{1}\left( A\right) \right) _{-}\hookrightarrow \vee ^{2}h^{1}\left(
A\right) \otimes \vee ^{2}h^{1}\left( A\right) \rightarrow \mathbb{I}\left(
-2\right)
\end{equation*}%
is still non-degenerate and the kernel of the induced Laplace operator%
\begin{equation*}
\Delta _{-}^{n}:\mathrm{Sym}^{n}\left( \left( \vee ^{2}h^{1}\left( A\right)
\right) _{-}\right) \rightarrow \mathrm{Sym}^{n-2}\left( \left( \vee
^{2}h^{1}\left( A\right) \right) _{-}\right) \left( -2\right) \text{, }n\geq
2
\end{equation*}%
exists. When $A$ is taken to be the universal abelian surface, setting%
\begin{equation*}
M_{2n}:=\ker \left( \Delta ^{n}\right)
\end{equation*}%
gives a motive whose realizations gives incarnations of weight $k=2n+2$
modular forms.

The aim of this paper is to define a motive whose realizations afford
modular forms (of arbitrary weight) on an indefinite division quaternion
algebra. The idea of the construction, once again, is due to Jordan and Livn%
\'{e}. However some remarks are in order. First, it is worth noting that although the realizations of the motive
constructed in this paper are abstractly isomorphic to the $D=\operatorname{disc}(B)$-new part
of (two copies of) the realizations of the motive constructed in~\cite{Sc} via the Jacquet--Langlands
correspondence, a ``motivic Jacquet--Langlands correspondence'' has not yet described
that lifts this correspondence to the motivic setting. Therefore what we propose
is the first construction --as a Chow motive-- of $D$-new modular forms.
Second, following their
construction in this indefinite setting and working at the level of
realizations gives the various incarnations of two copies of odd weight
modular forms, rather than just one copy. It is not possible to canonically
split them in a single copy: this is possible only including a splitting
field for the quaternion algebra in the coefficients, but the resulting
splitting depends on the choice of an identification of the base changed
algebra with the split quaternion algebra. Indeed, we will construct a
motive whose realizations afford two copies of odd weight modular forms.
Finally, the idea of Jordan and Livn\'{e} is to construct square roots of the
Laplace operators after appropriately tensoring the source and the targets
of $\Delta ^{n}$; however the definition of these Dirac operators $\partial
_{JL}^{n}$ such that $\partial _{JL}^{n-1}\circ \partial _{JL}^{n}=\Delta
^{n}\otimes 1_{?}$ does not readily generalize to the setting of a rigid $%
\mathbb{Q}$-linear and pseudo-abelian $ACU$ category like that of motives.
To understand the linear algebra behind their construction, let us consider
the category \textrm{Rep}$\left( B^{\times }\right) $\ of algebraic $%
B^{\times }$-representations: let $B$ (resp. $B^{\iota }$) be the $B^{\times
}$-representation whose underlying vector space is $B$ on which $B^{\times }$
acts by left multiplication (resp. $b\cdot x=bxb^{\iota }$, where $b\mapsto
b^{\iota }$ denotes the main involution) and set $B_{0}:=\ker \left( \mathrm{%
Tr}\right) \subset B^{\iota }$. Then the trace form $\left\langle
x,y\right\rangle :=\mathrm{Tr}\left( x^{\iota }y\right) $ induces $B^{\iota
}\otimes B^{\iota }\rightarrow \mathbb{Q}\left( -2\right) $ and $B\otimes
B\rightarrow \mathbb{Q}\left( -1\right) $ and the first is perfect when
restricted to $B_{0}$ and gives Laplace operators%
\begin{equation}
\Delta _{-}^{n}:\mathrm{Sym}^{n}\left( B_{0}\right) \rightarrow \mathrm{Sym}%
^{n-2}\left( B_{0}\right) \left( -2\right) \text{, }n\geq 2\text{,}
\label{Intro F0}
\end{equation}%
while the second gives%
\begin{equation}
B=B^{\vee }\left( -1\right) \text{.}  \label{Intro F1}
\end{equation}%
We may realize $B_{0}=\left( \wedge ^{2}B\right) _{-}$ and it follows that $%
\left( \text{\ref{Intro F0}}\right) $ may be regarded as%
\begin{equation*}
\Delta _{-}^{n}:\mathrm{Sym}^{n}\left( \left( \wedge ^{2}B\right)
_{-}\right) \rightarrow \mathrm{Sym}^{n-2}\left( \left( \wedge ^{2}B\right)
_{-}\right) \left( -2\right) \text{, }n\geq 2\text{.}
\end{equation*}%
Following ideas of \cite[\S 10]{IS}, one can realize the various incarnation
of modular forms as the image via an appropriate additive $ACU$ tensor
functor%
\begin{equation*}
\mathcal{L}:\mathrm{Rep}\left( B^{\times }\right) \rightarrow \mathcal{H}%
\text{,}
\end{equation*}%
where $\mathcal{H}$ is the category we are interested in, i.e. they may be
for example variations of Hodge structures. Indeed, if $R$ is the
realization functor one shows that $R\left( \left( \vee ^{2}h^{1}\left(
A\right) \right) _{-}\right) =\mathcal{L}\left( \left( \wedge ^{2}B\right)
_{-}\right) $, from which it follows%
\begin{equation*}
R\left( M_{2n}\right) =\mathcal{L}\left( \ker \left( \Delta _{-}^{n}\right)
\right) 
\end{equation*}%
and $\mathcal{L}\left( \ker \left( \Delta ^{n}\right) \right) $ computes
weight $2n+2$ modular forms. If one is interested in \emph{odd} weight
modular forms, the Jordan and Livn\'{e} Dirac operators to be considered
would be of this form:%
\begin{equation}
\partial _{JL}^{n}:\mathrm{Sym}^{n}\left( B_{0}\right) \otimes B\rightarrow 
\mathrm{Sym}^{n-1}\left( B_{0}\right) \otimes B\left( -1\right) \text{, }%
n\geq 1\text{.}  \label{Intro F2}
\end{equation}%
However, as we have explained above, the Jordan and Livn\'{e} definition of
these operators as given in \cite{JL} does not generalize to motives. It is
a simple but key remark that one may replace $\partial _{JL}^{n}$ with any $%
\partial ^{n}$ having the same source and target and then the kernels would
be the same (see Lemma \ref{Realizations L1}). Furthermore, since $%
B_{0}=\left( \wedge ^{2}B\right) _{-}$, it follows from $\left( \text{\ref%
{Intro F1}}\right) $ that $\left( \text{\ref{Intro F2}}\right) $ with $%
\partial _{JL}^{n}$ replaced by $\partial ^{n}$ may be regarded as%
\begin{equation}
\partial ^{n}:\mathrm{Sym}^{n}\left( \left( \wedge ^{2}B\right) _{-}\right)
\otimes B\rightarrow \mathrm{Sym}^{n-1}\left( \left( \wedge ^{2}B\right)
_{-}\right) \otimes B^{\vee }\left( -2\right) \text{, }n\geq 1\text{.}
\label{Intro F3}
\end{equation}%
It is in this form that we will be able to define $\partial ^{n}$ and
another $\overline{\partial }^{n-1}$ in such a way that the construction
makes sense for rigid $\mathbb{Q}$-linear and pseudo-abelian $ACU$
categories and prove the generalization of the equality $\overline{\partial }%
^{n-1}\circ \partial ^{n}=\Delta _{-}^{n}\otimes 1_{B}$ in this setting.
Then one shows that $\mathcal{L}\left( \partial ^{n}\right) $ computes two
copies of weight $k=2n+3$ modular forms.%

The abstract framework we work with in this paper is the following. Suppose
that $\mathcal{C}$ is a rigid pseudo-abelian and $\mathbb{Q}$-linear $ACU$
tensor category with identity object $\mathbb{I}$; if $X\in \mathcal{C}$ we
write $r_{X}:=\mathrm{rank}\left( X\right) $. We recall from \cite{MS} that $%
V$ has \emph{alternating} (resp. \emph{symmetric}) rank $g\in \mathbb{N}%
_{\geq 1}$ if $L:=\wedge ^{g}V$ (resp. $L:=\vee ^{g}V$) is invertible and if 
$\binom{r+i-g}{i}$\ (resp. $\binom{r+g-1}{i}$) is invertible in $End\left( 
\mathbb{I}\right) $ for every $0\leq i\leq g$. Here, for an integer $k\geq 1$%
,%
\begin{equation*}
\binom{T}{k}:=\frac{1}{k!}T\left( T-1\right) ...\left( T-k+1\right) \in 
\mathbb{Q}\left[ T\right] \text{ and }\binom{T}{0}=1\text{.}
\end{equation*}

Suppose first that $V$ has alternating rank $g$. We will prove that, when $%
g=2i$ and $i$ is even (resp. odd), $L\simeq \mathbb{L}^{\otimes 2}$ for some
invertible object and $r_{\wedge ^{i}V}>0$ (resp. $r_{\wedge ^{i}V}<0$)\
(see definition \ref{Dirac definition positive rank}), then there is an
operator%
\begin{eqnarray*}
&&\text{ }\partial _{i-1}^{n}:\mathrm{Sym}^{n}\left( \wedge ^{i}V\right)
\otimes \wedge ^{i-1}V\rightarrow \mathrm{Sym}^{n-1}\left( \wedge
^{i}V\right) \otimes V^{\vee }\otimes L\text{, }n\geq 1 \\
&&\text{ (resp. }\partial _{i-1}^{n}:\mathrm{Alt}^{n}\left( \wedge
^{i}V\right) \otimes \wedge ^{i-1}V\rightarrow \mathrm{Alt}^{n-1}\left(
\wedge ^{i}V\right) \otimes V^{\vee }\otimes L\text{, }n\geq 1%
\text{)}
\end{eqnarray*}%
such that $\ker \left( \partial _{i-1}^{n}\right) $ exists (see Theorems \ref%
{Dirac Alternating T2} and \ref{Dirac Alternating T1}).

Suppose now that $V$ has symmetric rank $g$. Then we prove that, when $g=2i$%
, $L\simeq \mathbb{L}^{\otimes 2}$ for some invertible object and $r_{\vee
^{i}V}>0$, then there is an operator%
\begin{equation*}
\partial _{i-1}^{n}:\mathrm{Sym}^{n}\left( \vee ^{i}V\right) \otimes \vee
^{i-1}V\rightarrow \mathrm{Sym}^{n-1}\left( \vee ^{i}V\right) \otimes V^{\vee }\otimes L\text{, }n\geq 1
\end{equation*}%
such that $\ker \left( \partial _{i-1}^{n}\right) $ exists (see Theorem \ref%
{Dirac Symmetric T1}).

These operators are indeed square roots of the Laplace operators induced by
the multiplication pairings in the involved alternating or symmetric
algebras and the existence of these kernels follows from this fact and the
existence of the kernels of the Laplace operators.

Some remarks are in order about the range of applicability of our results.
First of all we note that, in general, the alternating or the symmetric rank
may be not uniquely determined. Suppose, however, that we know that there is
a field $K$ such that $r\in K\subset End\left( \mathbb{I}\right) $ admitting
an embedding $\iota :K\hookrightarrow \mathbb{R}$. Then it follows from the
formulas $\mathrm{rank}\left( \wedge ^{k}V\right) =\binom{r}{k}$ and $%
\mathrm{rank}\left( \vee ^{k}V\right) =\binom{r+k-1}{k}$ (see \cite[7.2.4
Proposition]{AKh} or \cite[$\left( 7.1.2\right) $]{De}) that we have $r\in
\left\{ -1,g\right\} $ (resp. $r\in \left\{ -g,1\right\} $) when $V$ has
alternating (resp. symmetric)\ rank $g$. In particular, when $r>0$ (resp. $%
r<0$) with respect to the ordering induced by $\iota $, we deduce that $r=g$
(resp. $r=-g$), so that $g$ is a uniquely determined and $V$ has alternating
(resp. symmetric)\ rank $g=r$ (resp. $g=-r$)

We recall that $V$ is Kimura positive (resp. negative) when $\wedge
^{N+1}V=0 $ (resp. $\vee ^{N+1}V=0$) for $N\geq 0$ large enough. In this
case, the formula $\mathrm{rank}\left( \wedge ^{k}V\right) =\binom{r}{k}$
(resp. $\mathrm{rank}\left( \vee ^{k}V\right) =\binom{r+k-1}{k}$) implies
that $r\in \mathbb{Z}_{\geq 0}$ (resp. $r\in \mathbb{Z}_{\leq 0}$)\ and the
smallest integer $N$ such that $\wedge ^{N+1}V=0$ (resp. $\vee ^{N+1}V=0$)
is $r$ (resp. $-r$). Furthermore, it is known that $\wedge ^{r}V$ (resp. $%
\vee ^{-r}V$) is invertible in this case (see \cite[11.2 Lemma]{Kh}): in
other words $V$ has alternating (resp. symmetric) rank $g=r$ (resp. $g=-r$).
Suppose in particular that $V$ is Kimura positive (resp. negative); then $%
r_{\wedge ^{i}V}>0$ (resp. $r_{\vee ^{i}V}>0$)\ for $i$ even and Theorem \ref%
{Dirac Alternating T2} (resp. Theorem \ref{Dirac Symmetric T1}) applies. On
the other hand, when $i$ is odd, the condition $r_{\wedge ^{i}V}<0$ (resp. $%
r_{\vee ^{i}V}>0$) required by Theorem \ref{Dirac Alternating T2} (resp.
Theorem \ref{Dirac Symmetric T1}) is not satisfied and we cannot apply our
results.

It is known that the motive $h^{1}\left( A\right) $ of an abelian scheme of
dimension $d$ is Kimura negative of Kimura rank $2d$ (see \cite[Definitions 3.8 and 6.4]{Ki1} for
the precise definitions).
Suppose that $d=2i\equiv 0$ $\func{mod}4$, so that $i$ is even and $r_{\vee
^{i}V}>0$. Since $d$ is even, $\vee ^{2d}h^{1}\left( A\right) \simeq
h^{2d}\left( A\right) \simeq \mathbb{I}\left( -d\right) $ is the square of
an invertible object. Theorem \ref{Dirac Symmetric T1} applied to $%
V=h^{1}\left( A\right) $ implies the existence of canonical pieces%
\begin{equation*}
\ker \left( \partial _{d/2-1}^{n}\right) \subset \mathrm{Sym}^{n}\left( \vee
^{d/2}h^{1}\left( A\right) \right) \otimes \vee ^{d/2-1}h^{1}\left( A\right)
\simeq \mathrm{Sym}^{n}\left( h^{d/2}\left( A\right) \right) \otimes
h^{d/2-1}\left( A\right)
\end{equation*}%
for every $n\geq 1$. Note that in~\cite{Ku} there is a different notation that is being used for the symmetric powers of a motive, namely $\Lambda^\cdot$, which is the authors' opinion can be slightly misleading.

The paper is organized as follows. In \S 2 we recall the needed results from 
\cite{MS}. In \S 3 we discuss generalities on Laplace and Dirac operators in
rigid and pseudo-abelian tensor categories, giving condition for the
existence of kernel of Laplace operators and for the Dirac operators to be
square roots of Laplace operators. We remark that the existence of kernels
of Laplace operators is stated in \cite[\S 10]{IS} for the category of Chow
motives; the authors are indebted with M. Spiess for providing them some
notes on the topics. In \S 4 and \S 5 we use the Poincar\'{e} morphisms from 
\S 2 to define our Dirac operators on the alternating and symmetric powers
and prove that they are indeed square roots of the Laplace operators;
together with the result from \S 3 we deduce Theorems \ref{Dirac Alternating
T2}, \ref{Dirac Alternating T1} and \ref{Dirac Symmetric T1}. In \S 6 we
discuss how the constructions behaves with respect to additive $AU$ tensor
functors which may not respect the associativity constraint, as needed for
the realization functor $R$ (see \cite{Ku}); we also apply the results to
the specific case of a quaternionic object, as needed for the construction
of the motives of modular forms. The subsequent section is devoted to the
computation of the realization of the motives of modular forms: the reader
is strongly suggested to first give a look to this section as a motivation
for the abstract constructions. We work with variations of Hodge structures
as a target category, following ideas of \cite{IS}, but the same
computations could be worked out for other realizations following the same
pattern.

Throughout this paper we will always work in a $\mathbb{Q}$-linear rigid and
pseudo-abelian $ACU$ category $\mathcal{C}$ with unit object $\mathbb{I}$
and internal $\hom $s. We let $ev_{X}:X^{\vee }\otimes X\rightarrow \mathbb{I%
}$ be the evaluation and $ev_{X}^{\tau }:=ev_{X}\circ \tau _{X,X}:$ $\
X\otimes X^{\vee }\rightarrow \mathbb{I}$ be the opposite evaluation.

\subsection*{Acknowledgements}
The authors wish to thank Michael Spiess for encouraging us to start this project and providing
us with helpful conversations. Masdeu was supported by MSC--IF--H2020--ExplicitDarmonProg.

\section{Poincar\'{e} duality isomorphism}

Given an object $V\in \mathcal{C}$, we may consider the associated
alternating and symmetric algebras, denoted by $\wedge ^{\cdot }V$ and,
respectively, $\vee ^{\cdot }V$. If $A_{\cdot }$ denotes one of these
algebras, we have multiplication morphisms%
\begin{equation*}
\varphi _{i,j}:A_{i}\otimes A_{j}\rightarrow A_{i+j}\text{,}
\end{equation*}%
a data which is equivalent to%
\begin{equation*}
f_{i,j}:A_{i}\rightarrow \hom \left( A_{j},A_{i+j}\right) \text{.}
\end{equation*}%
When $g\geq i$, we may consider the composite%
\begin{equation*}
D^{i,g}:A_{i}\overset{f_{i,g-i}}{%
\rightarrow }\hom \left( A_{g-i},A_{g}\right) \overset{d}{%
\rightarrow }\hom \left( A_{g}^{\vee },A_{g-i}^{\vee }\right) \overset{%
\alpha ^{-1}}{\rightarrow }A_{g-i}^{\vee }\otimes A_{g}^{\vee \vee }\text{,}
\end{equation*}%
where $d:\hom \left( X,Y\right) \rightarrow \hom \left( Y^{\vee },X^{\vee
}\right) $ is the internal duality morphism and $\alpha :\hom \left(
X,Y\right) \rightarrow Y\otimes X^{\vee }$ is the canonical morphism (see 
\cite[\S 2]{MS}). Working with the alternating or symmetric
algebra of the dual $V^{\vee }$ yields a morphism%
\begin{equation*}
D^{i,g}:A_{i}^{\vee }\overset{f_{i,g-i}}{%
\rightarrow } \hom \left( A_{g-i}^{\vee },A_{g}^{\vee
}\right) \overset{d}{\rightarrow }\hom \left( A_{g}^{\vee \vee
},A_{g-i}^{\vee \vee }\right) \overset{\alpha ^{-1}}{\rightarrow }%
A_{g-i}^{\vee \vee }\otimes A_{g}^{\vee \vee \vee }\text{.}
\end{equation*}%
Employing the reflexivity morphism $i:X\rightarrow X^{\vee \vee }$ we can
define (see \cite[$\left( 20\right) $]{MS}):%
\begin{equation*}
D_{i,g}:A_{i}^{\vee }\overset{D^{i,g}}{%
\rightarrow } A_{g-i}^{\vee \vee }\otimes A_{g}^{\vee
\vee \vee }\overset{i^{-1}\otimes i^{-1}}{\rightarrow }A_{g-i}\otimes
A_{g}^{\vee }\text{.}
\end{equation*}

\bigskip

The following results have been proved in \cite[\S 5 and \S 6]{MS}. In order to state
them, we first need to define the following morphisms:%
\begin{eqnarray*}
\varphi _{i,j}^{13} &:&A_{i}\otimes B\otimes A_{j}\otimes C\overset{1\otimes
\tau _{B,A_{j}}\otimes 1}{\rightarrow }A_{i}\otimes A_{j}\otimes B\otimes C%
\overset{\varphi _{i,j}\otimes 1}{\rightarrow }A_{i+j}\otimes B\otimes C%
\text{,} \\
\varphi _{i,j}^{13} &:&A_{i}^{\vee }\otimes B\otimes A_{j}^{\vee }\otimes C%
\overset{1\otimes \tau _{B,A_{j}}\otimes 1}{\rightarrow }A_{i}^{\vee
}\otimes A_{j}^{\vee }\otimes B\otimes C\overset{\varphi _{i,j}\otimes 1}{%
\rightarrow }A_{i+j}^\vee\otimes B^{\vee }\otimes C^{\vee }
\end{eqnarray*}%
and then%
\begin{eqnarray*}
\varphi _{g-i,i}^{13\rightarrow A_{g}^{\vee }} &:&A_{g-i}\otimes A_{g}^{\vee
}\otimes A_{i}\otimes A_{g}^{\vee }\overset{\varphi _{g-i,i}^{13}}{%
\rightarrow }A_{g}\otimes A_{g}^{\vee }\otimes A_{g}^{\vee }\overset{%
ev_{A_{g}}^{\tau }\otimes 1}{\rightarrow }A_{g}^{\vee }\text{,} \\
\varphi _{g-i,i}^{13\rightarrow A_{g}^{\vee \vee }} &:&A_{g-i}^{\vee
}\otimes A_{g}^{\vee \vee }\otimes A_{i}^{\vee }\otimes A_{g}^{\vee \vee }%
\overset{\varphi _{g-i,i}^{13}}{\rightarrow }A_{g}^{\vee }\otimes
A_{g}^{\vee \vee }\otimes A_{g}^{\vee \vee }\overset{ev_{A_{g}^{\vee
}}^{\tau }\otimes 1}{\rightarrow }A_{g}^{\vee\vee }\text{.}
\end{eqnarray*}

In the following discussion we let $r:=\mathrm{rank}\left( V\right) \in
End\left( \mathbb{I}\right) $.

\begin{theorem}
\label{Alternating algebras T}The following diagrams are commutative, for
every $g\geq i\geq 0$.

\begin{enumerate}
\item[$\left( 1\right) $] 
\begin{equation*}
\xymatrix@C=40pt{ \wedge^{i}V
\ar@/^{2pc}/[rrr]^-{\left(-1\right)^{i\left(g-i\right)}\binom{g}{g-i}^{-1}%
\binom{r-i}{g-i}} \ar[r]_-{D^{i,g}} &
\wedge^{g-i}V^{\vee}\otimes\wedge^{g}V^{\vee\vee}
\ar[r]_-{D_{g-i,g}\otimes1_{\wedge^{g}V^{\vee\vee}}} &
\wedge^{i}V\otimes\wedge^{g}V^{\vee}\otimes\wedge^{g}V^{\vee\vee}
\ar[r]_-{1_{\wedge^{i}V}\otimes ev_{V^{\vee},a}^{g,\tau}} & \wedge^{i}V }
\end{equation*}
and 
\begin{equation*}
\xymatrix@C=40pt{ \wedge^{g-i}V^{\vee}
\ar@/^{2pc}/[rrr]^-{\left(-1\right)^{i\left(g-i\right)}\binom{g}{i}^{-1}%
\binom{r+i-g}{i}} \ar[r]_-{D_{g-i,g}} & \wedge^{i}V\otimes\wedge^{g}V^{\vee}
\ar[r]_-{D^{i,g}\otimes1_{\wedge^{g}V^{\vee}}} &
\wedge^{g-i}V^{\vee}\otimes\wedge^{g}V^{\vee\vee}\otimes\wedge^{g}V^{\vee}
\ar[r]_-{1_{\wedge^{g-i}V^{\vee}}\otimes ev_{V^{\vee},a}^{g}} &
\wedge^{g-i}V^{\vee}\text{.} }
\end{equation*}

\item[$\left( 2\right) $] 
\begin{equation*}
\xymatrix{ \wedge^{i}V\otimes\wedge^{g-i}V \ar[r]^-{\varphi_{i,g-i}}
\ar[d]|{D^{i,g}\otimes D^{g-i,g}} & \wedge^{g}V
\ar[d]|{\binom{g}{g-i}^{-1}\binom{r-i}{g-i}\cdot i_{\wedge^{g}V}} &
\wedge^{i}V^{\vee}\otimes\wedge^{g-i}V^{\vee} \ar[r]^-{\varphi_{i,g-i}}
\ar[d]|{D_{i,g}\otimes D_{g-i,g}} & \wedge^{g}V^{\vee}
\ar[d]|{\binom{g}{g-i}^{-1}\binom{r-i}{g-i}} \\
\wedge^{g-i}V^{\vee}\otimes\wedge^{g}V^{\vee\vee}\otimes\wedge^{i}V^{\vee}%
\otimes\wedge^{g}V^{\vee\vee} \ar[r]^-{\varphi_{g-i,i}^{13\rightarrow
\wedge^{g}V^{\vee\vee}}} & \wedge^{g}V^{\vee\vee}\text{,} &
\wedge^{g-i}V\otimes\wedge^{g}V^{\vee}\otimes\wedge^{i}V\otimes\wedge^{g}V^{%
\vee} \ar[r]^-{\varphi_{g-i,i}^{13\rightarrow \wedge^{g}V^{\vee}}} &
\wedge^{g}V^{\vee}\text{.}}
\end{equation*}
\end{enumerate}
\end{theorem}

We say that $V$ has \emph{alternating rank }$g\in \mathbb{N}_{\geq 1}$, if $%
\wedge ^{g}V$ is an invertible object and $\binom{r-i}{g-i}$ and $\binom{%
r+i-g}{i}$ are invertible for every $0\leq i\leq g$. For example, when $%
End\left( \mathbb{I}\right) $ is a field or $r\in \mathbb{Q}$, the second
condition means that $r$ is not a root of the polynomials $\binom{T-i}{g-i}%
\in \mathbb{Q}\left[ T\right] $ and $\binom{T+i-g}{i}\in \mathbb{Q}\left[ T%
\right] $ for every $0\leq i\leq g$, i.e. that $r\neq i,i+1,...,g-i-1$ and $%
r\neq g-i,g-i+1,...,g-1$ for every $1\leq i\leq g$.

We say that $V$ has \emph{strong alternating rank }$g\in \mathbb{N}_{\geq 1}$%
, if $\wedge ^{g}V $ is an invertible object and $r=g$ (hence $V$ has
alternating rank $g$).

\begin{corollary}
\label{Alternating algebras CT}If $V$ has alternating rank $g\in 
\mathbb{N}$ then, for every $0\leq i\leq g$, the morphisms $D^{i,g}$, $%
D_{g-i,g}$, $D^{g-i,g}$ and $D_{i,g}$ are isomorphisms and the
multiplication maps $\varphi _{i,g-i}^{V}$, $\varphi _{g-i,i}^{V}$, $\varphi
_{i,g-i}^{V^{\vee }}$ and $\varphi _{g-i,i}^{V^{\vee }}$ are perfect
pairings (meaning that the associate $\hom $ valued morphisms are
isomorphisms). Furthermore, when $V$ has strong alternating rank $g$, we have $\binom{%
r-i}{g-i}=\binom{r+i-g}{i}=1$ in the commutative diagrams of Theorem \ref%
{Alternating algebras T}.
\end{corollary}

\begin{proposition}
\label{Alternating algebras P2}The following diagrams are commutative, when $%
\wedge ^{g}V$ is invertible of rank $r_{\wedge ^{g}V}$ (hence $r_{\wedge
^{g}V}\in \left\{ \pm 1\right\} $):%
{\tiny
\[
\xymatrix{
\wedge^{i}V\otimes\wedge^{g-i}V\otimes V
\ar[r]^-{\tau_{\wedge^{i}V\otimes\wedge^{g-i}V,V}}
\ar[d]|{\left(1_{\wedge^{i}V}\otimes\varphi_{g-i,1},\left(1_{\wedge^{g-i}V}%
\otimes\varphi_{i,1}\right)\circ\left(\tau_{\wedge^{i}V,\wedge^{g-i}V}%
\otimes1_{V}\right)\right)} & V\otimes\wedge^{i}V\otimes\wedge^{g-i}V
\ar[d]|{D^{1,g}\otimes\varphi_{i,g-i}} \\
\wedge^{i}V\otimes\wedge^{g-i+1}V\oplus\wedge^{g-i}V\otimes\wedge^{i+1}V
\ar[d]|{D^{i,g}\otimes D^{g-i+1,g}\oplus D^{g-i,g}\otimes D^{i+1,g}} &
\wedge^{g-1}V^{\vee}\otimes\wedge^{g}V^{\vee\vee}\otimes\wedge^{g}V
\ar[d]|{r_{\wedge^{g}V}g\binom{g}{g-i}^{-1}\binom{r-i}{g-i}\cdot1_{%
\wedge^{g-1}V^{\vee}\otimes\wedge^{g}V^{\vee\vee}}\otimes i_{\wedge^{g}V}}
\\
\wedge^{g-i}V^{\vee}\otimes\wedge^{g}V^{\vee\vee}\otimes\wedge^{i-1}V^{\vee}%
\otimes\wedge^{g}V^{\vee\vee}\oplus\wedge^{i}V^{\vee}\otimes\wedge^{g}V^{%
\vee\vee}\otimes\wedge^{g-i-1}V^{\vee}\otimes\wedge^{g}V^{\vee\vee}
\ar[r]_-{\left(-1\right)^{g-i}i\cdot\varphi_{g-i,i-1}^{13}\oplus\left(-1%
\right)^{i\left(g-i-1\right)}\left(g-i\right)\cdot\varphi_{i,g-i-1}^{13}} &
\wedge^{g-1}V^{\vee}\otimes\wedge^{g}V^{\vee\vee}\otimes\wedge^{g}V^{\vee%
\vee}
}
\]}
and%
{\tiny
\[
\xymatrix{
\wedge^{i}V^{\vee}\otimes\wedge^{g-i}V^{\vee}\otimes V^{\vee}
\ar[r]^-{\tau_{\wedge^{i}V^{\vee}\otimes\wedge^{g-i}V^{\vee},V^{\vee}}}
\ar[d]|{\left(1_{\wedge^{i}V^{\vee}}\otimes\varphi_{g-i,1},\left(1_{%
\wedge^{g-i}V^{\vee}}\otimes\varphi_{i,1}\right)\circ\left(\tau_{%
\wedge^{i}V^{\vee},\wedge^{g-i}V^{\vee}}\otimes1_{V^{\vee}}\right)\right)} &
V^{\vee}\otimes\wedge^{i}V^{\vee}\otimes\wedge^{g-i}V^{\vee}
\ar[d]|{D_{1,g}\otimes\varphi_{i,g-i}} \\
\wedge^{i}V^{\vee}\otimes\wedge^{g-i+1}V^{\vee}\oplus\wedge^{g-i}V^{\vee}%
\otimes\wedge^{i+1}V^{\vee} \ar[d]|{D_{i,g}\otimes D_{g-i+1,g}\oplus
D_{g-i,g}\otimes D_{i+1,g}} &
\wedge^{g-1}V\otimes\wedge^{g}V^{\vee}\otimes\wedge^{g}V^{\vee}
\ar[d]|{r_{\wedge^{g}V}g\binom{g}{g-i}^{-1}\binom{r-i}{g-i}\cdot1_{%
\wedge^{g-1}V^{\vee}\otimes\wedge^{g}V^{\vee}\otimes\wedge^{g}V^{\vee}}} \\
\wedge^{g-i}V\otimes\wedge^{g}V^{\vee}\otimes\wedge^{i-1}V\otimes%
\wedge^{g}V^{\vee}\oplus\wedge^{i}V\otimes\wedge^{g}V^{\vee}\otimes%
\wedge^{g-i-1}V\otimes\wedge^{g}V^{\vee}
\ar[r]_-{\left(-1\right)^{g-i}i\cdot\varphi_{g-i,i-1}^{13}\oplus\left(-1%
\right)^{i\left(g-i-1\right)}\left(g-i\right)\cdot\varphi_{i,g-i-1}^{13}} &
\wedge^{g-1}V\otimes\wedge^{g}V^{\vee}\otimes\wedge^{g}V^{\vee}\text{.}
}
\]
}
\end{proposition}

Here are the analogue of the above results for the symmetric algebras.

\begin{theorem}
\label{Symmetric algebras T}The following diagrams are commutative, for
every $g\geq i\geq 0$.

\begin{enumerate}
\item[$\left( 1\right) $] 
\begin{equation*}
\xymatrix@C=40pt{ \vee^{i}V
\ar@/^{2pc}/[rrr]^-{\binom{g}{g-i}^{-1}\binom{r+g-1}{g-i}} \ar[r]_-{D^{i,g}}
& \vee^{g-i}V^{\vee}\otimes\vee^{g}V^{\vee\vee}
\ar[r]_-{D_{g-i,g}\otimes1_{\vee^{g}V^{\vee\vee}}} &
\vee^{i}V\otimes\vee^{g}V^{\vee}\otimes\vee^{g}V^{\vee\vee}
\ar[r]_-{1_{\vee^{i}V}\otimes ev_{V^{\vee},s}^{g,\tau}} & \vee^{i}V }
\end{equation*}%
and 
\begin{equation*}
\xymatrix@C=40pt{ \vee^{g-i}V^{\vee}
\ar@/^{2pc}/[rrr]^-{\binom{g}{i}^{-1}\binom{r+g-1}{i}} \ar[r]_-{D_{g-i,g}} &
\vee^{i}V\otimes\vee^{g}V^{\vee}
\ar[r]_-{D^{i,g}\otimes1_{\vee^{g}V^{\vee}}} &
\vee^{g-i}V^{\vee}\otimes\vee^{g}V^{\vee\vee}\otimes\vee^{g}V^{\vee}
\ar[r]_-{1_{\vee^{g-i}V^{\vee}}\otimes ev_{V^{\vee},s}^{g}} &
\vee^{g-i}V^{\vee}\text{.} }
\end{equation*}

\item[$\left( 2\right) $] 
\begin{equation*}
\xymatrix{ \vee^{i}V\otimes\vee^{g-i}V \ar[r]^-{\varphi_{i,g-i}}
\ar[d]|{D^{i,g}\otimes D^{g-i,g}} & \vee^{g}V
\ar[d]|{\binom{g}{g-i}^{-1}\binom{r+g-1}{g-i}\cdot i_{\vee^{g}V}} &
\vee^{i}V^{\vee}\otimes\vee^{g-i}V^{\vee} \ar[r]^-{\varphi_{i,g-i}}
\ar[d]|{D_{i,g}\otimes D_{g-i,g}} & \vee^{g}V^{\vee}
\ar[d]|{\binom{g}{g-i}^{-1}\binom{r+g-1}{g-i}} \\
\vee^{g-i}V^{\vee}\otimes\vee^{g}V^{\vee\vee}\otimes\vee^{i}V^{\vee}\otimes%
\vee^{g}V^{\vee\vee}
\ar[r]^(0.8){\varphi_{g-i,i}^{13\rightarrow\vee^{g}V^{\vee\vee}}} &
\vee^{g}V^{\vee\vee}\text{,} &
\vee^{g-i}V\otimes\vee^{g}V^{\vee}\otimes\vee^{i}V\otimes\vee^{g}V^{\vee}
\ar[r]^-(0.7){\varphi_{g-i,i}^{13\rightarrow\vee^{g}V^{\vee}}} &
\vee^{g}V^{\vee}\text{.}}
\end{equation*}
\end{enumerate}
\end{theorem}

We say that $V$ has \emph{symmetric rank }$g\in \mathbb{N}_{\geq 1}$, if $%
\vee ^{g}V$ is an invertible object and $\binom{r+g-1}{g-i}$ and $\binom{%
r+g-1}{i}$ are invertible for every $0\leq i\leq g$. For example, when $%
End\left( \mathbb{I}\right) $ is a field or $r\in \mathbb{Q}$, the second
condition means that $r $ is not a root of the polynomials $\binom{T+g-1}{g-i%
}\in \mathbb{Q}\left[ T\right] $ and $\binom{T+g-1}{i}\in \mathbb{Q}\left[ T%
\right] $ for every $0\leq i\leq g$, i.e. that $r\neq 1-g,2-g,...,-i$ and $%
r\neq 1-g,2-g,...,i-g$ for every $1\leq i\leq g$.

We say that $V$ has \emph{strong symmetric rank }$g\in \mathbb{N}_{\geq 1}$,
if $\vee ^{g}V$ is an invertible object and $r=-g$ (hence $V$ has symmetric
rank $g$).

\begin{corollary}
\label{Symmetric algebras CT}If $V$ has symmetric rank $g\in \mathbb{N}$
then, for every $0\leq i\leq g$, the morphisms $D^{i,g}$, $D_{g-i,g}$, $%
D^{g-i,g}$ and $D_{i,g}$ are isomorphisms and the multiplication maps $%
\varphi _{i,g-i}^{V}$, $\varphi _{g-i,i}^{V}$, $\varphi _{i,g-i}^{V^{\vee }}$
and $\varphi _{g-i,i}^{V^{\vee }}$ are perfect pairings (meaning that the
associate $\hom $ valued morphisms are isomorphisms). Furthermore, when $V$
has strong symmetric rank $g$, we have $\binom{r+g-1}{g-i}=\left( -1\right)
^{g-i}$ and $\binom{r+g-1}{i}=\left( -1\right) ^{i}$ in the commutative
diagrams of Theorem \ref{Symmetric algebras T}.
\end{corollary}

\begin{proposition}
\label{Symmetric algebras P2}The following diagrams are commutative when $%
\vee ^{g}V$ is invertible of rank $r_{\vee ^{g}V}$ (hence $r_{\vee
^{g}V}\in \left\{ \pm 1\right\} $):%
{\tiny
\[
\xymatrix{ \vee^{i}V\otimes\vee^{g-i}V\otimes V \ar[r]^-{\tau_{\vee^{i}V\otimes\vee^{g-i}V,V}} \ar[d]|{\left(1_{\vee^{i}V}\otimes\varphi_{g-i,1},\left(1_{\vee^{g-i}V}\otimes\varphi_{i,1}\right)\circ\left(\tau_{\vee^{i}V,\vee^{g-i}V}\otimes1_{V}\right)\right)} & V\otimes\vee^{i}V\otimes\vee^{g-i}V \ar[d]|{D^{1,g}\otimes\varphi_{i,g-i}} \\ \vee^{i}V\otimes\vee^{g-i+1}V\oplus\vee^{g-i}V\otimes\vee^{i+1}V \ar[d]|{D^{i,g}\otimes D^{g-i+1,g}\oplus D^{g-i,g}\otimes D^{i+1,g}} & \vee^{g-1}V^{\vee}\otimes\vee^{g}V^{\vee\vee}\otimes\vee^{g}V \ar[d]|{r_{\vee^{g}V}g\binom{g}{g-i}^{-1}\binom{r+g-1}{g-i}\cdot1_{\vee^{g-1}V^{\vee}\otimes\vee^{g}V^{\vee\vee}}\otimes i_{\vee^{g}V}} \\ \vee^{g-i}V^{\vee}\otimes\vee^{g}V^{\vee\vee}\otimes\vee^{i-1}V^{\vee}\otimes\vee^{g}V^{\vee\vee}\oplus\vee^{i}V^{\vee}\otimes\vee^{g}V^{\vee\vee}\otimes\vee^{g-i-1}V^{\vee}\otimes\vee^{g}V^{\vee\vee} \ar[r]_-{i\cdot\varphi_{g-i,i-1}^{13}\oplus\left(g-i\right)\cdot\varphi_{i,g-i-1}^{13}} & \vee^{g-1}V^{\vee}\otimes\vee^{g}V^{\vee\vee}\otimes\vee^{g}V^{\vee\vee}}
\]}
and
{\tiny
\[
\xymatrix{ \vee^{i}V^{\vee}\otimes\vee^{g-i}V^{\vee}\otimes V^{\vee} \ar[r]^-{\tau_{\vee^{i}V^{\vee}\otimes\vee^{g-i}V^{\vee},V^{\vee}}} \ar[d]|{\left(1_{\vee^{i}V^{\vee}}\otimes\varphi_{g-i,1},\left(1_{\vee^{g-i}V^{\vee}}\otimes\varphi_{i,1}\right)\circ\left(\tau_{\vee^{i}V^{\vee},\vee^{g-i}V^{\vee}}\otimes1_{V^{\vee}}\right)\right)} & V^{\vee}\otimes\vee^{i}V^{\vee}\otimes\vee^{g-i}V^{\vee} \ar[d]|{D_{1,g}\otimes\varphi_{i,g-i}} \\ \vee^{i}V^{\vee}\otimes\vee^{g-i+1}V^{\vee}\oplus\vee^{g-i}V^{\vee}\otimes\vee^{i+1}V^{\vee} \ar[d]|{D_{i,g}\otimes D_{g-i+1,g}\oplus D_{g-i,g}\otimes D_{i+1,g}} & \vee^{g-1}V\otimes\vee^{g}V^{\vee}\otimes\vee^{g}V^{\vee} \ar[d]|{r_{\vee^{g}V}g\binom{g}{g-i}^{-1}\binom{r+g-1}{g-i}\cdot1_{\vee^{g-1}V^{\vee}\otimes\vee^{g}V^{\vee}\otimes\vee^{g}V^{\vee}}} \\ \vee^{g-i}V\otimes\vee^{g}V^{\vee}\otimes\vee^{i-1}V\otimes\vee^{g}V^{\vee}\oplus\vee^{i}V\otimes\vee^{g}V^{\vee}\otimes\vee^{g-i-1}V\otimes\vee^{g}V^{\vee} \ar[r]_-{i\cdot\varphi_{g-i,i-1}^{13}\oplus\left(g-i\right)\cdot\varphi_{i,g-i-1}^{13}} & \vee^{g-1}V\otimes\vee^{g}V^{\vee}\otimes\vee^{g}V^{\vee}\text{.} }
\]}
\end{proposition}

\section{\label{Dirac and Laplace operators}Dirac and Laplace operators}

If we are given $\psi :X\otimes Y\rightarrow Z$, we may
consider%
\begin{equation*}
\partial _{\psi }^{n}:=1_{\otimes ^{n-1}X}\otimes \psi :\otimes ^{n}X\otimes
Y\rightarrow \otimes ^{n-1}X\otimes Z\text{ for }n\geq 1
\end{equation*}%
and then we define%
\begin{eqnarray*}
\partial _{\psi ,a}^{n} &:&\wedge ^{n}X\otimes Y\overset{i_{X,a}^{n}\otimes
1_{Y}}{\rightarrow }\otimes ^{n}X\otimes Y\overset{\partial _{\psi }^{n}}{%
\rightarrow }\otimes ^{n-1}X\otimes Z\overset{p_{X,a}^{n-1}\otimes 1_{Z}}{%
\rightarrow }\wedge ^{n-1}X\otimes Z\text{,} \\
\partial _{\psi ,s}^{n} &:&\vee ^{n}X\otimes Y\overset{i_{X,s}^{n}\otimes
1_{Y}}{\rightarrow }\otimes ^{n}X\otimes Y\overset{\partial _{\psi }^{n}}{%
\rightarrow }\otimes ^{n-1}X\otimes Z\overset{p_{X,s}^{n-1}\otimes 1_{Z}}{%
\rightarrow }\vee ^{n-1}X\otimes Z\text{.}
\end{eqnarray*}%
Here we write $i_{X,\ast }^{n}$ and $p_{X,\ast }^{n}$ for the canonical
injective and, respectively, surjective morphisms arising from the
idempotent defining the alternating when $\ast =a$ and the symmetric when $%
\ast =s$ powers.

In particular, when $X=Y$, we have%
\begin{equation*}
\Delta _{\psi }^{n}=\partial _{\psi }^{n-1}=1_{\otimes ^{n-2}X}\otimes \psi
:\otimes ^{n}X\rightarrow \otimes ^{n-2}X\otimes Z\text{ for }n\geq 2
\end{equation*}%
inducing%
\begin{eqnarray*}
\Delta _{\psi ,a}^{n} &:&\wedge ^{n}X\overset{i_{X,a}^{n}}{\rightarrow }%
\otimes ^{n}X\overset{\Delta _{\psi }^{n}}{\rightarrow }\otimes
^{n-2}X\otimes Z\overset{p_{X,a}^{n-2}\otimes 1_{Z}}{\rightarrow }\wedge
^{n-2}X\otimes Z\text{,} \\
\Delta _{\psi ,s}^{n} &:&\vee ^{n}X\overset{i_{X,s}^{n}}{\rightarrow }%
\otimes ^{n}X\overset{\Delta _{\psi }^{n}}{\rightarrow }\otimes
^{n-2}X\otimes Z\overset{p_{X,s}^{n-2}\otimes 1_{Z}}{\rightarrow }\vee
^{n-2}X\otimes Z\text{.}
\end{eqnarray*}

\bigskip

We may lift these morphisms to the tensor products\ as follows. Let $%
\varepsilon $ (resp. $1$)\ be the sign character (resp. trivial)\ character
of the symmetric group and, if $\chi \in \left\{ \varepsilon ,1\right\} $
and $R\subset S_{k}$ is any subset, define%
\begin{equation*}
e_{R}^{\chi }:=\frac{1}{\#R}\tsum\nolimits_{\delta \in R}\chi ^{-1}\left(
\delta \right) \delta \in \mathbb{Q}\left[ S_{k}\right] \text{.}
\end{equation*}%
In particular, taking $R=S_{k}$ gives the idempotents $e_{X,a}^{k}:=e_{R}^{%
\varepsilon }$\ and $e_{X,s}^{k}:=e_{R}^{1}$\ defining the alternating and
symmetric $k$-powers of any object $X$. We have that $\partial _{\psi }^{n}$
(resp. $\Delta _{\psi }^{n}$)\ is equivariant for the action of $%
S_{n-1}=S_{\left\{ 1,...,n-1\right\} }\subset S_{n}$ (resp. $%
S_{n-2}=S_{\left\{ 1,...,n-2\right\} }\subset S_{n}$). Furthermore, if we
choose, for every $p\in \left\{ 1,...,n\right\} =:I_{n}$ (resp. $\left(
p,q\right) \in I_{n}\times I_{n}$ with $p\neq q$), elements $\delta
_{p}^{n}\in S_{n}$ (resp. $\delta _{p,q}^{n}\in S_{n}$) such that $\delta
_{p}^{n}\left( p\right) =n$ (resp. $\delta _{p,q}^{n-1,n}\left( p,q\right)
=\left( n-1,n\right) $), then $R_{S_{n-1}\backslash S_{n}}:=\left\{ \delta
_{p}^{n}:p\in I_{n}\right\} $ (resp. $R_{S_{n-2}\backslash S_{n}}:=\left\{
\delta _{p,q}^{n-1,n}:\left( p,q\right) \in I_{n}\times I_{n},p\neq
q\right\} $) is a set of coset representatives for $S_{n-1}\backslash S_{n}$
(resp. $S_{n-2}\backslash S_{n}$). Using these facts it is not difficult to
check that, setting%
\begin{eqnarray}
&&\widetilde{\partial }_{\psi ,\ast}^{n}:=\partial _{\psi
}^{n}\circ e_{R_{S_{n-1}\backslash S_{n}}}^{\chi }=\frac{1}{n}%
\tsum\nolimits_{p=1}^{n}\chi ^{-1}\left( \delta _{p}^{n}\right) \cdot \left(
1_{\otimes ^{n-1}X}\otimes \psi \right) \circ \left( \delta _{p}^{n}\otimes
1_{Y}\right) \text{,}  \label{Dirac F1} \\
&&\widetilde{\Delta }_{\psi ,\ast}^{n}:=\Delta _{\psi }^{n}\circ
e_{R_{S_{n-2}\backslash S_{n}}}^{\chi }=\frac{1}{n\left( n-1\right) }%
\tsum\nolimits_{p,q\in I_{n}:p\neq q}\chi ^{-1}\left( \delta
_{p,q}^{n-1,n}\right) \cdot \left( 1_{\otimes ^{n-2}X}\otimes \psi \right)
\circ \delta _{p,q}^{n-1,n}\text{,}  \label{Dirac F2}
\end{eqnarray}%
where $\ast\in\{a,s\}$ depending, respectively, on whether $\chi$ is $\varepsilon$ or $1$. This gives morphisms making the following diagrams commutative:%
\begin{equation}
\xymatrix{ \otimes^{n}X\otimes Y \ar[r]^-{\widetilde{\partial}_{\psi,a}^{n}} \ar[d]|{p_{X,a}^{n}\otimes1_{Y}} & \otimes^{n-1}X\otimes Z \ar[d]|{p_{X,a}^{n-1}\otimes1_{Z}} & \otimes^{n}X\otimes Y \ar[r]^-{\widetilde{\partial}_{\psi,a}^{n}} \ar[d]|{p_{X,s}^{n}\otimes1_{Y}} & \otimes^{n-1}X\otimes Z \ar[d]|{p_{X,s}^{n-1}\otimes1_{Z}} \\ \wedge^{n}X\otimes Y \ar[r]^-{\partial_{\psi,a}^{n}} & \wedge^{n-1}X\otimes Z\text{,} & \vee^{n}X\otimes Y \ar[r]^-{\partial_{\psi,s}^{n}} & \vee^{n-1}X\otimes Z\text{,} \\ \otimes^{n}X \ar[r]^-{\widetilde{\Delta}_{\psi,a}^{n}} \ar[d]|{p_{X,a}^{n}} & \otimes^{n-2}X\otimes Z \ar[d]|{p_{X,a}^{n-2}\otimes1_{Z}} & \otimes^{n}X \ar[r]^-{\widetilde{\Delta}_{\psi,a}^{n}} \ar[d]|{p_{X,s}^{n}} & \otimes^{n-2}X\otimes Z \ar[d]|{p_{X,s}^{n-2}\otimes1_{Z}} \\ \wedge^{n}X \ar[r]^-{\Delta_{\psi,a}^{n}} & \wedge^{n-2}X\otimes Z\text{,} & \vee^{n}X \ar[r]^-{\Delta_{\psi,s}^{n}} & \vee^{n-2}X\otimes Z\text{.}}  \label{Dirac D1}
\end{equation}

\bigskip

When $\psi $ is alternating or symmetric, we can refine $\left( \text{\ref%
{Dirac F2}}\right) $ as follows.

\begin{lemma}
\label{Dirac L1}Suppose that $\psi :X\otimes X\rightarrow Z$ is such that $%
\psi \circ \tau _{X,X}=-\psi $ (resp. $\psi \circ \tau _{X,X}=\psi $). Then $%
\Delta _{\psi ,s}^{n}=0$ (resp. $\Delta _{\psi ,a}^{n}=0$) and $\Delta
_{\psi ,a}^{n}$ (resp. $\Delta _{\psi ,s}^{n}$) is induced by%
\begin{equation*}
\widehat{\Delta }_{\psi ,\ast}^{n}:=\frac{2}{n\left( n-1\right) }%
\tsum\nolimits_{p,q\in I_{n}:p<q}\chi ^{-1}\left( \delta
_{p,q}^{n-1,n}\right) \cdot \left( 1_{\otimes ^{n-2}X}\otimes \psi \right)
\circ \delta _{p,q}^{n-1,n}
\end{equation*}%
where $\chi =\varepsilon $ (resp. $\chi =1$), $\ast=a$ (resp. $\ast
=s$) and $\delta _{p,q}^{n-1,n}\left( p,q\right) =\left(
n-1,n\right) $.
\end{lemma}

\begin{proof}
The proof, based on $\left( \text{\ref{Dirac F2}}\right) $ and the
subsequent Remark \ref{Dirac R1}, is left to the reader.
\end{proof}

\bigskip

\begin{remark}
\label{Dirac R1}Suppose that we are given actions of $S_{n}$ on $A$ and of $%
S_{n-2}$ on $B$ and that $f:A\rightarrow B$ is an $S_{n-2}$-equivariant map,
for an integer $n\geq 2$. Then we have, setting $\tau _{n-1,n}:=\left(
n-1,n\right) $,%
\begin{eqnarray*}
&&e_{S_{n-2}}^{\chi }\circ f\circ e_{R_{S_{n-2}\backslash S_{n}}}^{\chi
}:=e_{S_{n-2}}^{\chi }\circ \frac{1}{n\left( n-1\right) }\tsum\nolimits_{p,q%
\in I_{n}:p\neq q}\chi ^{-1}\left( \delta _{p,q}^{n-1,n}\right) \cdot f\circ
\delta _{p,q}^{n-1,n} \\
&&\text{ \ }=e_{S_{n-2}}^{\chi }\circ \frac{1}{n\left( n-1\right) }%
\tsum\nolimits_{p,q\in I_{n}:p<q}\chi ^{-1}\left( \delta
_{p,q}^{n-1,n}\right) \cdot \left( f+\chi ^{-1}\left( \tau _{n-1,n}\right)
\cdot f\circ \tau _{n-1,n}\right) \circ \delta _{p,q}^{n-1,n}\text{.}
\end{eqnarray*}
\end{remark}

\bigskip

Suppose now that we are given $\psi _{1}:X\otimes Y\rightarrow Z$ and $\psi
_{2}:X\otimes Z\rightarrow Y\otimes W$ and $\psi :X\otimes X\rightarrow W$.
They induce%
\begin{eqnarray*}
&&\wedge ^{n}X\otimes Y\overset{\partial _{\psi _{1},a}^{n}}{\rightarrow }%
\wedge ^{n-1}X\otimes Z\overset{\partial _{\psi _{2},a}^{n-1}}{\rightarrow }%
\wedge ^{n-2}X\otimes Y\otimes W\text{,} \\
&&\vee ^{n}X\otimes Y\overset{\partial _{\psi _{1},s}^{n}}{\rightarrow }\vee
^{n-1}X\otimes Z\overset{\partial _{\psi _{2},s}^{n-1}}{\rightarrow }\vee
^{n-2}X\otimes Y\otimes W
\end{eqnarray*}%
and%
\begin{eqnarray*}
\Delta _{\psi ,a}^{n} &:&\wedge ^{n}X\rightarrow \wedge ^{n-2}X\otimes W%
\text{,} \\
\Delta _{\psi ,s}^{n} &:&\vee ^{n}X\rightarrow \vee ^{n-2}X\otimes W\text{.}
\end{eqnarray*}

\begin{lemma}
\label{Dirac L2}Suppose that $\psi :X\otimes X\rightarrow W$ is such that $%
\psi \circ \tau _{X,X}=\nu _{\ast}\cdot \psi $, where $\nu _{a}:=-1$%
, $\nu _{s}:=1$ and $\ast\in \left\{ a,s\right\} $, and that, for
some $\rho \in End\left( \mathbb{I}\right) $, the following diagram is
commutative:%
\begin{equation*}
\xymatrix@C=110pt{ X\otimes X\otimes Y \ar[r]^-{\left(1_{X}\otimes\psi_{1},\left(1_{X}\otimes\psi_{1}\right)\circ\left(\tau_{X,X}\otimes1_{Y}\right)\right)} \ar[d]_{\psi\otimes1_{Y}} & X\otimes Z\oplus X\otimes Z \ar[d]^{\psi_{2}\oplus\nu_{\ast}\cdot\psi_{2}} \\ W\otimes Y \ar[r]^-{\rho\cdot\tau_{W,Y}} & Y\otimes W\text{.} }
\end{equation*}%
Then, when $\ast =a$, the following diagram is commutative%
\begin{equation*}
\xymatrix@C=70pt{ \wedge^{n}X\otimes Y \ar[r]^-{\partial_{\psi_{1},a}^{n}} \ar[d]_{\Delta_{\psi,a}^{n}\otimes1_{Y}} & \wedge^{n-1}X\otimes Z \ar[d]^{\partial_{\psi_{2},a}^{n-1}} \\ \wedge^{n-2}X\otimes W\otimes Y \ar[r]^-{\frac{\rho}{2}\cdot1_{\wedge^{n-2}X}\otimes\tau_{W,Y}} & \wedge^{n-2}X\otimes Y\otimes W }
\end{equation*}%
and, when $\ast =s$, the following diagram is commutative:%
\begin{equation*}
\xymatrix@C=70pt{ \vee^{n}X\otimes Y \ar[r]^-{\partial_{\psi_{1},s}^{n}} \ar[d]_{\Delta_{\psi,s}^{n}\otimes1_{Y}} & \vee^{n-1}X\otimes Z \ar[d]^{\partial_{\psi_{2},s}^{n-1}} \\ \vee^{n-2}X\otimes W\otimes Y \ar[r]^-{\frac{\rho}{2}\cdot1_{\vee^{n-2}X}\otimes\tau_{W,Y}} & \vee^{n-2}X\otimes Y\otimes W\text{.} }
\end{equation*}
\end{lemma}

\begin{proof}
We compute, using the notations in $\left( \text{\ref{Dirac F1}}\right) $,%
\begin{eqnarray}
\widetilde{\partial }_{\psi _{2},\ast }^{n-1}\circ \widetilde{%
\partial }_{\psi _{1},\ast }^{n} &=&\frac{1}{n\left( n-1\right) }%
\dsum\nolimits_{\substack{ q=1,...,n  \\ p=1,...,n-1}}\chi ^{-1}\left(
\delta _{p}^{n-1}\delta _{q}^{n}\right) \cdot \left( 1_{\otimes
^{n-2}X}\otimes \psi _{2}\right) \circ \left( \delta _{p}^{n-1}\otimes
1_{Z}\right)  \notag \\
&&\circ \left( 1_{\otimes ^{n-1}X}\otimes \psi _{1}\right) \circ \left(
\delta _{q}^{n}\otimes 1_{Y}\right)  \notag \\
&=&\frac{1}{n\left( n-1\right) }\dsum\nolimits_{\substack{ q=1,...,n  \\ %
p=1,...,n-1}}\chi ^{-1}\left( \delta _{p}^{n-1}\delta _{q}^{n}\right) \cdot
\left( 1_{\otimes ^{n-2}X}\otimes \psi _{2}\right) \circ \left( 1_{\otimes
^{n-1}X}\otimes \psi _{1}\right)  \notag \\
&&\circ \left( \delta _{p}^{n-1}\otimes 1_{X\otimes Y}\right) \circ \left(
\delta _{q}^{n}\otimes 1_{Y}\right) \text{.}  \label{Dirac L2 F1}
\end{eqnarray}%
Here $\delta _{p}^{n-1}\otimes 1_{X\otimes Y}$ acts on $\otimes ^{n}X\otimes
Y$ as $\delta _{p}^{n-1}\otimes 1_{Y}$, where now $\delta _{p}^{n-1}\in
S_{n-1}=S_{\left\{ 1,...,n-1\right\} }\subset S_{n}$ is viewed in $S_{n}$,
so that $\left( \delta _{p}^{n-1}\otimes 1_{X\otimes Y}\right) \circ \left(
\delta _{q}^{n}\otimes 1_{Y}\right) =\delta _{p}^{n-1}\delta _{q}^{n}\otimes
1_{Y}$. We now remark that we may choose $\delta _{q}^{n}$ so that $\delta
_{q}^{n}\left( p\right) =p$ if $p\in \left\{ 1,...,n-1\right\} -\left\{
q\right\} $ and then we find%
\begin{equation*}
\delta _{p}^{n-1}\delta _{q}^{n}\left( p,q\right) =\delta _{p}^{n-1}\left(
p,n\right) =\left( n-1,n\right) \text{, if }p\in \left\{ 1,...,n-1\right\}
-\left\{ q\right\} \text{.}
\end{equation*}%
On the other hand, we may further assume that $\delta _{q}^{n}\left(
n\right) =q$ (with $\delta _{q}^{n}=\left( q,n\right) $ both the imposed
conditions are indeed satisfied). Then we find 
\begin{equation*}
\delta _{q}^{n-1}\delta _{q}^{n}\left( n,q\right) =\delta _{q}^{n-1}\left(
q,n\right) =\left( n-1,n\right) \text{, if }q\in \left\{ 1,...,n-1\right\} 
\text{.}
\end{equation*}%
Summarizing, setting $\delta _{p,q}^{n-1,n}:=\delta _{p}^{n-1}\delta
_{q}^{n} $ if $p\in \left\{ 1,...,n-1\right\} -\left\{ q\right\} $ and $%
\delta _{n,q}^{n-1,n}:=\delta _{q}^{n-1}\delta _{q}^{n}$ if $q\in \left\{
1,...,n-1\right\} $, we see that%
\begin{equation*}
\left\{ \left( p,q\right) \in I_{n}\times I_{n},p\neq q\right\} =\left\{
\left( p,q\right) :p\in I_{n-1}-\left\{ q\right\} \right\} \sqcup \left\{
\left( n,q\right) :q\in I_{n-1}\right\}
\end{equation*}%
and then, since $\delta _{p,q}^{n-1,n}\left( p,q\right) =\left( n-1,n\right) 
$, we have%
\begin{equation*}
R_{S_{n-2}\backslash S_{n}}=\left\{ \delta _{p,q}^{n-1,n}:\left( p,q\right)
\in I_{n}\times I_{n},p\neq q\right\} \text{.}
\end{equation*}

Setting $f:=\left( 1_{\otimes ^{n-2}X}\otimes \psi _{2}\right) \circ \left(
1_{\otimes ^{n-1}X}\otimes \psi _{1}\right) $ it follows from $\left( \text{%
\ref{Dirac L2 F1}}\right) $ and the above discussion that we have%
\begin{equation}
\widetilde{\partial }_{\psi _{2},\ast }^{n-1}\circ \widetilde{%
\partial }_{\psi _{1},\ast }^{n}=\frac{1}{n\left( n-1\right) }%
\tsum\nolimits_{p,q\in I_{n}:p\neq q}\chi ^{-1}\left( \delta
_{p,q}^{n-1,n}\right) \cdot f\circ \left( \delta _{p,q}^{n-1,n}\otimes
1_{Y}\right) \text{.}  \label{Dirac L2 F2}
\end{equation}%
Noticing that $f$ is $S_{n-2}$-equivariant we may apply Remark \ref{Dirac R1}
to get%
\begin{eqnarray*}
&&e_{S_{n-2}}^{\chi }\circ \widetilde{\partial }_{\psi _{2},\ast }^{n-1}\circ \widetilde{\partial }_{\psi _{1},\ast }^{n} \\
&&\text{ }=e_{S_{n-2}}^{\chi }\circ \frac{1}{n\left( n-1\right) }%
\tsum\nolimits_{p,q\in I_{n}:p<q}\chi ^{-1}\left( \delta
_{p,q}^{n-1,n}\right) \cdot \left( f+\chi ^{-1}\left( \tau _{n-1,n}\right)
\cdot f\circ \tau _{n-1,n}\right) \circ \left( \delta _{p,q}^{n-1,n}\otimes
1_{Y}\right) \text{.}
\end{eqnarray*}

We now remark that the relation%
\begin{equation*}
\psi _{2}\circ \left( 1_{X}\otimes \psi _{1}\right) +\nu _{\ast }\cdot \psi _{2}\circ \left( 1_{X}\otimes \psi _{1}\right) \circ \left(
\tau _{X,X}\otimes 1_{Y}\right) =\rho \cdot \tau _{W,Y}\circ \left( \psi
\otimes 1_{Y}\right)
\end{equation*}%
gives, thanks to $\nu _{\ast }=\chi ^{-1}\left( \tau _{n-1,n}\right) 
$,%
\begin{eqnarray*}
f+\chi ^{-1}\left( \tau _{n-1,n}\right) \cdot f\circ \tau _{n-1,n} &=&\left(
1_{\otimes ^{n-2}X}\otimes \psi _{2}\right) \circ \left( 1_{\otimes
^{n-1}X}\otimes \psi _{1}\right) \\
&&+\nu _{\ast }\cdot \left( 1_{\otimes ^{n-2}X}\otimes \psi
_{2}\right) \circ \left( 1_{\otimes ^{n-1}X}\otimes \psi _{1}\right) \circ
\left( \tau _{n-1,n}\otimes 1_{Y}\right) \\
&=&\rho \cdot \left( 1_{\otimes ^{n-2}X}\otimes \tau _{W,Y}\right) \circ
\left( 1_{\otimes ^{n-2}X}\otimes \psi \otimes 1_{Y}\right) \text{.}
\end{eqnarray*}%
Hence $\left( \text{\ref{Dirac L2 F2}}\right) $ gives%
\begin{eqnarray}
&&e_{S_{n-2}}^{\chi }\circ \widetilde{\partial }_{\psi _{2},\ast }^{n-1}\circ \widetilde{\partial }_{\psi _{1},\ast }^{n}=e_{S_{n-2}}^{\chi }\circ \left( 1_{\otimes ^{n-2}X}\otimes \tau
_{W,Y}\right)  \notag \\
&&\text{ \ \ \ \ \ \ }\circ \frac{\rho }{n\left( n-1\right) }%
\tsum\nolimits_{p,q\in I_{n}:p<q}\chi ^{-1}\left( \delta
_{p,q}^{n-1,n}\right) \cdot \left( 1_{\otimes ^{n-2}X}\otimes \psi \otimes
1_{Y}\right) \circ \left( \delta _{p,q}^{n-1,n}\otimes 1_{Y}\right) \text{.}
\label{Dirac L2 F3}
\end{eqnarray}%
We have $e_{S_{n-2}}^{\chi }=e_{X,\ast }^{n-2}\otimes 1_{T}=\left(
p_{X,\ast }^{n-2}\otimes 1_{T}\right) \circ \left( i_{X,\ast }^{n-2}\otimes 1_{T}\right) $ and $i_{X,\ast }^{n-2}\otimes 1_{T}$
is a monomorphism $\left( \text{\ref{Dirac L2 F3}}\right) $, where $%
T=Y\otimes W$ on the left hand side while $T=W\otimes Y$ on the right hand
side of $\left( \text{\ref{Dirac L2 F3}}\right) $. Hence $\left( \text{\ref%
{Dirac L2 F3}}\right) $ gives, with the notations of Lemma \ref{Dirac L1},%
\begin{eqnarray*}
&&2\cdot \left( p_{X,\ast }^{n-2}\otimes 1_{Y\otimes W}\right) \circ 
\widetilde{\partial }_{\psi _{2},\ast }^{n-1}\circ \widetilde{%
\partial }_{\psi _{1},\ast }^{n}=\rho \cdot \left( p_{X,\ast }^{n-2}\otimes 1_{W\otimes Y}\right) \circ \left( 1_{\otimes
^{n-2}X}\otimes \tau _{W,Y}\right) \circ \left( \widehat{\Delta }_{\psi
,\ast }^{n}\otimes 1_{Y}\right) \\
&&\text{ }=\rho \cdot \left( 1_{n-2}\otimes \tau _{W,Y}\right) \circ \left(
p_{X,\ast }^{n-2}\otimes 1_{W\otimes Y}\right) \circ \left( \widehat{%
\Delta }_{\psi ,\ast }^{n}\otimes 1_{Y}\right) \text{,}
\end{eqnarray*}%
where $1_{n-2}=1_{\wedge ^{n-2}X}$ when $\ast =a$ or, respectively, $%
1_{n-2}=1_{\vee ^{n-2}X}$ when $\ast =s$. Now the claim follows from 
$\left( \text{\ref{Dirac D1}}\right) $, which gives%
\begin{equation*}
\left( p_{X,\ast }^{n-2}\otimes 1_{Y\otimes W}\right) \circ 
\widetilde{\partial }_{\psi _{2},\ast }^{n-1}\circ \widetilde{%
\partial }_{\psi _{1},\ast }^{n}=\partial _{\psi _{2},\ast }^{n-1}\circ \left( p_{X,\ast }^{n-1}\otimes 1_{Z}\right) \circ 
\widetilde{\partial }_{\psi _{1},\ast }^{n}=\partial _{\psi
_{2},\ast }^{n-1}\circ \partial _{\psi _{1},\ast }^{n}\circ
\left( p_{X,\ast }^{n}\otimes 1_{Y}\right) \text{,}
\end{equation*}%
and Lemma \ref{Dirac L1}, which gives%
\begin{equation*}
\left( p_{X,\ast }^{n-2}\otimes 1_{W\otimes Y}\right) \circ \left( 
\widehat{\Delta }_{\psi ,\ast }^{n}\otimes 1_{Y}\right) =\left(
\Delta _{\psi ,\ast }^{n}\otimes 1_{Y}\right) \circ \left( p_{X,\ast
}^{n}\otimes 1_{Y}\right) \text{,}
\end{equation*}%
because $p_{X,\ast }^{n}\otimes 1_{Y}$ is an epimorphism.
\end{proof}

\bigskip

Suppose now that we are given a perfect pairing $\psi :X\otimes
X\rightarrow \mathbb{I}$, meaning that the associated $\hom $ valued
morphism $f_{\psi }:X\rightarrow X^{\vee }$ is an isomorphism. Then $\left(
X,\psi \right) $ is a dual pair for $X$ and we have the Casimir element $%
C_{\psi }:\mathbb{I}\rightarrow X\otimes X$. It follows from well known
properties of the Casimir element that we have the following commutative
diagrams:%
\begin{eqnarray}
&&1_{X}:X\overset{C_{\psi }\otimes 1_{X}}{\rightarrow }X\otimes X\otimes X%
\overset{1_{X}\otimes \psi }{\rightarrow }X\text{,}  \label{Dirac FCasimir 1}
\\
&&1_{X}:X\overset{1_{X}\otimes C_{\psi }}{\rightarrow }X\otimes X\otimes X%
\overset{\psi \otimes 1_{X}}{\rightarrow }X\text{.}  \label{Dirac FCasimir 2}
\end{eqnarray}%
Suppose that we have $\psi \circ \tau _{X,X}=\chi \left( \tau _{X,X}\right)
\cdot \psi $, where $\chi \in \left\{ 1,\varepsilon \right\} $\footnote{%
We remark that, assuming $2$ is invertible in $Hom\left( X\otimes X,\mathbb{I%
}\right) $, we may always write $\psi $ as the direct sum of its alternating
and symmetric part, defined respectively by the formulas $\psi _{a}:=\frac{%
\psi -\psi \circ \tau _{X,X}}{2}$ and $\psi _{s}:=\frac{\psi +\psi \circ
\tau _{X,X}}{2}$. This means that $\psi =\psi _{a}\oplus \psi _{s}$, up to
the identification $Hom\left( X\otimes X,\mathbb{I}\right) =Hom\left( \wedge
^{2}X,\mathbb{I}\right) \oplus Hom\left( \vee ^{2}X,\mathbb{I}\right) $ and
the above assumption is always achieved by $\psi _{a}$ and $\psi _{s}$.}.
Recall that we write $r_X=\operatorname{rank}(X):=\psi \circ \tau _{X,X}\circ C_{\psi
}$. Then we have $r_{X}=\chi \left( \tau _{X,X}\right) \cdot \psi \circ C_{\psi }$, implying that
the following diagram is commutative:%
\begin{equation}
\chi \left( \tau _{X,X}\right) r_{X}:\mathbb{I}\overset{C_{\psi }}{%
\rightarrow }X\otimes X\overset{\psi }{\rightarrow }\mathbb{I}\text{.}
\label{Dirac FCasimir 3}
\end{equation}%
We may consider%
\begin{equation*}
C_{\psi }^{n}:=1_{\otimes ^{n}X}\otimes C_{\psi }:\otimes ^{n}X\rightarrow
\otimes ^{n+2}X\text{ for }n\geq 0
\end{equation*}%
and then we define%
\begin{eqnarray*}
C_{\psi ,a}^{n} &:&\wedge ^{n}X\overset{i_{X,a}^{n}}{\rightarrow }\otimes
^{n}X\overset{C_{\psi }^{n}}{\rightarrow }\otimes ^{n+2}X\overset{%
p_{X,a}^{n+2}}{\rightarrow }\wedge ^{n+2}X\text{,} \\
C_{\psi ,s}^{n} &:&\vee ^{n}X\overset{i_{X,s}^{n}}{\rightarrow }\otimes ^{n}X%
\overset{C_{\psi }^{n}}{\rightarrow }\otimes ^{n+2}X\overset{p_{X,s}^{n+2}}{%
\rightarrow }\vee ^{n+2}X\text{.}
\end{eqnarray*}%
Since $C_{\psi }^{n}$ is $S_{n}$-equivariant, the following diagrams are
commutative:%
\begin{equation}
\xymatrix{ \otimes^{n}X \ar[r]^-{C_{\psi}^{n}} \ar[d]|{p_{X,a}^{n}} & \otimes^{n+2}X \ar[d]|{p_{X,a}^{n+2}} & \otimes^{n}X \ar[r]^-{C_{\psi}^{n}} \ar[d]|{p_{X,s}^{n}} & \otimes^{n+2}X \ar[d]|{p_{X,s}^{n+2}} \\ \wedge^{n}X \ar[r]^-{C_{\psi,a}^{n}} & \wedge^{n+2}X\text{,} & \vee^{n}X \ar[r]^-{C_{\psi,s}^{n}} & \vee^{n+2}X\text{.}}  \label{Dirac D2}
\end{equation}

\begin{lemma}
\label{Dirac L3}Suppose that $\psi :X\otimes X\rightarrow \mathbb{I}$ is a
perfect pairing such that $\psi \circ \tau _{X,X}=\nu _{\ast }\cdot
\psi $, where $\nu _{a}:=-1$, $\nu _{s}:=1$ and $\ast \in \left\{
a,s\right\} $. Then we have the formulas $\Delta _{\psi ,\ast}^{2}\circ C_{\psi ,\ast}^{0}=\nu _{\ast }r_{X}$, $3\Delta
_{\psi ,\ast}^{3}\circ C_{\psi ,}^{1}=\left( 2+\nu
_{\ast}r_{X}\right) \cdot 1_{X}$\ and, for every $n\geq 2$,%
\begin{equation*}
\frac{\left( n+2\right) \left( n+1\right) }{2}\cdot \Delta _{\psi ,\ast}^{n+2}\circ C_{\psi ,\ast}^{n}-\frac{n\left( n-1\right) }{2%
}\cdot C_{\psi ,\ast }^{n-2}\circ \Delta _{\psi ,\ast }^{n}=\left( 2n+\nu _{\ast }r_{X}\right) \cdot 1_{(\ast) ^{n}X}\text{,}
\end{equation*}%
where $(\ast)^{n}X:=\wedge ^{n}X$ for $\ast =a$, $(\ast)^{n}X:=\vee ^{n}X$ for $\ast=s$ and $r_{X}:=\mathrm{rank}\left(
X\right) $.
\end{lemma}

\begin{proof}
We have, employing the notations in Lemma \ref{Dirac L1},%
\begin{eqnarray}
&&\frac{\left( n+2\right) \left( n+1\right) }{2}\cdot \widehat{\Delta }%
_{\psi ,\ast }^{n+2}\circ C_{\psi }^{n}  \notag \\
&&\text{ \ }=\tsum\nolimits_{p,q\in I_{n+2}:p<q}\chi ^{-1}\left( \delta
_{p,q}^{n+1,n+2}\right) \cdot \left( 1_{\otimes ^{n}X}\otimes \psi \right)
\circ \delta _{p,q}^{n+1,n+2}\circ \left( 1_{\otimes ^{n}X}\otimes C_{\psi
}\right) \text{,}  \label{Dirac L3 F1} \\
&&\frac{n\left( n-1\right) }{2}\cdot C_{\psi }^{n-2}\circ \widehat{\Delta }%
_{\psi ,\ast }^{n}  \notag \\
&&\text{ \ }=\tsum\nolimits_{p,q\in I_{n}:p<q}\chi ^{-1}\left( \delta
_{p,q}^{n-1,n}\right) \cdot \left( 1_{\otimes ^{n-2}X}\otimes C_{\psi
}\right) \circ \left( 1_{\otimes ^{n-2}X}\otimes \psi \right) \circ \delta
_{p,q}^{n-1,n}\text{.}  \label{Dirac L3 F2}
\end{eqnarray}%
We claim that we have $\widehat{\Delta }_{\psi ,\ast }^{2}\circ
C_{\psi }^{0}=\nu _{\ast }r_{X}$, $3\widehat{\Delta }_{\psi ,\ast
}^{3}\circ C_{\psi }^{1}=\left( 2+\nu _{\ast }r_{X}\right)
\cdot 1_{X}$ and, for every $n\geq 2$,%
\begin{equation*}
\frac{\left( n+2\right) \left( n+1\right) }{2}\cdot \widehat{\Delta }_{\psi
,\ast }^{n+2}\circ C_{\psi }^{n}-\frac{n\left( n-1\right) }{2}\cdot
C_{\psi }^{n-2}\circ \widehat{\Delta }_{\psi ,\ast }^{n}=2n\cdot
e_{R_{S_{n-1}\backslash S_{n}}}^{\chi }+\nu _{\ast }r_{X}\cdot
1_{\otimes ^{n}X}\text{.}
\end{equation*}%
It will follow from Lemma \ref{Dirac L1}$\ $and $\left( \text{\ref{Dirac D2}}%
\right) $ that we have%
\begin{equation*}
\Delta _{\psi ,\ast}^{2}\circ C_{\psi ,\ast }^{0}=\nu
_{\ast }r_{X},3\Delta _{\psi ,\ast }^{3}\circ C_{\psi ,\ast
}^{1}=\left( 2+\nu _{\ast }r_{X}\right) \cdot 1_{X}\ 
\end{equation*}%
and, for every $n\geq 2$,%
\begin{eqnarray*}
&&\left( \frac{\left( n+2\right) \left( n+1\right) }{2}\cdot \Delta _{\psi
,\ast }^{n+2}\circ C_{\psi ,\ast }^{n}-\frac{n\left(
n-1\right) }{2}\cdot C_{\psi ,\ast }^{n-2}\circ \Delta _{\psi ,\ast
}^{n}\right) \circ p_{X,\ast }^{n} \\
&&\text{ }=p_{X,\ast }^{n}\circ \left( \frac{\left( n+2\right)
\left( n+1\right) }{2}\cdot \widehat{\Delta }_{\psi ,\ast }^{n+2}\circ C_{\psi }^{n}-\frac{n\left( n-1\right) }{2}\cdot C_{\psi
}^{n-2}\circ \widehat{\Delta }_{\psi ,\ast }^{n}\right) \\
&&\text{ }=2n\cdot p_{X,\ast }^{n}\circ e_{R_{S_{n-1}\backslash
S_{n}}}^{\chi }+\nu _{\ast }r_{X}\cdot p_{X,\ast }^{n}=\left( 2n+\nu _{\ast}r_{X}\right) \cdot p_{X,\ast _{\chi
}}^{n}\text{.}
\end{eqnarray*}%
Here in the last equality we have employed the relation $p_{X,\ast _{\chi
}}^{n}\circ e_{R_{S_{n-1}\backslash S_{n}}}^{\chi }=p_{X,\ast}^{n}$%
, which is proved noticing that, since $i_{X,\ast}^{n}$ is a
monomorphism, it is equivalent to $e_{X,\ast}^{n}\circ
e_{R_{S_{n-1}\backslash S_{n}}}^{\chi }=e_{X,\ast}^{n}$; this last
relation follows from the relations $e_{G}^{\chi }=e_{G}^{\chi }e_{H}^{\chi
} $ and $e_{H}^{\chi }e_{R_{H\backslash G}}^{\chi }=e_{G}^{\chi }$, implying 
$e_{G}^{\chi }e_{R_{H\backslash G}}^{\chi }=e_{G}^{\chi }e_{H}^{\chi
}e_{R_{H\backslash G}}^{\chi }=e_{G}^{\chi }e_{G}^{\chi }=e_{G}^{\chi }$, which can be easily checked in the group algebras. Then the
claim will follow from the fact that $p_{X,\ast}^{n}$ is an
epimorphism.

When $n=0$ in $\left( \text{\ref{Dirac L3 F1}}\right) $, we have $C_{\psi
}^{0}=C_{\psi }$, $\widehat{\Delta }_{\psi ,\ast}^{2}=\psi $ and
the equality $\widehat{\Delta }_{\psi ,\ast}^{2}\circ C_{\psi
}^{0}=\nu _{\ast}r_{X}$ follows from $\left( \text{\ref{Dirac
FCasimir 3}}\right) $. When $n=1$ in $\left( \text{\ref{Dirac L3 F1}}\right) 
$, we have%
\begin{eqnarray}
3\widehat{\Delta }_{\psi ,\ast}^{3}\circ C_{\psi }^{1} &=&\chi
^{-1}\left( \tau _{\left( 123\right) }\right) \cdot \left( 1_{X}\otimes \psi
\right) \circ \tau _{\left( 123\right) }\circ \left( 1_{X}\otimes C_{\psi
}\right)  \notag \\
&&+\chi ^{-1}\left( \tau _{\left( 12\right) }\right) \cdot \left(
1_{X}\otimes \psi \right) \circ \tau _{\left( 12\right) }\circ \left(
1_{X}\otimes C_{\psi }\right)  \notag \\
&&+\left( 1_{X}\otimes \psi \right) \circ \left( 1_{X}\otimes C_{\psi
}\right) \text{,}  \label{Dirac L3 F3}
\end{eqnarray}%
because we may take $\delta _{1,2}^{2,3}=\tau _{\left( 123\right) }$, $%
\delta _{1,3}^{2,3}=\tau _{\left( 12\right) }$ and $\delta _{2,3}^{2,3}=1$,
where $\tau _{\sigma }$ denotes the morphism attached to the permutation $%
\sigma $. We have $\tau _{\left( 123\right) }=\tau _{X\otimes X,X}$ and $%
\left( \psi \otimes 1_{X}\right) \circ \left( 1_{X}\otimes C_{\psi }\right)
=1_{X}$ by $\left( \text{\ref{Dirac FCasimir 2}}\right) $. Hence we deduce
the equality%
\begin{eqnarray}
&&\chi ^{-1}\left( \tau _{\left( 123\right) }\right) \cdot \left(
1_{X}\otimes \psi \right) \circ \tau _{\left( 123\right) }\circ \left(
1_{X}\otimes C_{\psi }\right) =\left( 1_{X}\otimes \psi \right) \circ \tau
_{X\otimes X,X}\circ \left( 1_{X}\otimes C_{\psi }\right)  \notag \\
&=&\left( \psi \otimes 1_{X}\right) \circ \left( 1_{X}\otimes C_{\psi
}\right) =1_{X}\text{.}  \label{Dirac L3 F4}
\end{eqnarray}%
Consider the following diagram:%
\begin{equation*}
\xymatrix{ \ar@{}[dr]|(0.55){(\tau)} & X\otimes X\otimes X \ar[d]^{\tau_{X,X\otimes X}} & \\ X \ar[r]^-{C_{\psi}\otimes1_{X}} \ar@/^{0.75pc}/[ur]^-{1_{X}\otimes C_{\psi}} & X\otimes X\otimes X \ar[r]^-{\nu_{\ast}\cdot1_{X}\otimes\psi} \ar[d]_{1_{X}\otimes\tau_{X,X}} & X \\ & X\otimes X\otimes X \ar@/_{0.75pc}/[ur]_-{1_{X}\otimes\psi} & \ar@{}[ul]|(0.55){(\tau)} }
\end{equation*}%
The region $\left( A\right) $ is commutative thanks to our assumption $\psi
\circ \tau _{X,X}=\nu _{\ast}\cdot \psi $. Noticing that $\tau
_{\left( 12\right) }=\left( 1_{X}\otimes \tau _{X,X}\right) \circ \tau
_{X,X\otimes X}$ and that $\left( 1_{X}\otimes \psi \right) \circ \left(
C_{\psi }\otimes 1_{X}\right) =1_{X}$ by $\left( \text{\ref{Dirac FCasimir 1}%
}\right) $, we deduce the equality:%
\begin{eqnarray}
&&\chi ^{-1}\left( \tau _{\left( 12\right) }\right) \cdot \left(
1_{X}\otimes \psi \right) \circ \tau _{\left( 12\right) }\circ \left(
1_{X}\otimes C_{\psi }\right) =\chi ^{-1}\left( \tau _{\left( 12\right)
}\right) \cdot \left( 1_{X}\otimes \psi \right) \circ \left( 1\otimes \tau
_{X,X}\right) \circ \tau _{X,X\otimes X}\circ \left( 1_{X}\otimes C_{\psi
}\right)  \notag \\
&&\text{ }=\chi ^{-1}\left( \tau _{\left( 12\right) }\right) \nu _{\ast
}\cdot \left( 1_{X}\otimes \psi \right) \circ \left( C_{\psi
}\otimes 1_{X}\right) =1_{X}\text{.}  \label{Dirac L3 F5}
\end{eqnarray}%
Inserting $\left( \text{\ref{Dirac L3 F4}}\right) $, $\left( \text{\ref%
{Dirac L3 F5}}\right) $ and $\left( \text{\ref{Dirac FCasimir 3}}\right) $
in $\left( \text{\ref{Dirac L3 F3}}\right) $ gives $3\widehat{\Delta }_{\psi
,\ast}^{3}\circ C_{\psi }^{1}=\left( 2+\nu _{\ast }r_{X}\right) \cdot 1_{X}$.

Suppose now that $n\geq 2$. We remark that we have%
\begin{eqnarray}
\left\{ \left( p,q\right) \in I_{n+2}\times I_{n+2}:p<q\right\} &=&\left\{
\left( p,q\right) \in I_{n}\times I_{n}:p<q\right\}  \notag \\
&&\sqcup \left( I_{n}\times \left\{ n+1\right\} \right) \sqcup \left(
I_{n}\times \left\{ n+2\right\} \right) \sqcup \left\{ \left( n+1,n+2\right)
\right\} \text{.}  \label{Dirac L3 F6}
\end{eqnarray}%
We may assume that we are given our choice of $\delta _{p,q}^{n-1,n}\in
S_{n}$ and we choose the elements $\delta _{p,q}^{n+1,n+2}\in S_{n+2}$ as
follows. First of all we view $S_{n}\subset S_{n+2}$ in the natural way, via 
$I_{n}\subset I_{n+2}$, and we choose elements $\delta _{p}^{n}\in S_{n}$
such that $\delta _{p}^{n}\left( p\right) =n$. If $\left( p,q\right) \in
I_{n}\times I_{n}$ and $p<q$, we set $\delta _{p,q}^{n+1,n+2}:=\tau _{\left(
n-1,n+1\right) \left( n,n+2\right) }\circ \delta _{p,q}^{n-1,n}$ and then we
define $\delta _{p,n+1}^{n+1,n+2}:=\tau _{\left( n,n+1,n+2\right) }\circ
\delta _{p}^{n}$, $\delta _{p,n+2}^{n+1,n+2}:=\tau _{\left( n,n+1\right)
}\circ \delta _{p}^{n}$ and $\delta _{n+1,n+2}^{n+1,n+2}=1$, noticing that,
in every case, we have the required relation $\delta _{p,q}^{n+1,n+2}\left(
p,q\right) =\left( n+1,n+2\right) $ satisfied. Thanks to $\left( \text{\ref%
{Dirac L3 F6}}\right) $, we may rewrite $\left( \text{\ref{Dirac L3 F1}}%
\right) $ as follows:%
\begin{eqnarray}
&&\text{ }\frac{\left( n+2\right) \left( n+1\right) }{2}\cdot \widehat{%
\Delta }_{\psi ,\ast}^{n+2}\circ C_{\psi
}^{n}=\tsum\nolimits_{p,q\in I_{n}:p<q}\chi ^{-1}\left( \delta
_{p,q}^{n+1,n+2}\right) \cdot \left( 1_{\otimes ^{n}X}\otimes \psi \right)
\circ \delta _{p,q}^{n+1,n+2}\circ \left( 1_{\otimes ^{n}X}\otimes C_{\psi
}\right)  \notag \\
&&\text{ \ \ \ \ \ \ }+\tsum\nolimits_{p\in I_{n}}\chi ^{-1}\left( \delta
_{p,n+1}^{n+1,n+2}\right) \cdot \left( 1_{\otimes ^{n}X}\otimes \psi \right)
\circ \delta _{p,n+1}^{n+1,n+2}\circ \left( 1_{\otimes ^{n}X}\otimes C_{\psi
}\right)  \notag \\
&&\text{ \ \ \ \ \ \ }+\tsum\nolimits_{p\in I_{n}}\chi ^{-1}\left( \delta
_{p,n+2}^{n+1,n+2}\right) \cdot \left( 1_{\otimes ^{n}X}\otimes \psi \right)
\circ \delta _{p,n+2}^{n+1,n+2}\circ \left( 1_{\otimes ^{n}X}\otimes C_{\psi
}\right)  \notag \\
&&\text{ \ \ \ \ \ \ }+\chi ^{-1}\left( \delta _{n+1,n+2}^{n+1,n+2}\right)
\cdot \left( 1_{\otimes ^{n}X}\otimes \psi \right) \circ \delta
_{n+1,n+2}^{n+1,n+2}\circ \left( 1_{\otimes ^{n}X}\otimes C_{\psi }\right) 
\notag \\
&&\text{ \ \ \ }=\tsum\nolimits_{p,q\in I_{n}:p<q}\chi ^{-1}\left( \delta
_{p,q}^{n-1,n}\right) \cdot \left( 1_{\otimes ^{n}X}\otimes \psi \right)
\circ \tau _{\left( n-1,n+1\right) \left( n,n+2\right) }\circ \left(
1_{\otimes ^{n}X}\otimes C_{\psi }\right) \circ \delta _{p,q}^{n-1,n}
\label{Dirac L3 F7 a} \\
&&\text{ \ \ \ \ \ \ }+\tsum\nolimits_{p\in I_{n}}\chi ^{-1}\left( \delta
_{p}^{n}\right) \chi ^{-1}\left( \tau _{\left( n,n+1,n+2\right) }\right)
\cdot \left( 1_{\otimes ^{n}X}\otimes \psi \right) \circ \tau _{\left(
n,n+1,n+2\right) }\circ \left( 1_{\otimes ^{n}X}\otimes C_{\psi }\right)
\circ \delta _{p}^{n}  \label{Dirac L3 F7 b} \\
&&\text{ \ \ \ \ \ \ }+\tsum\nolimits_{p\in I_{n}}\chi ^{-1}\left( \delta
_{p}^{n}\right) \chi ^{-1}\left( \tau _{\left( n,n+1\right) }\right) \cdot
\left( 1_{\otimes ^{n}X}\otimes \psi \right) \circ \tau _{\left(
n,n+1\right) }\circ \left( 1_{\otimes ^{n}X}\otimes C_{\psi }\right) \circ
\delta _{p}^{n}  \label{Dirac L3 F7 c} \\
&&\text{ \ \ \ \ \ \ }+\left( 1_{\otimes ^{n}X}\otimes \psi \right) \circ
\left( 1_{\otimes ^{n}X}\otimes C_{\psi }\right) \text{.}
\label{Dirac L3 F7 d}
\end{eqnarray}%
Making the substitution $\left( n-1,n,n+1,n+2\right) =\left( 1,2,3,4\right) $%
, we may write $\tau _{\left( n-1,n+1\right) \left( n,n+2\right)
}=1_{\otimes ^{n-2}X}\otimes \tau _{\left( 13\right) \left( 24\right) }$,
where $\tau _{\left( 13\right) \left( 24\right) }=\tau _{X\otimes X,X\otimes
X}$ is acting on the last four factors $X\otimes X\otimes X\otimes X$ of $%
\otimes ^{n+2}X$. Then the relation%
\begin{equation*}
\left( 1_{\otimes ^{2}X}\otimes \psi \right) \circ \tau _{X\otimes
X,X\otimes X}\circ \left( 1_{\otimes ^{2}X}\otimes C_{\psi }\right) =\left(
1_{\otimes ^{2}X}\otimes \psi \right) \circ \left( C_{\psi }\otimes
1_{\otimes ^{2}X}\right) =\left( C_{\psi }\otimes \psi \right) =C_{\psi
}\circ \psi
\end{equation*}%
implies that we have%
\begin{equation*}
\left( 1_{\otimes ^{n}X}\otimes \psi \right) \circ \tau _{\left(
n-1,n+1\right) \left( n,n+2\right) }\circ \left( 1_{\otimes ^{n}X}\otimes
C_{\psi }\right) =\left( 1_{\otimes ^{n-2}X}\otimes C_{\psi }\right) \circ
\left( 1_{\otimes ^{n-2}X}\otimes \psi \right) \text{.}
\end{equation*}%
Hence it follows from $\left( \text{\ref{Dirac L3 F2}}\right) $ that we have:%
\begin{equation}
\left( \text{\ref{Dirac L3 F7 a}}\right) =\frac{n\left( n-1\right) }{2}\cdot
C_{\psi }^{n-2}\circ \widehat{\Delta }_{\psi ,\ast}^{n}\text{.}
\label{Dirac L3 F8}
\end{equation}%
We are now going to compute the sums $\left( \text{\ref{Dirac L3 F7 b}}%
\right) $ and $\left( \text{\ref{Dirac L3 F7 c}}\right) $. Making the
substitution $\left( n,n+1,n+2\right) =\left( 1,2,3\right) $, we may write $%
\tau _{\left( n,n+1,n+2\right) }=1_{\otimes ^{n-1}X}\otimes \tau _{\left(
123\right) }$ (resp. $\tau _{\left( n,n+1\right) }=1_{\otimes
^{n-1}X}\otimes \tau _{\left( 12\right) }$), where $\tau _{\left( 123\right)
}$ (resp. $\tau _{\left( 12\right) }$)\ is acting on the last three factors $%
X\otimes X\otimes X$ of $\otimes ^{n+2}X$. It follows from $\left( \text{\ref%
{Dirac L3 F4}}\right) $ and $\left( \text{\ref{Dirac L3 F5}}\right) $)\ that
we have, respectively,%
\begin{eqnarray*}
&&\chi ^{-1}\left( \tau _{\left( n,n+1,n+2\right) }\right) \cdot \left(
1_{\otimes ^{n}X}\otimes \psi \right) \circ \tau _{\left( n,n+1,n+2\right)
}\circ \left( 1_{\otimes ^{n}X}\otimes C_{\psi }\right) =1_{\otimes ^{n}X}%
\text{,} \\
&&\chi ^{-1}\left( \tau _{\left( n,n+1\right) }\right) \cdot \left(
1_{\otimes ^{n}X}\otimes \psi \right) \circ \tau _{\left( n,n+1\right)
}\circ \left( 1_{\otimes ^{n}X}\otimes C_{\psi }\right) =1_{\otimes ^{n}X}%
\text{.}
\end{eqnarray*}%
Hence we find%
\begin{equation}
\left( \text{\ref{Dirac L3 F7 b}}\right) =\left( \text{\ref{Dirac L3 F7 c}}%
\right) =\tsum\nolimits_{p\in I_{n}}\chi ^{-1}\left( \delta _{p}^{n}\right)
\delta _{p}^{n}=n\cdot e_{R_{S_{n-1}\backslash S_{n}}}^{\chi }\text{.}
\label{Dirac L3 F9}
\end{equation}%
Finally, it follows from $\left( \text{\ref{Dirac FCasimir 3}}\right) $ that
we have%
\begin{equation}
\left( \text{\ref{Dirac L3 F7 d}}\right) =\nu _{\ast}r_{X}\cdot
1_{\otimes ^{n}X}\text{.}  \label{Dirac L3 F10}
\end{equation}

It now follows from $\left( \text{\ref{Dirac L3 F8}}\right) $, $\left( \text{%
\ref{Dirac L3 F9}}\right) $ and $\left( \text{\ref{Dirac L3 F10}}\right) $
that we have, as claimed,%
\begin{eqnarray*}
&&\frac{\left( n+2\right) \left( n+1\right) }{2}\cdot \widehat{\Delta }%
_{\psi ,\ast}^{n+2}\circ C_{\psi }^{n}=\left( \text{\ref{Dirac L3
F7 a}}\right) +\left( \text{\ref{Dirac L3 F7 b}}\right) +\left( \text{\ref%
{Dirac L3 F7 c}}\right) +\left( \text{\ref{Dirac L3 F7 d}}\right) \\
&&\text{ \ \ \ \ }=\frac{n\left( n-1\right) }{2}\cdot C_{\psi }^{n-2}\circ 
\widehat{\Delta }_{\psi ,\ast}^{n}+2n\cdot e_{R_{S_{n-1}\backslash
S_{n}}}^{\chi }+\nu _{\ast}r_{X}\cdot 1_{\otimes ^{n}X}\text{.}
\end{eqnarray*}
\end{proof}

\bigskip

The following definition will be useful in the following subsections.

\begin{definition}
We say the a morphism $f:M\rightarrow M$ is diagonalizable if there is an
isomorphism $M\simeq \tbigoplus\nolimits_{\lambda \in End\left( \mathbb{I}%
\right) }M_{f,\lambda }$ such that $M_{f,\lambda }=0$ for almost every $%
\lambda $ and $f\simeq \tbigoplus\nolimits_{\lambda \in End\left( \mathbb{I}%
\right) }f_{\lambda }$ via this isomorphism, with $f_{\lambda }=\lambda
:M_{f,\lambda }\rightarrow M_{f,\lambda }$ the multiplication by $\lambda
\in End\left( \mathbb{I}\right) $. In this case, we call the set%
\begin{equation*}
\sigma \left( f\right) :=\left\{ \lambda :M_{f,\lambda }\neq 0\right\}
\subset End\left( \mathbb{I}\right)
\end{equation*}%
the spectrum of $f$.
\end{definition}

\bigskip

It will be also convenient to introduce the following definition.

\begin{definition}
\label{Dirac definition positive rank}If $S\subset End\left( \mathbb{I}%
\right) $ we say that $S$ is strictly positive (resp. positive, strictly
negative or negative)\ and we write $S>0$ (resp. $S\geq 0$, $S<0$ or $%
S\leq 0$)\ to mean that there exists an ordered field $\left( K,\geq \right) 
$ such that $S\subset K\subset End\left( \mathbb{I}\right) $ and $s>0$
(resp. $s\geq 0$, $s<0$ or $s\leq 0$)\ in $K$ for every $s\in S$. If $s\in
End\left( \mathbb{I}\right) $, we write $s>0$ (resp. $S\geq 0$, $S<0$ or $%
S\leq 0$) to mean that $S>0$ (resp. $S\geq 0$, $S<0$ or $S\leq 0$)\ with $%
S=\left\{ s\right\} $.
\end{definition}

\subsection{Laplace operators attached to $\mathbb{I}$-valued perfect
alternating pairings}

We suppose in this subsection that we are given $\psi :X\otimes
X\rightarrow \mathbb{I}$ which is perfect, i.e. such that the associated $%
\hom $ valued morphism is an isomorphism, and alternating, i.e. $\psi \circ
\tau _{X,X}=-\psi $. It follows from Lemma \ref{Dirac L1} that we have $%
\Delta _{\psi ,s}^{n}=0$ and, hence, we concentrate on $\Delta _{\psi
,a}^{n} $. We set $r_{X}:=\mathrm{rank}\left( X\right) $ in the subsequent
discussion.

\begin{proposition}
\label{Dirac P1}When $r_{X}<0$ we have that $\Delta _{\psi ,a}^{n+2}\circ
C_{\psi ,a}^{n}$ when $n\geq 0$ (resp. $C_{\psi ,a}^{n-2}\circ \Delta _{\psi
,a}^{n}$ when $n\geq 2$) is diagonalizable, with spectrum%
\begin{equation*}
\sigma \left( \Delta _{\psi ,a}^{n+2}\circ C_{\psi ,a}^{n}\right) >0\text{
(resp. }\sigma \left( C_{\psi ,a}^{n-2}\circ \Delta _{\psi ,a}^{n}\right)
\geq 0\text{).}
\end{equation*}
\end{proposition}

\begin{proof}
It will be convenient to set $\delta _{\psi ,a}^{n}:=\frac{n\left(
n-1\right) }{2}\cdot \Delta _{\psi ,a}^{n}$, so that Lemma \ref{Dirac L3}
gives $\delta _{\psi ,a}^{2}\circ C_{\psi ,a}^{0}=-r_{X}$, $\delta _{\psi
,a}^{3}\circ C_{\psi ,a}^{1}=\left( 2-r_{X}\right) \cdot 1_{X}$\ and, for
every $n\geq 2$,%
\begin{equation}
\delta _{\psi ,a}^{n+2}\circ C_{\psi ,a}^{n}-C_{\psi ,a}^{n-2}\circ \delta
_{\psi ,a}^{n}=\left( 2n-r_{X}\right) \cdot 1_{\wedge ^{n}X}\text{.}
\label{Dirac P1 F1}
\end{equation}%
In particular, we see that $\delta _{\psi ,a}^{n+2}\circ C_{\psi ,a}^{n}$ is
diagonalizable for $n=0,1$ with $\sigma \left( \delta _{\psi ,a}^{2}\circ
C_{\psi ,a}^{0}\right) =\left\{ -r_{X}\right\} >0$ and $\sigma \left( \delta
_{\psi ,a}^{3}\circ C_{\psi ,a}^{1}\right) =\left\{ 2-r_{X}\right\} >0$. We
can now assume that $n\geq 2$ and that, by induction, $\delta _{\psi
,a}^{n}\circ C_{\psi ,a}^{n-2}$ is diagonalizable with spectrum $\sigma
\left( \delta _{\psi ,a}^{n}\circ C_{\psi ,a}^{n-2}\right) >0$ and we claim
that this implies both that $\Delta _{\psi ,a}^{n+2}\circ C_{\psi ,a}^{n}$
is diagonalizable with spectrum $>0$ and that $C_{\psi ,a}^{n-2}\circ \Delta
_{\psi ,a}^{n}$ is diagonalizable with spectrum $\geq 0$. Here and in the
following, the ordered field $\left( K,\geq \right) $ in the definition of
being positive is always taken to be the one appearing in the definition of $%
-r_{X}>0$.

Since $\delta _{\psi ,a}^{n}\circ C_{\psi ,a}^{n-2}$ is diagonalizable with
spectrum $>0$, we have that $\delta _{\psi ,a}^{n}\circ C_{\psi ,a}^{n-2}$
is an isomorphism. It now follows from an abstract non-sense that there is a
biproduct decomposition%
\begin{equation*}
\wedge ^{n}X\simeq \ker \left( C_{\psi ,a}^{n-2}\circ \delta _{\psi
,a}^{n}\right) \oplus \wedge ^{n-2}X
\end{equation*}%
such that%
\begin{equation}
C_{\psi ,a}^{n-2}\circ \delta _{\psi ,a}^{n}\simeq 0\oplus \left( \delta
_{\psi ,a}^{n}\circ C_{\psi ,a}^{n-2}\right) \text{.}  \label{Dirac P1 F2}
\end{equation}%
Since $\delta _{\psi ,a}^{n}\circ C_{\psi ,a}^{n-2}$ is diagonalizable with
spectrum $\sigma \left( \delta _{\psi ,a}^{n}\circ C_{\psi ,a}^{n-2}\right)
>0$, it follows from $\left( \text{\ref{Dirac P1 F2}}\right) $ that $\sigma
\left( C_{\psi ,a}^{n-2}\circ \delta _{\psi ,a}^{n}\right) $ is
diagonalizable with spectrum%
\begin{equation*}
\sigma \left( C_{\psi ,a}^{n-2}\circ \delta _{\psi ,a}^{n}\right) \subset
\left\{ 0\right\} \cup \sigma \left( \delta _{\psi ,a}^{n}\circ C_{\psi
,a}^{n-2}\right) \geq 0\text{.}
\end{equation*}%
It now follows from $\left( \text{\ref{Dirac P1 F1}}\right) $ that $\delta
_{\psi ,a}^{n+2}\circ C_{\psi ,a}^{n}$ is diagonalizable with spectrum%
\begin{equation*}
\sigma \left( \Delta _{\psi ,a}^{n+2}\circ C_{\psi ,a}^{n}\right) \subset
\left\{ \lambda +\left( 2n-r_{X}\right) :\lambda \in \sigma \left( C_{\psi
,a}^{n-2}\circ \delta _{\psi ,a}^{n}\right) \right\} >0\text{.}
\end{equation*}
\end{proof}

\bigskip

\begin{corollary}
\label{Dirac C1}When $r_{X}<0$ we have that, for every $n\geq 2$, the
Laplace operator $\Delta _{\psi ,a}^{n}$ has a section $s_{\psi
,a}^{n-2}:\wedge ^{n-2}X\rightarrow \wedge ^{n}X$ such that $\Delta _{\psi
,a}^{n}\circ s_{\psi ,a}^{n-2}=1_{\wedge ^{n-2}X}$ and, in particular, $\ker
\left( \Delta _{\psi ,a}^{n}\right) $ exists.
\end{corollary}

\begin{proof}
Indeed $\Delta _{\psi ,a}^{n}\circ C_{\psi ,a}^{n-2}$ is diagonalizable with
spectrum $\sigma \left( \Delta _{\psi ,a}^{n}\circ C_{\psi ,a}^{n-2}\right)
>0$ by Proposition \ref{Dirac P1}\ and, in particular, it is an isomorphism.
\end{proof}

\subsection{Laplace operators attached to $\mathbb{I}$-valued perfect
symmetric pairings}

We suppose in this subsection that we are given $\psi :X\otimes
X\rightarrow \mathbb{I}$ which is perfect, i.e. such that the associated $%
\hom $ valued morphism is an isomorphism, and symmetric, i.e. $\psi \circ
\tau _{X,X}=\psi $. It follows from Lemma \ref{Dirac L1} that we have $%
\Delta _{\psi ,a}^{n}=0$ and, hence, we concentrate on $\Delta _{\psi
,s}^{n} $. We set $r_{X}:=\mathrm{rank}\left( X\right) $ in the subsequent
discussion.

\begin{proposition}
\label{Dirac P2}When $r_{X}>0$ we have that $\Delta _{\psi ,s}^{n+2}\circ
C_{\psi ,s}^{n}$ when $n\geq 0$ (resp. $C_{\psi ,s}^{n-2}\circ \Delta _{\psi
,s}^{n}$ when $n\geq 2$) is diagonalizable, with spectrum%
\begin{equation*}
\sigma \left( \Delta _{\psi ,a}^{n+2}\circ C_{\psi ,a}^{n}\right) >0\text{
(resp. }\sigma \left( C_{\psi ,a}^{n-2}\circ \Delta _{\psi ,a}^{n}\right)
\geq 0\text{).}
\end{equation*}
\end{proposition}

\begin{proof}
Setting $\delta _{\psi ,s}^{n}:=\frac{n\left( n-1\right) }{2}\cdot \Delta
_{\psi ,s}^{n}$, Lemma \ref{Dirac L3} gives the equalities $\delta _{\psi
,s}^{2}\circ C_{\psi ,s}^{0}=r_{X}$, $\delta _{\psi ,s}^{3}\circ C_{\psi
,s}^{1}=\left( 2+r_{X}\right) \cdot 1_{X}$\ and, for every $n\geq 2$,%
\begin{equation*}
\delta _{\psi ,s}^{n+2}\circ C_{\psi ,s}^{n}-C_{\psi ,s}^{n-2}\circ \delta
_{\psi ,s}^{n}=\left( 2n+r_{X}\right) \cdot 1_{\vee ^{n}X}\text{.}
\end{equation*}%
Then the proof is just a copy of those of Proposition \ref{Dirac P1}.
\end{proof}

\bigskip

The following corollary may be deduced from Proposition \ref{Dirac P2} in
the same way as Corollary \ref{Dirac C1} was deduced from Proposition \ref%
{Dirac P1}.

\begin{corollary}
\label{Dirac C2}When $r_{X}>0$ we have that, for every $n\geq 2$, the
Laplace operator $\Delta _{\psi ,s}^{n}$ has a section $s_{\psi
,s}^{n-2}:\vee ^{n-2}X\rightarrow \vee ^{n}X$ such that $\Delta _{\psi
,s}^{n}\circ s_{\psi ,s}^{n-2}=1_{\vee ^{n-2}X}$ and, in particular, $\ker
\left( \Delta _{\psi ,s}^{n}\right) $ exists.
\end{corollary}

\subsection{Laplace operators attached to perfect pairings valued in squares
of invertible objects}

We suppose in this subsection that we are given a perfect pairing $\psi
:X\otimes X\rightarrow Z$, i.e. such that $f_{\psi }:X\rightarrow \hom
\left( X,Z\right) $ is an isomorphism, and that we have $Z\simeq \mathbb{L}%
^{\otimes 2}$, where $\mathbb{L}$ an invertible object. We assume that $\psi 
$ is alternating or symmetric, i.e. $\psi \circ \tau _{X,X}=\chi \left( \tau
_{X,X}\right) \cdot \psi $, where $\chi \in \left\{ \varepsilon ,1\right\} $%
. As above, we define $\ast:=a$ when $\chi =\varepsilon $ and $\ast
:=s$ when $\chi =1$ and we write $\ast^{k}X:=\wedge ^{k}X$
when $\chi =\varepsilon $ and $\ast^{k}X:=\vee ^{k}X$ when $\chi =1$%
. It follows from Lemma \ref{Dirac L1} that we have $\Delta _{\psi ,\ast
}^{n}=0$ if $\left\{ \ast \right\} =\left\{ a,s\right\} -\left\{ \ast \right\} $ and , hence, we concentrate on $\Delta _{\psi ,\ast }^{n} $. We set $r_{X}:=\mathrm{rank}\left( X\right) $ and $r_{\mathbb{L}}:=%
\mathrm{rank}\left( \mathbb{L}\right) $ in the subsequent discussion, so
that $r_{\mathbb{L}}\in \left\{ \pm 1\right\} $.

\bigskip

Let $\tau _{\delta _{k}}:\otimes ^{k}\left( X\otimes \mathbb{L}\right) 
\overset{\sim }{\rightarrow }\left( \otimes ^{k}X\right) \otimes \mathbb{L}%
^{\otimes k}$ be the isomorphism induced by the permutation $\delta _{k}\in
S_{2k}$ such that $\delta _{k}\left( 2i-1\right) =i$ and $\delta _{k}\left(
2i\right) =k+i$ for every $i\in I_{k}$. It is not difficult to show, using 
\cite[7.2 Lemme]{De}, that one has%
\begin{eqnarray}
e_{X\otimes \mathbb{L},a}^{k} &\simeq &e_{X,a}^{k}\otimes 1_{\mathbb{L}%
^{\otimes k}}\text{ and }e_{X\otimes \mathbb{L},s}^{k}\simeq
e_{X,s}^{k}\otimes 1_{\mathbb{L}^{\otimes k}}\text{ if }r_{\mathbb{L}}=1%
\text{,}  \label{Traces invertible elements 1} \\
e_{X\otimes \mathbb{L},s}^{k} &\simeq &e_{X,a}^{k}\otimes 1_{\mathbb{L}%
^{\otimes k}}\text{ and }e_{X\otimes \mathbb{L},a}^{k}\simeq
e_{X,s}^{k}\otimes 1_{\mathbb{L}^{\otimes k}}\text{ if }r_{\mathbb{L}}=-1%
\text{.}  \label{Traces invertible elements -1}
\end{eqnarray}

\begin{lemma}
\label{Dirac L4}Suppose that $\varphi :X\otimes X\rightarrow \mathbb{L}%
^{\otimes 2}$ is alternating (resp. symmetric) and consider the composite%
\begin{equation*}
\varphi _{\mathbb{L}^{-1}}:\left( X\otimes \mathbb{L}^{-1}\right) \otimes
\left( X\otimes \mathbb{L}^{-1}\right) \overset{1_{X}\otimes \tau _{\mathbb{L%
}^{-1},X}\otimes 1_{\mathbb{L}^{-1}}}{\rightarrow }X\otimes X\otimes \mathbb{%
L}^{\otimes -2}\overset{\varphi \otimes 1_{\mathbb{L}^{\otimes -2}}}{%
\rightarrow }\mathbb{L}^{\otimes 2}\otimes \mathbb{L}^{\otimes -2}\overset{%
ev_{\mathbb{L}^{\otimes -2}}}{\rightarrow }\mathbb{I}\text{.}
\end{equation*}

\begin{itemize}
\item[$\left( a\right) $] If $r_{\mathbb{L}}=1$, the morphism $\varphi _{%
\mathbb{L}^{-1}}$ is alternating (resp. symmetric) and the following
diagrams are commutative%
\begin{equation*}
\xymatrix@C=30pt{ \wedge^{n}\left(X\otimes\mathbb{L}^{-1}\right) \ar[r]^-{\Delta_{\varphi_{\mathbb{L}^{-1}},a}^{n}} \ar[d]|{\tau_{\delta_{n}}} & \wedge^{n-2}\left(X\otimes\mathbb{L}^{-1}\right) \ar[d]|{\left(1_{\wedge^{n-2}X}\otimes\tau_{\mathbb{L}^{\otimes-\left(n-2\right)},\mathbb{L}^{\otimes2}}\otimes1_{\mathbb{L}^{\otimes-2}}\right)\circ\tau_{\delta_{n-2}}} & \vee^{n}\left(X\otimes\mathbb{L}^{-1}\right) \ar[r]^-{\Delta_{\varphi_{\mathbb{L}^{-1}},s}^{n}} \ar[d]|{\tau_{\delta_{n}}} & \vee^{n-2}\left(X\otimes\mathbb{L}^{-1}\right) \ar[d]|{\left(1_{\vee^{n-2}X}\otimes\tau_{\mathbb{L}^{\otimes-\left(n-2\right)},\mathbb{L}^{\otimes2}}\otimes1_{\mathbb{L}^{\otimes-2}}\right)\circ\tau_{\delta_{n-2}}} \\ \left(\wedge^{n}X\right)\otimes\mathbb{L}^{\otimes-n} \ar[r]^-{\Delta_{\varphi,a}^{n}\otimes1_{\mathbb{L}^{\otimes-n}}} & \left(\wedge^{n-2}X\right)\otimes\mathbb{L}^{\otimes2}\otimes\mathbb{L}^{\otimes-n}\text{,} & \left(\vee^{n}X\right)\otimes\mathbb{L}^{\otimes-n} \ar[r]^-{\Delta_{\varphi,s}^{n}\otimes1_{\mathbb{L}^{\otimes-n}}} & \left(\vee^{n-2}X\right)\otimes\mathbb{L}^{\otimes2}\otimes\mathbb{L}^{\otimes-n}\text{,}}
\end{equation*}%
where $\tau _{\delta _{k}}:\wedge ^{k}\left( X\otimes \mathbb{L}\right) 
\overset{\sim }{\rightarrow }\left( \wedge ^{k}X\right) \otimes \mathbb{L}%
^{\otimes k}$ and $\tau _{\delta _{k}}:\vee ^{k}\left( X\otimes \mathbb{L}%
\right) \overset{\sim }{\rightarrow }\left( \vee ^{k}X\right) \otimes 
\mathbb{L}^{\otimes k}$ are the isomorphisms induced by $\left( \text{\ref%
{Traces invertible elements 1}}\right) $.

\item[$\left( b\right) $] If $r_{\mathbb{L}}=-1$ the morphism $\varphi _{%
\mathbb{L}^{-1}}$ is symmetric (resp. alternating) and the following
diagrams are commutative:%
\begin{equation*}
\xymatrix@C=30pt{ \vee^{n}\left(X\otimes\mathbb{L}^{-1}\right) \ar[r]^-{\Delta_{\varphi_{\mathbb{L}^{-1}},s}^{n}} \ar[d]|{\tau_{\delta_{n}}} & \vee^{n-2}\left(X\otimes\mathbb{L}^{-1}\right) \ar[d]|{\left(1_{\wedge^{n-2}X}\otimes\tau_{\mathbb{L}^{\otimes-\left(n-2\right)},\mathbb{L}^{\otimes2}}\otimes1_{\mathbb{L}^{\otimes-2}}\right)\circ\tau_{\delta_{n-2}}} & \wedge^{n}\left(X\otimes\mathbb{L}^{-1}\right) \ar[r]^-{\Delta_{\varphi_{\mathbb{L}^{-1}},a}^{n}} \ar[d]|{\tau_{\delta_{n}}} & \wedge^{n-2}\left(X\otimes\mathbb{L}^{-1}\right) \ar[d]|{\left(1_{\vee^{n-2}X}\otimes\tau_{\mathbb{L}^{\otimes-\left(n-2\right)},\mathbb{L}^{\otimes2}}\otimes1_{\mathbb{L}^{\otimes-2}}\right)\circ\tau_{\delta_{n-2}}} \\ \left(\wedge^{n}X\right)\otimes\mathbb{L}^{\otimes-n} \ar[r]^-{\Delta_{\varphi,a}^{n}\otimes1_{\mathbb{L}^{\otimes-n}}} & \left(\wedge^{n-2}X\right)\otimes\mathbb{L}^{\otimes2}\otimes\mathbb{L}^{\otimes-n}\text{,} & \left(\vee^{n}X\right)\otimes\mathbb{L}^{\otimes-n} \ar[r]^-{\Delta_{\varphi,s}^{n}\otimes1_{\mathbb{L}^{\otimes-n}}} & \left(\vee^{n-2}X\right)\otimes\mathbb{L}^{\otimes2}\otimes\mathbb{L}^{\otimes-n}\text{,}}
\end{equation*}%
where $\tau _{\delta _{k}}:\vee ^{k}\left( X\otimes \mathbb{L}\right) 
\overset{\sim }{\rightarrow }\left( \wedge ^{k}X\right) \otimes \mathbb{L}%
^{\otimes k}$ and $\tau _{\delta _{k}}:\wedge ^{k}\left( X\otimes \mathbb{L}%
\right) \overset{\sim }{\rightarrow }\left( \vee ^{k}X\right) \otimes 
\mathbb{L}^{\otimes k}$ are the isomorphisms induced by $\left( \text{\ref%
{Traces invertible elements -1}}\right) $.

\item[$\left( c\right) $] Writing $f_{\varphi }:X\rightarrow \hom \left( X,%
\mathbb{L}^{\otimes 2}\right) $ and $f_{\varphi _{\mathbb{L}^{-1}}}:X\otimes 
\mathbb{L}^{-1}\rightarrow \left( X\otimes \mathbb{L}^{-1}\right) ^{\vee }$
for the associated morphisms we have that $f_{\varphi }$ is an isomorphism
is and only if $f_{\varphi _{\mathbb{L}^{-1}}}$ is an isomorphism.
\end{itemize}
\end{lemma}

\begin{proof}
$\left( a\text{-}b\right) $ We first claim that the following diagram is
commutative:%
\begin{equation}
\xymatrix{ \otimes^{n}\left(X\otimes\mathbb{L}^{-1}\right) \ar[r]^-{\Delta_{\varphi_{\mathbb{L}^{-1}}}^{n}} \ar[d]_{\tau_{\delta_{n}}} & \otimes^{n-2}\left(X\otimes\mathbb{L}^{-1}\right) \ar@/^{0.75pc}/[dr]^{\tau_{\delta_{n-2}}} & \\ \left(\otimes^{n}X\right)\otimes\mathbb{L}^{\otimes-n} \ar[r]_-{\Delta_{\varphi}^{n}\otimes1_{\mathbb{L}^{\otimes-n}}} & \left(\otimes^{n-2}X\right)\otimes\mathbb{L}^{\otimes2}\otimes\mathbb{L}^{\otimes-n} \ar[r]_-{1_{\otimes^{n-2}X}\otimes\tau_{\mathbb{L}^{\otimes2},\mathbb{L}^{\otimes-\left(n-2\right)}}\otimes1_{\mathbb{L}^{\otimes-2}}} & \left(\otimes^{n-2}X\right)\otimes\mathbb{L}^{\otimes-\left(n-2\right)} }  \label{Dirac L4 D}
\end{equation}%
A tedious computation reveals that:%
\begin{equation}
\left( \tau _{\delta _{n-2}}\otimes 1_{X}\otimes \tau _{\mathbb{L}%
^{-1},X}\otimes 1_{\mathbb{L}^{-1}}\right) =\left( 1_{\otimes
^{n-2}X}\otimes \tau _{\otimes ^{2}X,\mathbb{L}^{\otimes -\left( n-2\right)
}}\otimes 1_{\mathbb{L}^{\otimes -2}}\right) \circ \tau _{\delta _{n}}
\label{Dirac L4 F}
\end{equation}%
Hence we have:%
\begin{eqnarray*}
\tau _{\delta _{n-2}}\circ \Delta _{\varphi _{\mathbb{L}^{-1}}}^{n-2}
&=&\tau _{\delta _{n-2}}\circ \left( 1_{\otimes ^{n-2}\left( X\otimes 
\mathbb{L}^{-1}\right) }\otimes \varphi _{\mathbb{L}^{-1}}\right) =\tau
_{\delta _{n-2}}\circ \left( 1_{\otimes ^{n-2}\left( X\otimes \mathbb{L}%
^{-1}\right) }\otimes \varphi \otimes 1_{\mathbb{L}^{\otimes -2}}\right) \\
&&\circ \left( 1_{\otimes ^{n-2}\left( X\otimes \mathbb{L}^{-1}\right)
}\otimes 1_{X}\otimes \tau _{\mathbb{L}^{-1},X}\otimes 1_{\mathbb{L}%
^{-1}}\right) \\
&=&\left( 1_{\left( \otimes ^{n-2}X\right) \otimes \mathbb{L}^{\otimes
-\left( n-2\right) }}\otimes \varphi \otimes 1_{\mathbb{L}^{\otimes
-2}}\right) \circ \left( \tau _{\delta _{n-2}}\otimes 1_{X}\otimes \tau _{%
\mathbb{L}^{-1},X}\otimes 1_{\mathbb{L}^{-1}}\right) \text{ (by }\left( 
\text{\ref{Dirac L4 F}}\right) \text{)} \\
&=&\left( 1_{\left( \otimes ^{n-2}X\right) \otimes \mathbb{L}^{\otimes
-\left( n-2\right) }}\otimes \varphi \otimes 1_{\mathbb{L}^{\otimes
-2}}\right) \circ \left( 1_{\otimes ^{n-2}X}\otimes \tau _{\otimes ^{2}X,%
\mathbb{L}^{\otimes -\left( n-2\right) }}\otimes 1_{\mathbb{L}^{\otimes
-2}}\right) \circ \tau _{\delta _{n}} \\
&=&\left( 1_{\otimes ^{n-2}X}\otimes \tau _{\mathbb{L}^{\otimes 2},\mathbb{L}%
^{\otimes -\left( n-2\right) }}\otimes 1_{\mathbb{L}^{\otimes -2}}\right)
\circ \left( 1_{\otimes ^{n-2}X}\otimes \varphi \otimes 1_{\mathbb{L}%
^{\otimes -n}}\right) \circ \tau _{\delta _{n}} \\
&=&\left( 1_{\otimes ^{n-2}X}\otimes \tau _{\mathbb{L}^{\otimes 2},\mathbb{L}%
^{\otimes -\left( n-2\right) }}\otimes 1_{\mathbb{L}^{\otimes -2}}\right)
\circ \left( \Delta _{\varphi }^{n}\otimes 1_{\mathbb{L}^{\otimes
-n}}\right) \circ \tau _{\delta _{n}}\text{,}
\end{eqnarray*}%
showing that $\left( \text{\ref{Dirac L4 D}}\right) $ is commutative. The
claimed commutative diagrams in $\left( a\right) $ and $\left( b\right) $
now follows from $\left( \text{\ref{Traces invertible elements 1}}\right) $, 
$\left( \text{\ref{Traces invertible elements -1}}\right) $ and the
commutativity of $\left( \text{\ref{Dirac L4 D}}\right) $.

We view $1_{X}\otimes \tau _{\mathbb{L}^{-1},X}\otimes 1_{\mathbb{L}%
^{-1}}=\tau _{\left( 23\right) }$ and $\tau _{X\otimes \mathbb{L}%
^{-1},X\otimes \mathbb{L}^{-1}}=\tau _{\left( 13\right) \left( 24\right) }$
as induced by permutations in $S_{4}$ and then, noticing that $\left(
23\right) \left( 13\right) \left( 24\right) =\left( 1243\right) =\left(
12\right) \left( 34\right) \left( 23\right) $ and that we have $\varphi
\circ \tau _{X,X}=\chi \left( \tau _{X,X}\right) \cdot \psi $ with $\chi
=\varepsilon $ (resp. $\chi =1$), we find%
\begin{eqnarray*}
\varphi _{\mathbb{L}^{-1}}\circ \tau _{X\otimes \mathbb{L}^{-1},X\otimes 
\mathbb{L}^{-1}} &=&ev_{\mathbb{L}^{\otimes -2}}\circ \left( \varphi \otimes
1_{\mathbb{L}^{\otimes -2}}\right) \circ \left( 1_{X}\otimes \tau _{\mathbb{L%
}^{-1},X}\otimes 1_{\mathbb{L}^{-1}}\right) \circ \tau _{X\otimes \mathbb{L}%
^{-1},X\otimes \mathbb{L}^{-1}} \\
&=&ev_{\mathbb{L}^{\otimes -2}}\circ \left( \varphi \otimes 1_{\mathbb{L}%
^{\otimes -2}}\right) \circ \tau _{\left( 23\right) }\circ \tau _{\left(
13\right) \left( 24\right) }=e\circ \left( \varphi \otimes 1_{\mathbb{L}%
^{\otimes -2}}\right) \circ \tau _{\left( 1243\right) } \\
&=&ev_{\mathbb{L}^{\otimes -2}}\circ \left( \varphi \otimes 1_{\mathbb{L}%
^{\otimes -2}}\right) \circ \tau _{\left( 12\right) }\circ \tau _{\left(
34\right) }\circ \tau _{\left( 23\right) } \\
&=&ev_{\mathbb{L}^{\otimes -2}}\circ \left( \varphi \otimes 1_{\mathbb{L}%
^{\otimes -2}}\right) \circ \left( \tau _{X,X}\otimes \tau _{\mathbb{L}^{-1},%
\mathbb{L}^{-1}}\right) \circ \left( 1_{X}\otimes \tau _{\mathbb{L}%
^{-1},X}\otimes 1_{\mathbb{L}^{-1}}\right) \\
&=&\chi \left( \tau _{X,X}\right) \cdot ev_{\mathbb{L}^{\otimes -2}}\circ
\left( \varphi \otimes \tau _{\mathbb{L}^{-1},\mathbb{L}^{-1}}\right) \circ
\left( 1_{X}\otimes \tau _{\mathbb{L}^{-1},X}\otimes 1_{\mathbb{L}%
^{-1}}\right) \text{.}
\end{eqnarray*}%
It follows from \cite[7.2 Lemme]{De} that we have $\tau _{\mathbb{L}^{-1},%
\mathbb{L}^{-1}}=r_{\mathbb{L}}$, so that we find%
\begin{equation*}
\varphi _{\mathbb{L}^{-1}}\circ \tau _{X\otimes \mathbb{L}^{-1},X\otimes 
\mathbb{L}^{-1}}=r_{\mathbb{L}}\chi \left( \tau _{X,X}\right) \cdot \varphi
_{\mathbb{L}^{-1}}\text{.}
\end{equation*}

$\left( c\right) $ This is left to the reader.
\end{proof}

\bigskip

\begin{proposition}
\label{Dirac P3}When $\chi \left( \tau _{X,X}\right) r_{X}>0$ we have that,
for every $n\geq 2$, the Laplace operator%
\begin{equation*}
\Delta _{\psi ,\ast}^{n}:\ast^{n}X\rightarrow \ast ^{n-2}X
\end{equation*}%
has a section $s_{\psi ,\ast}^{n-2}:\ast^{n-2}X\rightarrow
\ast^{n}X$ such that $\Delta _{\psi ,\ast}^{n}\circ
s_{\psi ,\ast}^{n-2}=1_{\ast^{n-2}X}$ and, in particular, $%
\ker \left( \Delta _{\psi ,\ast}^{n}\right) $ exists.
\end{proposition}

\begin{proof}
If $\sigma :Z\overset{\sim }{\rightarrow }\mathbb{L}^{\otimes 2}$ is our
given isomorphism, we have that $\Delta _{\sigma \circ \psi ,\ast }^{n}=\left( 1_{\wedge ^{n-2}X}\otimes \sigma \right) \circ \Delta _{\psi
,\ast}^{n}$ has a section if and only if $\Delta _{\psi ,a}^{n}$
has a section: hence we may assume that $Z=\mathbb{L}^{\otimes 2}$. We can
now consider the composite:%
\begin{equation*}
\psi _{\mathbb{L}^{-1}}:\left( X\otimes \mathbb{L}^{-1}\right) \otimes
\left( X\otimes \mathbb{L}^{-1}\right) \overset{1_{X}\otimes \tau _{\mathbb{L%
}^{-1},X}\otimes 1_{\mathbb{L}^{-1}}}{\rightarrow }X\otimes X\otimes \mathbb{%
L}^{\otimes -2}\overset{\psi \otimes 1_{\mathbb{L}^{\otimes -2}}}{%
\rightarrow }\mathbb{L}^{\otimes 2}\otimes \mathbb{L}^{\otimes -2}\overset{%
ev_{\mathbb{L}^{\otimes -2}}}{\rightarrow }\mathbb{I}\text{.}
\end{equation*}%
When $r_{\mathbb{L}}=1$ (resp. $r_{\mathbb{L}}=-1$), Lemma \ref{Dirac L4} $%
\left( a\right) $ (resp. $\left( b\right) $)\ shows that $\Delta _{\psi
,\ast}^{n}$ has a section if and only if $\Delta _{\varphi _{%
\mathbb{L}^{-1}},\ast}^{n}$ (resp. $\Delta _{\varphi _{\mathbb{L}%
^{-1}},\ast }^{n}$) has a section and that $\varphi _{\mathbb{L}^{-1}}$
satisfies $\varphi _{\mathbb{L}^{-1}}\circ \tau _{X\otimes \mathbb{L}%
^{-1},X\otimes \mathbb{L}^{-1}}=\chi \left( \tau _{X,X}\right) \cdot \varphi
_{\mathbb{L}^{-1}}$ (resp. $\varphi _{\mathbb{L}^{-1}}\circ \tau _{X\otimes 
\mathbb{L}^{-1},X\otimes \mathbb{L}^{-1}}=-\chi \left( \tau _{X,X}\right)
\cdot \varphi _{\mathbb{L}^{-1}}$). It follows from this last relation that,
if we define $\varepsilon _{X\otimes \mathbb{L}^{-1},X\otimes \mathbb{L}%
^{-1}}$ by the rule $\varphi _{\mathbb{L}^{-1}}\circ \tau _{X\otimes \mathbb{%
L}^{-1},X\otimes \mathbb{L}^{-1}}=\varepsilon _{X\otimes \mathbb{L}%
^{-1},X\otimes \mathbb{L}^{-1}}\cdot \varphi _{\mathbb{L}^{-1}}$, then we
have $\varepsilon _{X\otimes \mathbb{L}^{-1},X\otimes \mathbb{L}^{-1}}=\chi
\left( \tau _{X,X}\right) $ (resp. $\varepsilon _{X\otimes \mathbb{L}%
^{-1},X\otimes \mathbb{L}^{-1}}=-\chi \left( \tau _{X,X}\right) $) and%
\begin{equation*}
\varepsilon _{X\otimes \mathbb{L}^{-1},X\otimes \mathbb{L}^{-1}}r_{X\otimes 
\mathbb{L}^{-1}}=\varepsilon _{X\otimes \mathbb{L}^{-1},X\otimes \mathbb{L}%
^{-1}}r_{X}r_{\mathbb{L}}=\chi \left( \tau _{X,X}\right) r_{X}>0\text{.}
\end{equation*}%
It follows that we may apply to $\psi _{\mathbb{L}^{-1}}$ Corollary \ref%
{Dirac C1}, when $\varepsilon _{X\otimes \mathbb{L}^{-1},X\otimes \mathbb{L}%
^{-1}}=-1$, or Corollary \ref{Dirac C2}, when $\varepsilon _{X\otimes 
\mathbb{L}^{-1},X\otimes \mathbb{L}^{-1}}=1$, to deduce that $\Delta
_{\varphi _{\mathbb{L}^{-1}},\ast}^{n}$ has a section.
\end{proof}

\section{Laplace and Dirac operators for the alternating algebras}

In this section we assume that we are given an object $V\in \mathcal{C}$such that $\wedge ^{g}V$ is invertible. If $X$ is an object we set 
$r_{X}:=\mathrm{rank}\left( X\right) $, so that $r_{\wedge ^{g}V}\in \left\{
\pm 1\right\} $, and we use the shorthand $r:=r_{V}$.

\subsection{\label{Subsection LD Alternating Lemmas}Preliminary lemmas}

We define%
\begin{equation*}
\psi _{i,1}^{V}:\wedge ^{i}V\otimes V\overset{\varphi _{i,1}}{\rightarrow }%
\wedge ^{i+1}V\overset{D^{i+1,g}}{\rightarrow }\wedge ^{g-i-1}V^{\vee
}\otimes \wedge ^{g}V^{\vee \vee }\text{,}
\end{equation*}%
and%
\begin{eqnarray*}
\overline{\psi }_{g-i,g-i-1}^{V} &:&\wedge ^{g-i}V\otimes \wedge
^{g-i-1}V^{\vee }\overset{D^{g-i,g}\otimes 1_{\wedge ^{g-i-1}V^{\vee }}}{%
\rightarrow }\wedge ^{i}V^{\vee }\otimes \wedge ^{g}V^{\vee \vee }\otimes
\wedge ^{g-i-1}V^{\vee }\overset{\varphi _{i,g-i-1}^{13}}{\rightarrow }%
\wedge ^{g-1}V^{\vee }\otimes \wedge ^{g}V^{\vee \vee } \\
&&\overset{D_{g-1,g}\otimes 1_{\wedge ^{g}V^{\vee \vee }}}{\rightarrow }%
V\otimes \wedge ^{g}V^{\vee }\otimes \wedge ^{g}V^{\vee \vee }\overset{%
1_{V}\otimes ev_{V^{\vee },a}^{g,\tau }}{\rightarrow }V\text{.}
\end{eqnarray*}

We may also consider%
\begin{equation*}
\psi _{g-i,1}^{V}:\wedge ^{g-i}V\otimes V\overset{\varphi _{g-i,1}}{%
\rightarrow }\wedge ^{g-i+1}V\overset{D^{g-i+1,g}}{\rightarrow }\wedge
^{i-1}V^{\vee }\otimes \wedge ^{g}V^{\vee \vee }\text{,}
\end{equation*}%
and%
\begin{eqnarray*}
\overline{\psi }_{i,i-1}^{V} &:&\wedge ^{i}V\otimes \wedge ^{i-1}V^{\vee }%
\overset{D^{i,g}\otimes 1_{\wedge ^{i-1}V^{\vee }}}{\rightarrow }\wedge
^{g-i}V^{\vee }\otimes \wedge ^{g}V^{\vee \vee }\otimes \wedge ^{i-1}V^{\vee
}\overset{\varphi _{g-i,i-1}^{13}}{\rightarrow }\wedge ^{g-1}V^{\vee
}\otimes \wedge ^{g}V^{\vee \vee } \\
&&\overset{D_{g-1,g}\otimes 1_{\wedge ^{g}V^{\vee \vee }}}{\rightarrow }%
V\otimes \wedge ^{g}V^{\vee }\otimes \wedge ^{g}V^{\vee \vee }\overset{%
1_{V}\otimes ev_{V^{\vee },a}^{g,\tau }}{\rightarrow }V\text{.}
\end{eqnarray*}

\begin{lemma}
\label{Dirac Alternating L1}Setting%
\begin{eqnarray*}
&&\rho _{V}^{i,g-i}:=\left( -1\right) ^{\left( g-1\right) }r_{\wedge ^{g}V}%
\binom{g}{g-1}^{-1}\binom{g}{g-i}^{-1}\binom{r-1}{g-1}\binom{r-i}{g-i}g\text{%
,} \\
&&\nu _{V}^{g-i,1}:=\left( -1\right) ^{g-i}i\text{ and }\nu
_{V}^{i,1}:=\left( -1\right) ^{i\left( g-i-1\right) }\left( g-i\right)
\end{eqnarray*}%
the following diagram is commutative:%
\begin{equation*}
\xymatrix@C=170pt{ \wedge^{i}V\otimes\wedge^{g-i}V\otimes V \ar[r]^-{\left(1_{\wedge^{i}V}\otimes\psi_{g-i,1},\left(1_{\wedge^{g-i}V}\otimes\psi_{i,1}\right)\circ\left(\tau_{\wedge^{i}V,\wedge^{g-i}V}\otimes1_{V}\right)\right)} \ar[d]|{\varphi_{i,g-i}\otimes1_{V}} & \wedge^{i}V\otimes\wedge^{i-1}V^{\vee}\otimes\wedge^{g}V^{\vee\vee}\oplus\wedge^{g-i}V\otimes\wedge^{g-i-1}V^{\vee}\otimes\wedge^{g}V^{\vee\vee} \ar[d]|{\nu_{V}^{g-i,1}\cdot\overline{\psi}_{i,i-1}\otimes1_{\wedge^{g}V^{\vee\vee}}\oplus\nu_{V}^{i,1}\cdot\overline{\psi}_{g-i,g-i-1}\otimes1_{\wedge^{g}V^{\vee\vee}}} \\ \wedge^{g}V\otimes V \ar[r]^-{\rho_{V}^{i,g-i}\cdot\tau_{\wedge^{g}V^{\vee\vee},V}\circ\left(i_{\wedge^{g}V}\otimes1_{V}\right)} & V\otimes\wedge^{g}V^{\vee\vee}\text{.} }
\end{equation*}
\end{lemma}

\begin{proof}
Set%
\begin{eqnarray*}
&&a:=\left( -1\right) ^{g-i}i\cdot \left( \varphi _{g-i,i-1}^{13}\otimes
1_{\wedge ^{g}V^{\vee \vee }}\right) \circ \left( D^{i,g}\otimes 1_{\wedge
^{i-1}V^{\vee }\otimes \wedge ^{g}V^{\vee \vee }}\right) \circ \left(
1_{\wedge ^{i}V}\otimes D^{g-i+1,g}\right) \circ \left( 1_{\wedge
^{i}V}\otimes \varphi _{g-i,1}\right)  \\
&&\text{ \ \ }=\left( -1\right) ^{g-i}i\cdot \left( \varphi
_{g-i,i-1}^{13}\otimes 1_{\wedge ^{g}V^{\vee \vee }}\right) \circ \left(
D^{i,g}\otimes D^{g-i+1,g}\right) \circ \left( 1_{\wedge ^{i}V}\otimes
\varphi _{g-i,1}\right) \text{,} \\
&&b:=\left( -1\right) ^{i\left( g-i-1\right) }\left( g-i\right) \cdot \left(
\varphi _{i,g-i-1}^{13}\otimes 1_{\wedge ^{g}V^{\vee \vee }}\right) \circ
\left( D^{g-i,g}\otimes 1_{\wedge ^{g-i-1}V^{\vee }\otimes \wedge
^{g}V^{\vee \vee }}\right) \circ \left( 1_{\wedge ^{g-i}V}\otimes
D^{i+1,g}\right)  \\
&&\text{ \ \ \ \ \ \ }\circ \left( 1_{\wedge ^{g-i}V}\otimes \varphi
_{i,1}\right) \circ \left( \tau _{\wedge ^{i}V,\wedge ^{g-i}V}\otimes
1_{V}\right)  \\
&&\text{ \ \ }=\left( -1\right) ^{i\left( g-i-1\right) }\left( g-i\right)
\cdot \left( \varphi _{i,g-i-1}^{13}\otimes 1_{\wedge ^{g}V^{\vee \vee
}}\right) \circ \left( D^{g-i,g}\otimes D^{i+1,g}\right)  \\
&&\text{ \ \ \ \ \ \ }\circ \left( 1_{\wedge ^{g-i}V}\otimes \varphi
_{i,1}\right) \circ \left( \tau _{\wedge ^{i}V,\wedge ^{g-i}V}\otimes
1_{V}\right) \text{,} \\
&&\phi :=\left( 1_{V}\otimes ev_{V^{\vee },a}^{g,\tau }\otimes 1_{\wedge
^{g}V^{\vee \vee }}\right) \circ \left( D_{g-1,g}\otimes 1_{\wedge
^{g}V^{\vee \vee }\otimes \wedge ^{g}V^{\vee \vee }}\right) \text{.}
\end{eqnarray*}%
Then we have $\left( \overline{\psi }_{i,i-1}^{V}\otimes 1_{\wedge
^{g}V^{\vee \vee }}\right) \circ \left( 1_{\wedge ^{i}V}\otimes \psi
_{g-i,1}^{V}\right) =\phi \circ a$ and $\left( \overline{\psi }%
_{g-i,g-i-1}^{V}\otimes 1_{\wedge ^{g}V^{\vee \vee }}\right) \circ \left(
1_{\wedge ^{g-i}V}\otimes \psi _{i,1}^{V}\right) \circ \left( \tau _{\wedge
^{i}V,\wedge ^{g-i}V}\otimes 1_{V}\right) =\phi \circ b$. With these
notations, it follows from Proposition \ref{Alternating algebras P2} that we
have, setting $\rho :=r_{\wedge ^{g}V}g\binom{g}{g-i}^{-1}\binom{r-i}{g-i}$%
\footnote{%
The morphism denoted by $\varphi _{g-i,i-1}^{13}$ (resp. $\varphi
_{i,g-i-1}^{13}$)\ in Proposition \ref{Alternating algebras P2} is the one
here denoted by $\varphi _{g-i,i-1}^{13}\otimes 1_{\wedge ^{g}V^{\vee \vee }}
$ (resp. $\varphi _{i,g-i-1}^{13}\otimes 1_{\wedge ^{g}V^{\vee \vee }}$).},%
\begin{equation}
a+b=\rho \cdot \left( 1_{\wedge ^{g-1}V^{\vee }\otimes \wedge ^{g}V^{\vee
\vee }}\otimes i_{\wedge ^{g}V}\right) \circ \left( D^{1,g}\otimes \varphi
_{i,g-i}\right) \circ \tau _{\wedge ^{i}V\otimes \wedge ^{g-i}V,V}\text{.}
\label{Dirac Alternating L1 F1}
\end{equation}%
Hence we find%
\begin{eqnarray*}
&&\left( \overline{\psi }_{i,i-1}^{V}\otimes 1_{\wedge ^{g}V^{\vee \vee
}}\right) \circ \left( 1_{\wedge ^{i}V}\otimes \psi _{g-i,1}^{V}\right)
+\left( \overline{\psi }_{g-i,g-i-1}^{V}\otimes 1_{\wedge ^{g}V^{\vee \vee
}}\right) \circ \left( 1_{\wedge ^{g-i}V}\otimes \psi _{i,1}^{V}\right)
\circ \left( \tau _{\wedge ^{i}V,\wedge ^{g-i}V}\otimes 1_{V}\right)  \\
&&\text{ \ }=\phi \circ \left( a+b\right) \text{ (by }\left( \text{\ref%
{Dirac Alternating L1 F1}}\right) \text{)}=\rho \cdot \phi \circ \left(
1_{\wedge ^{g-1}V^{\vee }\otimes \wedge ^{g}V^{\vee \vee }}\otimes i_{\wedge
^{g}V}\right) \circ \left( D^{1,g}\otimes \varphi _{i,g-i}\right) \circ \tau
_{\wedge ^{i}V\otimes \wedge ^{g-i}V,V} \\
&&\text{ \ }=\rho \cdot \left( 1_{V}\otimes ev_{V^{\vee },a}^{g,\tau
}\otimes 1_{\wedge ^{g}V^{\vee \vee }}\right) \circ \left( D_{g-1,g}\otimes
1_{\wedge ^{g}V^{\vee \vee }\otimes \wedge ^{g}V^{\vee \vee }}\right) \circ
\left( 1_{\wedge ^{g-1}V^{\vee }\otimes \wedge ^{g}V^{\vee \vee }}\otimes
i_{\wedge ^{g}V}\right)  \\
&&\text{ \ \ \ \ \ }\circ \left( D^{1,g}\otimes \varphi _{i,g-i}\right)
\circ \tau _{\wedge ^{i}V\otimes \wedge ^{g-i}V,V} \\
&&\text{ \ }=\rho \cdot \left( 1_{V}\otimes ev_{V^{\vee },a}^{g,\tau
}\otimes 1_{\wedge ^{g}V^{\vee \vee }}\right) \circ \left( D_{g-1,g}\otimes
1_{\wedge ^{g}V^{\vee \vee }\otimes \wedge ^{g}V^{\vee \vee }}\right) \circ
\left( 1_{\wedge ^{g-1}V^{\vee }\otimes \wedge ^{g}V^{\vee \vee }}\otimes
i_{\wedge ^{g}V}\right)  \\
&&\text{ \ \ \ \ \ }\circ \left( D^{1,g}\otimes 1_{\wedge ^{g}V}\right)
\circ \left( 1_{V}\otimes \varphi _{i,g-i}\right) \circ \tau _{\wedge
^{i}V\otimes \wedge ^{g-i}V,V} \\
&&\text{ \ }=\rho \cdot \left( 1_{V}\otimes ev_{V^{\vee },a}^{g,\tau
}\otimes 1_{\wedge ^{g}V^{\vee \vee }}\right) \circ \left( D_{g-1,g}\otimes
1_{\wedge ^{g}V^{\vee \vee }\otimes \wedge ^{g}V^{\vee \vee }}\right) \circ
\left( D^{1,g}\otimes 1_{\wedge ^{g}V^{\vee \vee }}\right)  \\
&&\text{ \ \ \ \ \ }\circ \left( 1_{V}\otimes i_{\wedge ^{g}V}\right) \circ
\left( 1_{V}\otimes \varphi _{i,g-i}\right) \circ \tau _{\wedge ^{i}V\otimes
\wedge ^{g-i}V,V}\text{ (by subsequent }\left( \text{\ref{Dirac Alternating
L1 F2}}\right) \text{)} \\
&&\text{ \ }=\left( -1\right) ^{\left( g-1\right) }\binom{g}{g-1}^{-1}\binom{%
r-1}{g-1}\rho \cdot \left( 1_{V}\otimes i_{\wedge ^{g}V}\right) \circ \left(
1_{V}\otimes \varphi _{i,g-i}\right) \circ \tau _{\wedge ^{i}V\otimes \wedge
^{g-i}V,V} \\
&&\text{ \ }=\left( -1\right) ^{\left( g-1\right) }\binom{g}{g-1}^{-1}\binom{%
r-1}{g-1}\rho \cdot \left( 1_{V}\otimes i_{\wedge ^{g}V}\right) \circ \tau
_{\wedge ^{g}V,V}\circ \left( \varphi _{i,g-i}\otimes 1_{V}\right)  \\
&&\text{ \ }=\left( -1\right) ^{\left( g-1\right) }\binom{g}{g-1}^{-1}\binom{%
r-1}{g-1}\rho \cdot \tau _{\wedge ^{g}V^{\vee \vee },V}\circ \left(
i_{\wedge ^{g}V}\otimes 1_{V}\right) \circ \left( \varphi _{i,g-i}\otimes
1_{V}\right) \text{,}
\end{eqnarray*}%
where we have employed the equality%
\begin{equation}
\left( 1_{V}\otimes ev_{V^{\vee },a}^{g,\tau }\right) \circ \left(
D_{g-1,g}\otimes 1_{\wedge ^{g}V^{\vee \vee }}\right) \circ D^{1,g}=\left(
-1\right) ^{\left( g-1\right) }\binom{g}{g-1}^{-1}\binom{r-1}{g-1}
\label{Dirac Alternating L1 F2}
\end{equation}%
from Theorem \ref{Alternating algebras T} $\left( 1\right) $.
\end{proof}

\bigskip

We now consider the following morphisms. We have%
\begin{equation*}
\psi _{i,g-i-1}^{V}:\wedge ^{i}V\otimes \wedge ^{g-i-1}V\overset{\varphi
_{i,g-i-1}}{\rightarrow }\wedge ^{g-1}V\overset{D^{g-1,g}}{\rightarrow }%
V^{\vee }\otimes \wedge ^{g}V^{\vee \vee }
\end{equation*}%
and%
\begin{eqnarray*}
\overline{\psi }_{g-i,1}^{V} &:&\wedge ^{g-i}V\otimes V^{\vee }\overset{%
D^{g-i,g}\otimes 1_{V^{\vee }}}{\rightarrow }\wedge ^{i}V^{\vee }\otimes
\wedge ^{g}V^{\vee \vee }\otimes V^{\vee }\overset{\varphi _{i,1}^{13}}{%
\rightarrow }\wedge ^{i+1}V^{\vee }\otimes \wedge ^{g}V^{\vee \vee } \\
&&\overset{D_{i+1,g}\otimes 1_{\wedge ^{g}V^{\vee \vee }}}{\rightarrow }%
\wedge ^{g-i-1}V\otimes \wedge ^{g}V^{\vee }\otimes \wedge ^{g}V^{\vee \vee }%
\overset{1_{\wedge ^{g-i-1}V}\otimes ev_{V^{\vee },a}^{g,\tau }}{\rightarrow 
}\wedge ^{g-i-1}V\text{.}
\end{eqnarray*}

On the other hand we have%
\begin{equation*}
\psi _{g-i,i-1}^{V}:\wedge ^{g-i}V\otimes \wedge ^{i-1}V\overset{\varphi
_{g-i,i-1}}{\rightarrow }\wedge ^{g-1}V\overset{D^{g-1,g}}{\rightarrow }%
V^{\vee }\otimes \wedge ^{g}V^{\vee \vee }
\end{equation*}%
and%
\begin{eqnarray*}
\overline{\psi }_{i,1}^{V} &:&\wedge ^{i}V\otimes V^{\vee }\overset{%
D^{i,g}\otimes 1_{V^{\vee }}}{\rightarrow }\wedge ^{g-i}V^{\vee }\otimes
\wedge ^{g}V^{\vee \vee }\otimes V^{\vee }\overset{\varphi _{g-i,1}^{13}}{%
\rightarrow }\wedge ^{g-i+1}V^{\vee }\otimes \wedge ^{g}V^{\vee \vee } \\
&&\overset{D_{g-i+1,g}\otimes 1_{\wedge ^{g}V^{\vee \vee }}}{\rightarrow }%
\wedge ^{i-1}V\otimes \wedge ^{g}V^{\vee }\otimes \wedge ^{g}V^{\vee \vee }%
\overset{1_{\wedge ^{i-1}V}\otimes ev_{V^{\vee },a}^{g,\tau }}{\rightarrow }%
\wedge ^{i-1}V\text{.}
\end{eqnarray*}

The proof of the following result is a bit more involved and we will leave
some of the details to the reader.

\begin{lemma}
\label{Dirac Alternating L2}Setting%
\begin{eqnarray*}
&&\rho _{V^{\vee }}^{g-i,i}:=\left( -1\right) ^{\left( g-1\right) }\binom{g}{%
g-1}^{-1}\binom{g}{g-i}^{-1}\binom{g}{i}^{-1}\binom{r-1}{g-1}\binom{r-i}{g-i}%
\binom{r+i-g}{i}g\text{,} \\
&&\nu _{V^{\vee }}^{i,1}:=\left( -1\right) ^{\left( i+1\right) \left(
g-i\right) }r_{\wedge ^{g}V}\binom{g}{i}^{-1}\binom{r+i-g}{i}i\text{ and} \\
&&\nu _{V^{\vee }}^{g-i,1}:=\left( -1\right) ^{i}\binom{g}{g-i}^{-1}\binom{%
r-i}{g-i}\left( g-i\right) \text{,}
\end{eqnarray*}%
the following diagram is commutative:%
\begin{equation*}
\xymatrix@C=180pt{ \wedge^{g-i}V\otimes\wedge^{i}V\otimes V^{\vee} \ar[r]^-{\left(1_{\wedge^{g-i}V}\otimes\overline{\psi}_{i,1},\left(1_{\wedge^{i}V}\otimes\overline{\psi}_{g-i,1}\right)\circ\left(\tau_{\wedge^{g-i}V,\wedge^{i}V}\otimes1_{V^{\vee}}\right)\right)} \ar[d]|{\varphi_{g-i,i}\otimes1_{V^{\vee}}} & \wedge^{g-i}V\otimes\wedge^{i-1}V\oplus\wedge^{i}V\otimes\wedge^{g-i-1}V \ar[d]^{\nu_{V^{\vee}}^{i,1}\cdot\psi_{g-i,i-1}\oplus\nu_{V^{\vee}}^{g-i,1}\cdot\psi_{i,g-i-1}} \\ \wedge^{g}V\otimes V^{\vee} \ar[r]^-{\rho_{V^{\vee}}^{g-i,i}\cdot\left(1_{V^{\vee}}\otimes i_{\wedge^{g}V}\right)\circ\tau_{\wedge^{g}V,V^{\vee}}} & V^{\vee}\otimes\wedge^{g}V^{\vee\vee}\text{.} }
\end{equation*}
\end{lemma}

\begin{proof}
By definition%
\begin{eqnarray}
&&\psi _{g-i,i-1}^{V}\circ \left( 1_{\wedge ^{i}V}\otimes \overline{\psi }%
_{i,1}^{V}\right) =D^{g-1,g}\circ \varphi _{g-i,i-1}\circ \left( 1_{\wedge
^{g-i}V\otimes \wedge ^{i-1}V}\otimes ev_{V^{\vee },a}^{g,\tau }\right)  
\notag \\
&&\text{ \ \ \ \ }\circ \left( 1_{\wedge ^{g-i}V}\otimes D_{g-i+1,g}\otimes
1_{\wedge ^{g}V^{\vee \vee }}\right) \circ \left( 1_{\wedge ^{g-i}V}\otimes
\varphi _{g-i,1}^{13}\right) \circ \left( 1_{\wedge ^{g-i}V}\otimes
D^{i,g}\otimes 1_{V^{\vee }}\right) \text{,}
\label{Dirac Alternating L1 F1 def} \\
&&\psi _{i,g-i-1}^{V}\circ \left( 1_{\wedge ^{i}V}\otimes \overline{\psi }%
_{g-i,1}^{V}\right) \circ \left( \tau _{\wedge ^{g-i}V,\wedge ^{i}V}\otimes
1_{V^{\vee }}\right) =D^{g-1,g}\circ \varphi _{i,g-i-1}\circ \left(
1_{\wedge ^{i}V\otimes \wedge ^{g-i-1}V}\otimes ev_{V^{\vee },a}^{g,\tau
}\right)   \notag \\
&&\text{ \ \ \ \ }\circ \left( 1_{\wedge ^{i}V}\otimes D_{i+1,g}\otimes
1_{\wedge ^{g}V^{\vee \vee }}\right) \circ \left( 1_{\wedge ^{i}V}\otimes
\varphi _{i,1}^{13}\right) \circ \left( 1_{\wedge ^{i}V}\otimes
D^{g-i,g}\otimes 1_{V^{\vee }}\right) \circ \left( \tau _{\wedge
^{g-i}V,\wedge ^{i}V}\otimes 1_{V^{\vee }}\right) \text{.}
\label{Dirac Alternating L1 F2 def}
\end{eqnarray}%
It follows from Theorem \ref{Alternating algebras T} $\left( 1\right) $ that
we have, setting $\mu _{g-i,g}:=\binom{g}{g-i}^{-1}\binom{r-i}{g-i}$ and $%
\mu _{i,g}:=\binom{g}{i}^{-1}\binom{r+i-g}{i}$:%
\begin{eqnarray}
\left( -1\right) ^{i\left( g-i\right) }\mu _{i,g}\cdot 1_{\wedge ^{g-i}V}
&=&\left( 1_{\wedge ^{g-i}V}\otimes ev_{V^{\vee },a}^{g,\tau }\right) \circ
\left( D_{i,g}\otimes 1_{\wedge ^{g}V^{\vee \vee }}\right) \circ D^{g-i,g}%
\text{,}  \label{Dirac Alternating L2 F1} \\
\left( -1\right) ^{i\left( g-i\right) }\mu _{g-i,g}\cdot 1_{\wedge ^{i}V}
&=&\left( 1_{\wedge ^{i}V}\otimes ev_{V^{\vee },a}^{g,\tau }\right) \circ
\left( D_{g-i,g}\otimes 1_{\wedge ^{g}V^{\vee \vee }}\right) \circ D^{i,g}%
\text{.}  \label{Dirac Alternating L2 F2}
\end{eqnarray}%
Inserting $\left( \text{\ref{Dirac Alternating L2 F1}}\right) $ in the
definition $\left( \text{\ref{Dirac Alternating L1 F1 def}}\right) $, one
checks that:%
\begin{eqnarray}
&&\left( -1\right) ^{i\left( g-i\right) }\left( -1\right) ^{g-i}\mu
_{i,g}i\cdot \psi _{g-i,i-1}^{V}\circ \left( 1_{\wedge ^{i}V}\otimes 
\overline{\psi }_{i,1}^{V}\right)   \notag \\
&&\text{ }=D^{g-1,g}\circ \left( 1_{\wedge ^{g-1}V}\otimes ev_{V^{\vee
},a}^{g,\tau }\otimes ev_{V^{\vee },a}^{g,\tau }\right) \circ a\circ \left(
D^{g-i,g}\otimes D^{i,g}\otimes 1_{V^{\vee }}\right) \text{,}
\label{Dirac Alternating L2 F3}
\end{eqnarray}%
where%
\begin{eqnarray*}
&&a:=\left( -1\right) ^{g-i}i\cdot \left( \varphi _{g-i,i-1}\otimes
1_{\wedge ^{g}V^{\vee }\otimes \wedge ^{g}V^{\vee \vee }\otimes \wedge
^{g}V^{\vee }\otimes \wedge ^{g}V^{\vee \vee }}\right)  \\
&&\text{ \ \ }\circ \left( 1_{\wedge ^{g-i}V}\otimes \tau _{\wedge
^{g}V^{\vee }\otimes \wedge ^{g}V^{\vee \vee },\wedge ^{i-1}V}\otimes
1_{\wedge ^{g}V^{\vee }\otimes \wedge ^{g}V^{\vee \vee }}\right)  \\
&&\text{ \ \ }\circ \left( D_{i,g}\otimes 1_{\wedge ^{g}V^{\vee \vee
}}\otimes D_{g-i+1,g}\otimes 1_{\wedge ^{g}V^{\vee \vee }}\right) \circ
\left( 1_{\wedge ^{i}V^{\vee }\otimes \wedge ^{g}V^{\vee \vee }}\otimes
\varphi _{g-i,1}^{13}\right) 
\end{eqnarray*}%
Similarly, inserting $\left( \text{\ref{Dirac Alternating L2 F2}}\right) $
in the definition $\left( \text{\ref{Dirac Alternating L1 F2 def}}\right) $,
one finds:%
\begin{eqnarray}
&&\left( -1\right) ^{i\left( g-i\right) }\left( -1\right) ^{i\left(
g-i-1\right) }\mu _{g-i,g}\left( g-i\right) \cdot \psi _{i,g-i-1}^{V}\circ
\left( 1_{\wedge ^{i}V}\otimes \overline{\psi }_{g-i,1}^{V}\right) \circ
\left( \tau _{\wedge ^{g-i}V,\wedge ^{i}V}\otimes 1_{V^{\vee }}\right)  
\notag \\
&&\text{ }=D^{g-1,g}\circ \left( 1_{\wedge ^{g-1}V}\otimes ev_{V^{\vee
},a}^{g,\tau }\otimes ev_{V^{\vee },a}^{g,\tau }\right) \circ b\circ \left(
D^{g-i,g}\otimes D^{i,g}\otimes 1_{V^{\vee }}\right) 
\label{Dirac Alternating L2 F4}
\end{eqnarray}%
where%
\begin{eqnarray*}
&&b:=\left( -1\right) ^{i\left( g-i-1\right) }\left( g-i\right) \cdot \left(
\varphi _{i,g-i-1}\otimes 1_{\wedge ^{g}V^{\vee }\otimes \wedge ^{g}V^{\vee
\vee }\otimes \wedge ^{g}V^{\vee }\otimes \wedge ^{g}V^{\vee \vee }}\right) 
\\
&&\text{ \ \ \ \ \ }\circ \left( 1_{\wedge ^{i}V}\otimes \tau _{\wedge
^{g}V^{\vee }\otimes \wedge ^{g}V^{\vee \vee },\wedge ^{g-i-1}V}\otimes
1_{\wedge ^{g}V^{\vee }\otimes \wedge ^{g}V^{\vee \vee }}\right) \circ
\left( D_{g-i,g}\otimes 1_{\wedge ^{g}V^{\vee \vee }}\otimes
D_{i+1,g}\otimes 1_{\wedge ^{g}V^{\vee \vee }}\right)  \\
&&\text{ \ \ \ \ \ }\circ \left( 1_{\wedge ^{g-i}V^{\vee }\otimes \wedge
^{g}V^{\vee \vee }}\otimes \varphi _{i,1}^{13}\right) \circ \left( \tau
_{\wedge ^{i}V^{\vee }\otimes \wedge ^{g}V^{\vee \vee },\wedge ^{g-i}V^{\vee
}\otimes \wedge ^{g}V^{\vee \vee }}\otimes 1_{V^{\vee }}\right) 
\end{eqnarray*}%
and we have used a similar commutative diagram in the last equality.

By Proposition \ref{Alternating algebras P2} (second diagram)\ we have,
setting $\rho :=r_{\wedge ^{g}V}g\binom{g}{g-i}^{-1}\binom{r-i}{g-i}$,%
\begin{equation}
\rho \cdot \left( D_{1,g}\otimes \varphi _{i,g-i}\right) \circ \tau _{\wedge
^{i}V^{\vee }\otimes \wedge ^{g-i}V^{\vee },V^{\vee }}=a_{0}+b_{0}\text{,}
\label{Dirac Alternating L2 F5}
\end{equation}%
where%
\begin{eqnarray*}
&&a_{0}:=\left( -1\right) ^{g-i}i\cdot \left( \varphi _{g-i,i-1}\otimes
1_{\wedge ^{g}V^{\vee }\otimes \wedge ^{g}V^{\vee }}\right) \circ \left(
1_{\wedge ^{g-i}V}\otimes \tau _{\wedge ^{g}V^{\vee },\wedge ^{i-1}V}\otimes
1_{\wedge ^{g}V^{\vee }}\right) \\
&&\text{ \ \ \ \ \ \ }\circ \left( D_{i,g}\otimes D_{g-i+1,g}\right) \circ
\left( 1_{\wedge ^{i}V^{\vee }}\otimes \varphi _{g-i,1}\right) \text{,} \\
&&b_{0}:=\left( -1\right) ^{i\left( g-i-1\right) }\left( g-i\right) \cdot
\left( \varphi _{i,g-i-1}\otimes 1_{\wedge ^{g}V^{\vee }\otimes \wedge
^{g}V^{\vee }}\right) \circ \left( 1_{\wedge ^{i}V}\otimes \tau _{\wedge
^{g}V^{\vee },\wedge ^{g-i-1}V}\otimes 1_{\wedge ^{g}V^{\vee }}\right) \\
&&\text{ \ \ \ \ \ \ }\circ \left( D_{g-i,g}\otimes D_{i+1,g}\right) \circ
\left( 1_{\wedge ^{g-i}V^{\vee }}\otimes \varphi _{i,1}\right) \circ \left(
\tau _{\wedge ^{i}V^{\vee },\wedge ^{g-i}V^{\vee }}\otimes 1_{V^{\vee
}}\right) \text{.}
\end{eqnarray*}

Consider the morphism%
\begin{equation*}
\tau _{\left( 235\right) }:\wedge ^{i}V^{\vee }\otimes \wedge ^{g-i}V^{\vee
}\otimes V^{\vee }\otimes \wedge ^{g}V^{\vee \vee }\otimes \wedge
^{g}V^{\vee \vee }\rightarrow \wedge ^{i}V^{\vee }\otimes \wedge ^{g}V^{\vee
\vee }\otimes \wedge ^{g-i}V^{\vee }\otimes \wedge ^{g}V^{\vee \vee }\otimes
V^{\vee }
\end{equation*}%
attached to the permutation $\left( 235\right) $. After a tedious
computation one can verify the following relations:%
\begin{eqnarray}
&&\tau _{\left( 35\right) }\circ a\circ \tau _{\left( 235\right) }=\tau
_{\left( 34\right) }\circ \left( a_{0}\otimes 1_{\wedge ^{g}V^{\vee \vee
}\otimes \wedge ^{g}V^{\vee \vee }}\right) \text{,}
\label{Dirac Alternating L2 F5 a} \\
&&b\circ \tau _{\left( 235\right) }=\tau _{\left( 34\right) }\circ \left(
b_{0}\otimes 1_{\wedge ^{g}V^{\vee \vee }\otimes \wedge ^{g}V^{\vee \vee
}}\right) \text{,}  \label{Dirac Alternating L2 F5 b} \\
&&\tau _{\left( 345\right) }\circ \left( D_{1,g}\otimes \varphi
_{i,g-i}^{13}\right) \circ \tau _{\wedge ^{i}V^{\vee }\otimes \wedge
^{g}V^{\vee \vee }\otimes \wedge ^{g-i}V^{\vee }\otimes \wedge ^{g}V^{\vee
\vee },V^{\vee }}\circ \tau _{\left( 235\right) }=\tau _{\left( 34\right)
}\circ \left( c_{0}\otimes 1_{\wedge ^{g}V^{\vee \vee }\otimes \wedge
^{g}V^{\vee \vee }}\right) \text{.}  \label{Dirac Alternating L2 F5 c}
\end{eqnarray}

Thanks to $\left( \text{\ref{Dirac Alternating L2 F5 a}}\right) $, $\left( 
\text{\ref{Dirac Alternating L2 F5 b}}\right) $ and $\left( \text{\ref{Dirac
Alternating L2 F5 c}}\right) $, the equality $\left( \text{\ref{Dirac
Alternating L2 F5}}\right) $ gives%
\begin{equation}
\rho \cdot \tau _{\left( 345\right) }\circ \left( D_{1,g}\otimes \varphi
_{i,g-i}^{13}\right) \circ \tau _{\wedge ^{i}V^{\vee }\otimes \wedge
^{g}V^{\vee \vee }\otimes \wedge ^{g-i}V^{\vee }\otimes \wedge ^{g}V^{\vee
\vee },V^{\vee }}=\tau _{\left( 35\right) }\circ a+b\text{.}
\label{Dirac Alternating L2 F6}
\end{equation}%
Finally, we need to remark that we have the following commutative diagram
(by a computation of the involved permutations):%
\begin{equation*}
\xymatrix{ \wedge^{g-1}V\otimes\wedge^{g}V^{\vee}\otimes\wedge^{g}V^{\vee\vee}\otimes\wedge^{g}V^{\vee}\otimes\wedge^{g}V^{\vee\vee} \ar[r]^-{\tau_{\left(35\right)}} \ar[d]|{1_{\wedge^{g-1}V\otimes\wedge^{g}V^{\vee}}\otimes\tau_{\wedge^{g}V^{\vee\vee},\wedge^{g}V^{\vee}}\otimes1_{\wedge^{g}V^{\vee\vee}}} & \wedge^{g-1}V\otimes\wedge^{g}V^{\vee}\otimes\wedge^{g}V^{\vee\vee}\otimes\wedge^{g}V^{\vee}\otimes\wedge^{g}V^{\vee\vee} \ar[d]|{1_{\wedge^{g-1}V\otimes\wedge^{g}V^{\vee}}\otimes\tau_{\wedge^{g}V^{\vee\vee},\wedge^{g}V^{\vee}}\otimes1_{\wedge^{g}V^{\vee\vee}}} \\ \wedge^{g-1}V\otimes\wedge^{g}V^{\vee}\otimes\wedge^{g}V^{\vee}\otimes\wedge^{g}V^{\vee\vee}\otimes\wedge^{g}V^{\vee\vee} \ar[r]^-{\tau_{\wedge^{g}V^{\vee\vee},\wedge^{g}V^{\vee\vee}}} & \wedge^{g-1}V\otimes\wedge^{g}V^{\vee}\otimes\wedge^{g}V^{\vee}\otimes\wedge^{g}V^{\vee\vee}\otimes\wedge^{g}V^{\vee\vee}\text{.} }
\end{equation*}%
Since $\wedge ^{g}V^{\vee \vee }$ is invertible, it follows from \cite[7.2
Lemme]{De}$\ $that we have $\tau _{\wedge ^{g}V^{\vee \vee },\wedge
^{g}V^{\vee \vee }}=r_{\wedge ^{g}V^{\vee \vee }}=r_{\wedge ^{g}V}$ in the
above diagram, implying that $\tau _{\left( 35\right) }=r_{\wedge ^{g}V}$ as
well. Hence $\left( \text{\ref{Dirac Alternating L2 F6}}\right) $ becomes%
\begin{equation}
\rho \cdot \tau _{\left( 345\right) }\circ \left( D_{1,g}\otimes \varphi
_{i,g-i}^{13}\right) \circ \tau _{\wedge ^{i}V^{\vee }\otimes \wedge
^{g}V^{\vee \vee }\otimes \wedge ^{g-i}V^{\vee }\otimes \wedge ^{g}V^{\vee
\vee },V^{\vee }}=r_{\wedge ^{g}V}\cdot a+b\text{.}
\label{Dirac Alternating L2 F7}
\end{equation}

We can now compute:%
\begin{eqnarray}
&&\left( -1\right) ^{i\left( g-i\right) }\left( -1\right) ^{g-i}r_{\wedge
^{g}V}\mu _{i,g}i\cdot \psi _{g-i,i-1}^{V}\circ \left( 1_{\wedge
^{i}V}\otimes \overline{\psi }_{i,1}^{V}\right)   \notag \\
&&+\left( -1\right) ^{i\left( g-i\right) }\left( -1\right) ^{i\left(
g-i-1\right) }\mu _{g-i,g}\left( g-i\right) \cdot \psi _{i,g-i-1}^{V}\circ
\left( 1_{\wedge ^{i}V}\otimes \overline{\psi }_{g-i,1}^{V}\right)   \notag
\\
&&\text{ \ \ \ \ }\circ \left( \tau _{\wedge ^{g-i}V,\wedge ^{i}V}\otimes
1_{V^{\vee }}\right) \text{ (by }\left( \text{\ref{Dirac Alternating L2 F3}}%
\right) \text{ and }\left( \text{\ref{Dirac Alternating L2 F4}}\right) \text{%
)}  \notag \\
&&\text{ }=D^{g-1,g}\circ \left( 1_{\wedge ^{g-1}V}\otimes ev_{V^{\vee
},a}^{g,\tau }\otimes ev_{V^{\vee },a}^{g,\tau }\right) \circ r_{\wedge
^{g}V}\cdot a\circ \left( D^{g-i,g}\otimes D^{i,g}\otimes 1_{V^{\vee
}}\right)   \notag \\
&&\text{ \ \ \ \ }+D^{g-1,g}\circ \left( 1_{\wedge ^{g-1}V}\otimes
ev_{V^{\vee },a}^{g,\tau }\otimes ev_{V^{\vee },a}^{g,\tau }\right) \circ
b\circ \left( D^{g-i,g}\otimes D^{i,g}\otimes 1_{V^{\vee }}\right)   \notag
\\
&&\text{ }=D^{g-1,g}\circ \left( 1_{\wedge ^{g-1}V}\otimes ev_{V^{\vee
},a}^{g,\tau }\otimes ev_{V^{\vee },a}^{g,\tau }\right) \circ \left(
r_{\wedge ^{g}V}\cdot a+b\right) \circ \left( D^{g-i,g}\otimes
D^{i,g}\otimes 1_{V^{\vee }}\right) \text{ (by }\left( \text{\ref{Dirac
Alternating L2 F7}}\right) \text{)}  \notag \\
&&\text{ }=\rho \cdot D^{g-1,g}\circ \left( 1_{\wedge ^{g-1}V}\otimes
ev_{V^{\vee },a}^{g,\tau }\otimes ev_{V^{\vee },a}^{g,\tau }\right) \circ
\tau _{\left( 345\right) }\circ \left( D_{1,g}\otimes \varphi
_{i,g-i}^{13}\right)   \notag \\
&&\text{ \ \ \ \ }\circ \tau _{\wedge ^{i}V^{\vee }\otimes \wedge
^{g}V^{\vee \vee }\otimes \wedge ^{g-i}V^{\vee }\otimes \wedge ^{g}V^{\vee
\vee },V^{\vee }}\circ \left( D^{g-i,g}\otimes D^{i,g}\otimes 1_{V^{\vee
}}\right) \text{ (by a formal computation)}  \notag \\
&&\text{ }=\rho \cdot D^{g-1,g}\circ \left( 1_{\wedge ^{g-1}V}\otimes
ev_{V^{\vee },a}^{g,\tau }\otimes ev_{V^{\vee },a}^{g,\tau }\right) \circ
\left( D_{1,g}\otimes 1_{\wedge ^{g}V^{\vee \vee }\otimes \wedge ^{g}V^{\vee
}\otimes \wedge ^{g}V^{\vee \vee }}\right) \circ \tau _{\left( 234\right) } 
\notag \\
&&\text{ \ \ \ \ }\circ \left( 1_{V^{\vee }}\otimes \varphi
_{i,g-i}^{13}\right) \circ \left( 1_{V^{\vee }}\otimes D^{g-i,g}\otimes
D^{i,g}\right) \circ \tau _{\wedge ^{g-i}V\otimes \wedge ^{i}V,V^{\vee }}%
\text{.}  \label{Dirac Alternating L2 F8}
\end{eqnarray}%
We remark that we have, by definition, $r_{\wedge ^{g}V^{\vee }}=ev_{V^{\vee
},a}^{g}\circ \tau _{\wedge ^{g}V^{\vee },\wedge ^{g}V^{\vee \vee }}\circ
C_{\wedge ^{g}V^{\vee }}$ and, since $\wedge ^{g}V^{\vee }$ is invertible, $%
r_{\wedge ^{g}V}=r_{\wedge ^{g}V^{\vee }}=r_{\wedge ^{g}V^{\vee }}^{-1}$ and
we deduce $\left( ev_{V^{\vee },a}^{g}\right) ^{-1}=r_{\wedge ^{g}V}\cdot
\tau _{\wedge ^{g}V^{\vee },\wedge ^{g}V^{\vee \vee }}\circ C_{\wedge
^{g}V^{\vee }}$. This gives the first of the subsequent equalities, while
the second follows from a standard property of the Casimir element:%
\begin{eqnarray}
&&\left( 1_{\wedge ^{g}V^{\vee \vee }}\otimes ev_{V^{\vee },a}^{g,\tau
}\right) \circ \left( \left( ev_{V^{\vee },a}^{g}\right) ^{-1}\otimes
1_{\wedge ^{g}V^{\vee \vee }}\right)   \notag \\
&&\text{ \ \ \ }=r_{\wedge ^{g}V}\cdot \left( 1_{\wedge ^{g}V^{\vee \vee
}}\otimes ev_{V^{\vee },a}^{g,\tau }\right) \circ \left( \tau _{\wedge
^{g}V^{\vee },\wedge ^{g}V^{\vee \vee }}\otimes 1_{\wedge ^{g}V^{\vee \vee
}}\right) \circ \left( C_{\wedge ^{g}V^{\vee }}\otimes 1_{\wedge ^{g}V^{\vee
\vee }}\right)   \notag \\
&&\text{ \ \ \ }=r_{\wedge ^{g}V}\cdot 1_{\wedge ^{g}V^{\vee \vee }}\text{.}
\label{Dirac Alternating L2 F9}
\end{eqnarray}%
Thanks to Theorem \ref{Alternating algebras T} $\left( 1\right) $, we know
that $\left( 1_{V^{\vee }}\otimes ev_{V^{\vee }}^{g}\right) \circ \left(
D^{g-1,g}\otimes 1_{\wedge ^{g}V^{\vee }}\right) \circ D_{1,g}=\left(
-1\right) ^{\left( g-1\right) }\mu _{g-1,g}$ with $\mu _{g-1,g}:=\binom{g}{%
g-1}^{-1}\binom{r-1}{g-1}$. Employing this relation in the second of the
subsequent equalities, we find%
\begin{eqnarray}
&&D^{g-1,g}\circ \left( 1_{\wedge ^{g-1}V}\otimes ev_{V^{\vee },a}^{g,\tau
}\otimes ev_{V^{\vee },a}^{g,\tau }\right) \circ \left( D_{1,g}\otimes
1_{\wedge ^{g}V^{\vee \vee }\otimes \wedge ^{g}V^{\vee }\otimes \wedge
^{g}V^{\vee \vee }}\right)   \notag \\
&&\text{ }=\left( 1_{V^{\vee }\otimes \wedge ^{g}V^{\vee \vee }}\otimes
ev_{V^{\vee },a}^{g,\tau }\otimes ev_{V^{\vee },a}^{g,\tau }\right) \circ
\left( D^{g-1,g}\otimes 1_{\wedge ^{g}V^{\vee }\otimes \wedge ^{g}V^{\vee
\vee }\otimes \wedge ^{g}V^{\vee }\otimes \wedge ^{g}V^{\vee \vee }}\right) 
\notag \\
&&\text{ \ \ \ \ }\circ \left( D_{1,g}\otimes 1_{\wedge ^{g}V^{\vee \vee
}\otimes \wedge ^{g}V^{\vee }\otimes \wedge ^{g}V^{\vee \vee }}\right)  
\notag \\
&&\text{ }=\left( -1\right) ^{\left( g-1\right) }\mu _{g-1,g}\cdot \left(
1_{V^{\vee }\otimes \wedge ^{g}V^{\vee \vee }}\otimes ev_{V^{\vee
},a}^{g,\tau }\otimes ev_{V^{\vee },a}^{g,\tau }\right) \circ \left(
1_{V^{\vee }}\otimes \left( ev_{V^{\vee }}^{g}\right) ^{-1}\otimes 1_{\wedge
^{g}V^{\vee \vee }\otimes \wedge ^{g}V^{\vee }\otimes \wedge ^{g}V^{\vee
\vee }}\right)   \notag \\
&&\text{ }=\left( -1\right) ^{\left( g-1\right) }\mu _{g-1,g}\cdot \left(
1_{V^{\vee }\otimes \wedge ^{g}V^{\vee \vee }}\otimes ev_{V^{\vee
},a}^{g,\tau }\right) \circ \left( 1_{V^{\vee }}\otimes \left( ev_{V^{\vee
},a}^{g}\right) ^{-1}\otimes 1_{\wedge ^{g}V^{\vee \vee }}\otimes
ev_{V^{\vee },a}^{g,\tau }\right)   \notag \\
&&\text{ }=\left( -1\right) ^{\left( g-1\right) }\mu _{g-1,g}\cdot \left(
1_{V^{\vee }\otimes \wedge ^{g}V^{\vee \vee }}\otimes ev_{V^{\vee
},a}^{g,\tau }\right) \circ \left( 1_{V^{\vee }}\otimes \left( ev_{V^{\vee
},a}^{g}\right) ^{-1}\otimes 1_{\wedge ^{g}V^{\vee \vee }}\right)   \notag \\
&&\text{ \ \ \ \ }\circ \left( 1_{V^{\vee }\otimes \wedge ^{g}V^{\vee \vee
}}\otimes ev_{V^{\vee },a}^{g,\tau }\right) \text{ (by }\left( \text{\ref%
{Dirac Alternating L2 F9}}\right) \text{)}  \notag \\
&&\text{ }=\left( -1\right) ^{\left( g-1\right) }\mu _{g-1,g}r_{\wedge
^{g}V}\cdot 1_{V^{\vee }\otimes \wedge ^{g}V^{\vee \vee }}\otimes
ev_{V^{\vee },a}^{g,\tau }\text{.}  \label{Dirac Alternating L2 F10}
\end{eqnarray}%
We also have, thanks to the relation $\varphi _{i,g-i}^{13\rightarrow \wedge
^{g}V^{\vee \vee }}\circ \left( D^{g-i,g}\otimes D^{i,g}\right) =\mu
_{i,g}\cdot i_{\wedge ^{g}V}\circ \varphi _{g-i,i}$ with $\mu _{i,g}:=\binom{%
g}{i}^{-1}\binom{r+i-g}{i}$\ arising from Theorem \ref{Alternating algebras
T} $\left( 2\right) $:%
\begin{eqnarray}
&&\left( 1_{V^{\vee }\otimes \wedge ^{g}V^{\vee \vee }}\otimes ev_{V^{\vee
},a}^{g,\tau }\right) \circ \tau _{\left( 234\right) }\circ \left(
1_{V^{\vee }}\otimes \varphi _{i,g-i}^{13}\right) \circ \left( 1_{V^{\vee
}}\otimes D^{g-i,g}\otimes D^{i,g}\right)   \notag \\
&&\text{ }=\left( 1_{V^{\vee }}\otimes ev_{V^{\vee },a}^{g,\tau }\otimes
1_{\wedge ^{g}V^{\vee \vee }}\right) \circ \left( 1_{V^{\vee }}\otimes
\varphi _{i,g-i}^{13}\right) \circ \left( 1_{V^{\vee }}\otimes
D^{g-i,g}\otimes D^{i,g}\right)   \notag \\
&&\text{ }=\left( 1_{V^{\vee }}\otimes \varphi _{i,g-i}^{13\rightarrow
\wedge ^{g}V^{\vee \vee }}\right) \circ \left( 1_{V^{\vee }}\otimes
D^{g-i,g}\otimes D^{i,g}\right) =\mu _{i,g}\cdot \left( 1_{V^{\vee }}\otimes
i_{\wedge ^{g}V}\right) \circ \left( 1_{V^{\vee }}\otimes \varphi
_{g-i,i}\right) \text{.}  \label{Dirac Alternating L2 F11}
\end{eqnarray}%
Inserting $\left( \text{\ref{Dirac Alternating L2 F10}}\right) $ and $\left( 
\text{\ref{Dirac Alternating L2 F11}}\right) $ in $\left( \text{\ref{Dirac
Alternating L2 F8}}\right) $ gives the claim after a small computation.
\end{proof}

\subsection{\label{Subsection LD Alternating}Laplace and Dirac operators}

We now specialize the above discussion to the case $g=2i$, i.e. $i=g-i$, and
we simply write $L$ for the invertible object $\wedge ^{g}V^{\vee \vee }$
and set $L^{-1}:=\wedge ^{g}V^{\vee }$. We write \textrm{Alt}$^{n}\left(
M\right) :=\wedge ^{n}M$ and \textrm{Sym}$^{n}\left( M\right) :=\vee ^{n}M$
when $M$ is an alternating power of $V$. Attached to the multiplication map $%
\wedge ^{i}V\otimes \wedge ^{i}V\overset{\varphi _{i,i}}{\rightarrow }\wedge
^{g}V\overset{i_{\wedge ^{g}V}}{\rightarrow }L$ there are the Laplace
operators%
\begin{align*}
\Delta _{i_{\wedge ^{g}V}\circ \varphi _{i,i},a}^{n} &:\mathrm{Alt}%
^{n}\left( \wedge ^{i}V\right) \rightarrow \mathrm{Alt}^{n-2}\left( \wedge
^{i}V\right) \otimes L\text{,} &
\Delta _{i_{\wedge ^{g}V}\circ \varphi _{i,i},s}^{n} &:\mathrm{Sym}%
^{n}\left( \wedge ^{i}V\right) \rightarrow \mathrm{Sym}^{n-2}\left( \wedge
^{i}V\right) \otimes L
\end{align*}%
and, since $\varphi _{i,i}\circ \tau _{\wedge ^{i}V,\wedge ^{i}V}=\left(
-1\right) ^{i}\varphi _{i,i}$, by Lemma \ref{Dirac L1}\ we have $\Delta
_{i_{\wedge ^{g}V}\circ \varphi _{i,i},s}^{n}=0$ (resp. $\Delta _{i_{\wedge
^{g}V}\circ \varphi _{i,i},a}^{n}=0$) when $i$ is odd (resp. $i$ is even).
Hence, we set $\Delta ^{n}:=\Delta _{i_{\wedge ^{g}V}\circ \varphi
_{i,i},a}^{n}$ (resp. $\Delta ^{n}:=\Delta _{i_{\wedge ^{g}V}\circ \varphi
_{i,i},s}^{n}$) when $i$ is odd (resp. even).

Looking at the pairings defined before Lemma \ref{Dirac Alternating L1}, we
note that we have $\psi _{i,1}^{V}=\psi _{g-i,1}^{V}$ and $\overline{\psi }%
_{g-i,g-i-1}^{V}=\overline{\psi }_{i,i-1}^{V}$, while looking at the
pairings defined before \ref{Dirac Alternating L2}, we remark the equalities 
$\psi _{i,g-i-1}^{V}=\psi _{g-i,i-1}^{V}$ and $\overline{\psi }_{g-i,1}^{V}=%
\overline{\psi }_{i,1}^{V}$.

\bigskip

Suppose first that $i$ is odd. Then we define the following Dirac operators,
for every integer $n\geq 1$:%
\begin{equation*}
\begin{array}{llll}
\partial _{1}^{n}:=\partial _{\psi _{g-i,1}^{V},a}^{n}: & \mathrm{Alt}%
^{n}\left( \wedge ^{i}V\right) \otimes V & \rightarrow & \mathrm{Alt}%
^{n-1}\left( \wedge ^{i}V\right) \otimes \wedge ^{i-1}V^{\vee }\otimes L%
\text{,} \\ 
\overline{\partial }_{i-1}^{n}:=\partial _{\overline{\psi }%
_{i,i-1}^{V},a}^{n}: & \mathrm{Alt}^{n}\left( \wedge ^{i}V\right) \otimes
\wedge ^{i-1}V^{\vee } & \rightarrow & \mathrm{Alt}^{n-1}\left( \wedge
^{i}V\right) \otimes V\text{,} \\ 
\partial _{i-1}^{n}:=\partial _{\psi _{g-i,i-1}^{V},a}^{n}: & \mathrm{Alt}%
^{n}\left( \wedge ^{i}V\right) \otimes \wedge ^{i-1}V & \rightarrow & 
\mathrm{Alt}^{n-1}\left( \wedge ^{i}V\right) \otimes V^{\vee }\otimes L\text{%
,} \\ 
\overline{\partial }_{1}^{n}:=\partial _{\overline{\psi }_{i,1}^{V},a}^{n}:
& \mathrm{Alt}^{n}\left( \wedge ^{i}V\right) \otimes V^{\vee } & \rightarrow
& \mathrm{Alt}^{n-1}\left( \wedge ^{i}V\right) \otimes \wedge ^{i-1}V\text{.}%
\end{array}%
\end{equation*}

\begin{theorem}
\label{Dirac Alternating T1}Suppose that $i$ is odd, so that $\Delta ^{n}:%
\mathrm{Alt}^{n}\left( \wedge ^{i}V\right) \rightarrow \mathrm{Alt}%
^{n-2}\left( \wedge ^{i}V\right) \otimes L$ and set%
\begin{eqnarray*}
\rho ^{i}:=\left( -1\right) ^{i+1}r_{L}\binom{g}{g-1}^{-1}\binom{g}{i}^{-1}%
\binom{r-1}{g-1}\binom{r-i}{i}\frac{g}{i} =r_{L}\binom{g}{g-1}^{-1}\binom{g}{i}^{-1}\binom{r-1}{g-1}\binom{%
r-i}{i}\frac{g}{i}\text{.}
\end{eqnarray*}

\begin{itemize}
\item[$\left( 1\right) $] The following diagram is commutative:%
\begin{equation*}
\xymatrix{ \mathrm{Sym}^{n}\left(\wedge^{i}V\right)\otimes V \ar[r]^-{\partial_{1}^{n}} \ar[d]_{\Delta^{n}\otimes1_{V}} & \mathrm{Sym}^{n-1}\left(\wedge^{i}V\right)\otimes\wedge^{i-1}V^{\vee}\otimes L \ar[d]^{\overline{\partial}_{i-1}^{n-1}\otimes1_{L}} \\ \mathrm{Sym}^{n-2}\left(\wedge^{i}V\right)\otimes L\otimes V \ar[r]_-{\frac{\rho^{i}}{2}\cdot1_{\mathrm{Sym}^{n-2}\left(\wedge^{i}V\right)}\otimes\tau_{L,V}} & \mathrm{Sym}^{n-2}\left(\wedge^{i}V\right)\otimes V\otimes L\text{.} }
\end{equation*}

\item[$\left( 2\right) $] When $r_{L}=1$, the first of the following
diagrams is commutative and it becomes equivalent to the second diagram when
we further assume that $\binom{r-i}{i}\in End\left( \mathbb{I}\right) $ is a
non-zero divisor:%
\begin{equation*}
\xymatrix{ \mathrm{Sym}^{n}\left(\wedge^{i}V\right)\otimes V^{\vee} \ar[r]^-{\overline{\partial}_{1}^{n}} \ar[d]|{\Delta^{n}\otimes1_{V^{\vee}}} & \mathrm{Sym}^{n-1}\left(\wedge^{i}V\right)\otimes\wedge^{i-1}V \ar[d]|{\binom{r-i}{i}\cdot\partial_{i-1}^{n-1}} & \mathrm{Sym}^{n}\left(\wedge^{i}V\right)\otimes V^{\vee} \ar[r]^-{\overline{\partial}_{1}^{n}} \ar[d]|{\Delta^{n}\otimes1_{V^{\vee}}} & \mathrm{Sym}^{n-1}\left(\wedge^{i}V\right)\otimes\wedge^{i-1}V \ar[d]|{\partial_{i-1}^{n-1}} \\ \mathrm{Sym}^{n-2}\left(\wedge^{i}V\right)\otimes L\otimes V^{\vee} \ar[r]_-{\binom{r-i}{i}\frac{\rho^{i}}{2}\cdot1_{\mathrm{Sym}^{n-2}\left(\wedge^{i}V\right)}\otimes\tau_{L,V^{\vee}}} & \mathrm{Sym}^{n-2}\left(\wedge^{i}V\right)\otimes V^{\vee}\otimes L\text{,} & \mathrm{Sym}^{n-2}\left(\wedge^{i}V\right)\otimes L\otimes V^{\vee} \ar[r]_-{\frac{\rho^{i}}{2}\cdot1_{\mathrm{Sym}^{n-2}\left(\wedge^{i}V\right)}\otimes\tau_{L,V^{\vee}}} & \mathrm{Sym}^{n-2}\left(\wedge^{i}V\right)\otimes V^{\vee}\otimes L\text{.}}
\end{equation*}

\item[$\left( 3\right) $] Suppose that $L\simeq \mathbb{L}^{\otimes 2}$ for
some invertible object $\mathbb{L}$, that $r_{\wedge ^{i}V}<0$ (see
definition \ref{Dirac definition positive rank})\ and that $V$ has
alternating rank $g$. Then there are morphisms%
\begin{eqnarray*}
s_{\Delta }^{n-2} &:&\mathrm{Alt}^{n-2}\left( \wedge ^{i}V\right) \otimes
L\rightarrow \mathrm{Alt}^{n}\left( \wedge ^{i}V\right) \text{ for }n\geq 2%
\text{,} \\
s_{\overline{\partial }_{i-1}}^{n-1} &:&\mathrm{Alt}^{n-1}\left( \wedge
^{i}V\right) \otimes V\rightarrow \mathrm{Alt}^{n}\left( \wedge ^{i}V\right)
\otimes \wedge ^{i-1}V^{\vee }\text{ for }n\geq 1\text{,} \\
s_{\partial _{i-1}}^{n-1} &:&\mathrm{Alt}^{n-1}\left( \wedge ^{i}V\right)
\otimes V^{\vee }\otimes L\rightarrow \mathrm{Alt}^{n}\left( \wedge
^{i}V\right) \otimes \wedge ^{i-1}V\text{ for }n\geq 1
\end{eqnarray*}%
such that%
\begin{eqnarray*}
\Delta ^{n}\circ s_{\Delta }^{n-2}=1_{\mathrm{Alt}^{n-2}\left( \wedge
^{i}V\right) \otimes L}\text{, }\quad\overline{\partial }_{i-1}^{n}\circ s_{%
\overline{\partial }_{i-1}}^{n-1}=1_{\mathrm{Alt}^{n-1}\left( \wedge
^{i}V\right) \otimes V} \\
\text{and }\quad\partial _{i-1}^{n}\circ s_{\partial _{i-1}}^{n-1}=1_{\mathrm{%
Alt}^{n-1}\left( \wedge ^{i}V\right) \otimes V^{\vee }\otimes L}.
\end{eqnarray*}%
In particular, the following objects exist:%
\begin{eqnarray*}
\ker \left( \Delta ^{n}\right) \subset \mathrm{Alt}^{n}\left( \wedge
^{i}V\right) \text{,}\quad
\ker \left( \overline{\partial }_{i-1}^{n}\right) \subset \mathrm{Alt}%
^{n}\left( \wedge ^{i}V\right) \otimes \wedge ^{i-1}V^{\vee }\text{,}\quad
\ker \left( \partial _{i-1}^{n}\right) \subset \mathrm{Alt}^{n}\left( \wedge
^{i}V\right) \otimes \wedge ^{i-1}V\text{.}
\end{eqnarray*}
\end{itemize}
\end{theorem}

\begin{proof}
$\left( 1\text{-}2\right) $ Looking at the quantities $\nu _{V}^{g-i,1}$ and 
$\nu _{V}^{i,1}$ (resp. $\nu _{V^{\vee }}^{i,1}$ and $\nu _{V^{\vee
}}^{g-i,1}$) from Lemma \ref{Dirac Alternating L1} (resp. Lemma \ref{Dirac
Alternating L2}) when $i=g-i$, we see that $\nu _{V}^{g-i,1}=\left(
-1\right) ^{i}\cdot \nu _{V}^{i,1}$ (resp. $\nu _{V^{\vee }}^{i,1}=\left(
-1\right) ^{i}r_{\wedge ^{g}V}\cdot \nu _{V^{\vee }}^{g-i,1}$). Since $i$ is
odd, it follows that $\nu _{V}^{g-i,1}=-\nu _{V}^{i,1}$ (resp. $\nu
_{V^{\vee }}^{i,1}=-r_{\wedge ^{g}V}\cdot \nu _{V^{\vee }}^{g-i,1}$). We
have that $i_{\wedge ^{g}V}\circ \varphi _{i,i}$ is alternating, so that we
may apply Lemma \ref{Dirac L2} to deduce the claimed commutativity in $%
\left( 1\right) $ (resp. the first commutative diagram in $\left( 2\right) $
when $r_{L}=r_{\wedge ^{g}V}=1$): we have indeed $\nu _{V}^{g-i,1}\in 
\mathbb{Q}^{\times }$ and%
\begin{equation*}
\rho _{V}^{i,g-i}/\nu _{V}^{g-i,1}=\rho ^{i}\text{ (resp. }\rho _{V^{\vee
}}^{g-i,i}=-r_{\wedge ^{g}V}\binom{g}{i}^{-1}\binom{r-i}{i}i\rho ^{i}\text{
and }\nu _{V^{\vee }}^{i,1}=r_{\wedge ^{g}V}\binom{g}{i}^{-1}\binom{r-i}{i}i%
\text{)}
\end{equation*}%
and the commutativity of $\left( 2\right) $ is deduced simplifying by $%
r_{\wedge ^{g}V}\binom{g}{i}^{-1}i$. If $\binom{r-i}{i}\in End\left( \mathbb{%
I}\right) $ is a non-zero divisor we may further simplify to get the second
commutative diagram in $\left( 2\right) $.

$\left( 3\right) $ Indeed $L\simeq \mathbb{L}^{\otimes 2}$ implies $r_{L}=1$
and, since $V$ has alternating rank $g$, by Corollary \ref{Alternating
algebras CT} $i_{\wedge ^{g}V}\circ \varphi _{i,i}$ is a perfect alternating
pairing. Since $r_{\wedge ^{i}V}<0$, Lemma \ref{Dirac P3} gives the
existence of $s_{\Delta }^{n-2}$ and $\ker \left( \Delta ^{n}\right) $. We
also remark that, since $\binom{r-i}{g-i}=\binom{r-i}{i}\in End\left( 
\mathbb{I}\right) $ and $\binom{r-1}{g-1}\in End\left( \mathbb{I}\right) $\
are invertible (once again because $V$ has alternating rank $g$), it follows
that $\pm \frac{\rho ^{i}}{2}$ is invertible, that $\alpha :=\frac{\rho ^{i}%
}{2}\cdot \left( 1_{\mathrm{Alt}^{n-2}\left( \wedge ^{i}V\right) }\otimes
\tau _{L,V}\right) $ is an isomorphism and, hence, that $f:=\alpha \circ
\left( \Delta ^{n}\otimes 1_{V}\right) $ has a section $s:=\left( s_{\Delta
}^{n-2}\otimes 1_{V}\right) \circ \alpha ^{-1}$ such that%
\begin{equation*}
f\circ s=\alpha \circ \left( \Delta ^{n}\otimes 1_{V}\right) \circ \left(
s_{\Delta }^{n-2}\otimes 1_{V}\right) \circ \alpha ^{-1}=\alpha \circ \alpha
^{-1}=1_{\mathrm{Alt}^{n-2}\left( \wedge ^{i}V\right) \otimes V^{\vee
}\otimes L}\text{.}
\end{equation*}%
Similarly $f:=\left( -\frac{\rho ^{i}}{2}\cdot 1_{\mathrm{Alt}^{n-2}\left(
\wedge ^{i}V\right) }\otimes \tau _{L,V^{\vee }}\right) \circ \left( \Delta
^{n}\otimes 1_{V^{\vee }}\right) $ has a section. We can now apply the
following simple remark to the commutative diagram in $\left( 1\right) $
(resp. the second commutative diagram in $\left( 2\right) $). Suppose that
we are given%
\begin{equation*}
f:X\overset{f_{1}}{\rightarrow }Y\overset{f_{2}}{\rightarrow }Z
\end{equation*}%
and that $s:Z\rightarrow X$ is a morphism such that $f\circ s=1_{Z}$. Then,
setting $s_{2}:=f_{1}\circ s$, we see that%
\begin{equation*}
f_{2}\circ s_{2}=f_{2}\circ f_{1}\circ s=f\circ s=1_{Z}\text{,}
\end{equation*}%
implying that $f_{2}$ has a section. But then there is an associated
idempotent $e_{2}:=s_{2}\circ f_{2}$ and $\ker \left( f_{2}\right) =\ker
\left( e_{2}\right) $ exists because $V$ is pseudo-abelian. This gives the
existence of a section of $\overline{\partial }_{i-1}^{n}\otimes 1_{L}$,
hence of $\overline{\partial }_{i-1}^{n}$ and $\ker \left( \overline{%
\partial }_{i-1}^{n}\right) $\ because $L$ is invertible, and of $%
s_{\partial _{i-1}}^{n-1}$ and $\ker \left( \partial _{i-1}\right) $.
\end{proof}

\bigskip

Suppose now that $i$ is even. Then we define the following Dirac operators,
for every integer $n\geq 1$:%
\begin{equation*}
\begin{array}{llll}
\partial _{1}^{n}:=\partial _{\psi _{g-i,1}^{V},s}^{n}: & \mathrm{Sym}%
^{n}\left( \wedge ^{i}V\right) \otimes V & \rightarrow & \mathrm{Sym}%
^{n-1}\left( \wedge ^{i}V\right) \otimes \wedge ^{i-1}V^{\vee }\otimes L%
\text{,} \\ 
\overline{\partial }_{i-1}^{n}:=\partial _{\overline{\psi }%
_{i,i-1}^{V},s}^{n}: & \mathrm{Sym}^{n}\left( \wedge ^{i}V\right) \otimes
\wedge ^{i-1}V^{\vee } & \rightarrow & \mathrm{Sym}^{n-1}\left( \wedge
^{i}V\right) \otimes V\text{,} \\ 
\partial _{i-1}^{n}:=\partial _{\psi _{g-i,i-1}^{V},s}^{n}: & \mathrm{Sym}%
^{n}\left( \wedge ^{i}V\right) \otimes \wedge ^{i-1}V & \rightarrow & 
\mathrm{Sym}^{n-1}\left( \wedge ^{i}V\right) \otimes V^{\vee }\otimes L\text{%
,} \\ 
\overline{\partial }_{1}^{n}:=\partial _{\overline{\psi }_{i,1}^{V},s}^{n}:
& \mathrm{Sym}^{n}\left( \wedge ^{i}V\right) \otimes V^{\vee } & \rightarrow
& \mathrm{Sym}^{n-1}\left( \wedge ^{i}V\right) \otimes \wedge ^{i-1}V\text{.}%
\end{array}%
\end{equation*}

\begin{theorem}
\label{Dirac Alternating T2}Suppose that $i$ is even, so that $\Delta ^{n}:\mathrm{Sym}^{n}\left( \wedge
^{i}V\right) \rightarrow \mathrm{Sym}^{n-2}\left( \wedge ^{i}V\right)
\otimes L$ and set%
\begin{eqnarray*}
&&\rho ^{i}:=\left( -1\right) ^{i+1}r_{L}\binom{g}{g-1}^{-1}\binom{g}{i}^{-1}%
\binom{r-1}{g-1}\binom{r-i}{i}\frac{g}{i} \\
&&\text{ }=-r_{L}\binom{g}{g-1}^{-1}\binom{g}{i}^{-1}\binom{r-1}{g-1}\binom{%
r-i}{i}\frac{g}{i}\text{.}
\end{eqnarray*}

\begin{itemize}
\item[$\left( 1\right) $] The following diagram is commutative:%
\begin{equation*}
\xymatrix{ \mathrm{Sym}^{n}\left(\wedge^{i}V\right)\otimes V \ar[r]^-{\partial_{1}^{n}} \ar[d]_{\Delta^{n}\otimes1_{V}} & \mathrm{Sym}^{n-1}\left(\wedge^{i}V\right)\otimes\wedge^{i-1}V^{\vee}\otimes L \ar[d]^{\overline{\partial}_{i-1}^{n}\otimes1_{L}} \\ \mathrm{Sym}^{n-2}\left(\wedge^{i}V\right)\otimes L\otimes V \ar[r]_-{\frac{\rho^{i}}{2}\cdot1_{\mathrm{Sym}^{n-2}\left(\wedge^{i}V\right)}\otimes\tau_{L,V}} & \mathrm{Sym}^{n-2}\left(\wedge^{i}V\right)\otimes V\otimes L\text{.} }
\end{equation*}

\item[$\left( 2\right) $] When $r_{L}=1$, the first of the following
diagrams is commutative and it becomes equivalent to the second diagram when
we further assume that $\binom{r-i}{i}\in End\left( \mathbb{I}\right) $ is a
non-zero divisor:%
\begin{equation*}
\xymatrix{ \mathrm{Sym}^{n}\left(\wedge^{i}V\right)\otimes V^{\vee} \ar[r]^-{\overline{\partial}_{1}^{n}} \ar[d]|{\Delta^{n}\otimes1_{V^{\vee}}} & \mathrm{Sym}^{n-1}\left(\wedge^{i}V\right)\otimes\wedge^{i-1}V \ar[d]|{\binom{r-i}{i}\cdot\partial_{i-1}^{n}} & \mathrm{Sym}^{n}\left(\wedge^{i}V\right)\otimes V^{\vee} \ar[r]^-{\overline{\partial}_{1}^{n}} \ar[d]|{\Delta^{n}\otimes1_{V^{\vee}}} & \mathrm{Sym}^{n-1}\left(\wedge^{i}V\right)\otimes\wedge^{i-1}V \ar[d]|{\partial_{i-1}^{n}} \\ \mathrm{Sym}^{n-2}\left(\wedge^{i}V\right)\otimes L\otimes V^{\vee} \ar[r]_-{\binom{r-i}{i}\frac{\rho^{i}}{2}\cdot1_{\mathrm{Sym}^{n-2}\left(\wedge^{i}V\right)}\otimes\tau_{L,V^{\vee}}} & \mathrm{Sym}^{n-2}\left(\wedge^{i}V\right)\otimes V^{\vee}\otimes L\text{,} & \mathrm{Sym}^{n-2}\left(\wedge^{i}V\right)\otimes L\otimes V^{\vee} \ar[r]_-{\frac{\rho^{i}}{2}\cdot1_{\mathrm{Sym}^{n-2}\left(\wedge^{i}V\right)}\otimes\tau_{L,V^{\vee}}} & \mathrm{Sym}^{n-2}\left(\wedge^{i}V\right)\otimes V^{\vee}\otimes L\text{.}}
\end{equation*}

\item[$\left( 3\right) $] Suppose that $L\simeq \mathbb{L}^{\otimes 2}$ for
some invertible object $\mathbb{L}$, that $r_{\wedge ^{i}V}>0$ (see
definition \ref{Dirac definition positive rank})\ and that $V$ has
alternating rank $g$. Then there are morphisms%
\begin{eqnarray*}
s_{\Delta }^{n-2} &:&\mathrm{Sym}^{n-2}\left( \wedge ^{i}V\right) \otimes
L\rightarrow \mathrm{Sym}^{n}\left( \wedge ^{i}V\right) \text{ for }n\geq 2%
\text{,} \\
s_{\overline{\partial }_{i-1}}^{n-1} &:&\mathrm{Sym}^{n-1}\left( \wedge
^{i}V\right) \otimes V\rightarrow \mathrm{Sym}^{n}\left( \wedge ^{i}V\right)
\otimes \wedge ^{i-1}V^{\vee }\text{ for }n\geq 1\text{,} \\
s_{\partial _{i-1}}^{n-1} &:&\mathrm{Sym}^{n-1}\left( \wedge ^{i}V\right)
\otimes V^{\vee }\otimes L\rightarrow \mathrm{Sym}^{n}\left( \wedge
^{i}V\right) \otimes \wedge ^{i-1}V\text{ for }n\geq 1
\end{eqnarray*}%
such that%
\begin{eqnarray*}
&&\Delta ^{n}\circ s_{\Delta }^{n-2}=1_{\mathrm{Sym}^{n-2}\left( \wedge
^{i}V\right) \otimes L},\overline{\partial }_{i-1}^{n}\circ s_{\overline{%
\partial }_{i-1}}^{n-1}=1_{\mathrm{Sym}^{n-1}\left( \wedge ^{i}V\right)
\otimes V} \\
&&\text{and }\partial _{i-1}^{n}\circ s_{\partial _{i-1}}^{n-1}=1_{\mathrm{%
Sym}^{n-1}\left( \wedge ^{i}V\right) \otimes V^{\vee }\otimes L}.
\end{eqnarray*}%
In particular, the following objects exist:%
\begin{eqnarray*}
&&\ker \left( \Delta ^{n}\right) \subset \mathrm{Sym}^{n}\left( \wedge
^{i}V\right) \text{,} \\
&&\ker \left( \overline{\partial }_{i-1}^{n}\right) \subset \mathrm{Sym}%
^{n}\left( \wedge ^{i}V\right) \otimes \wedge ^{i-1}V^{\vee }\text{,} \\
&&\ker \left( \partial _{i-1}^{n}\right) \subset \mathrm{Sym}^{n}\left( \wedge
^{i}V\right) \otimes \wedge ^{i-1}V\text{.}
\end{eqnarray*}
\end{itemize}
\end{theorem}

\begin{proof}
$\left( 1\text{-}2\right) $ As in the proof of Theorem \ref{Dirac
Alternating T1} we have $\nu _{V}^{g-i,1}=\left( -1\right) ^{i}\cdot \nu
_{V}^{i,1}$ (resp. $\nu _{V^{\vee }}^{i,1}=\left( -1\right) ^{i}r_{\wedge
^{g}V}\cdot \nu _{V^{\vee }}^{g-i,1}$). Since $i$ is even, it follows that $%
\nu _{V}^{g-i,1}=\nu _{V}^{i,1}$ (resp. $\nu _{V^{\vee }}^{i,1}=r_{\wedge
^{g}V}\cdot \nu _{V^{\vee }}^{g-i,1}$). Then the proof is identical to the
proof of Theorem \ref{Dirac Alternating T1}, noticing that we have once
again $\rho _{V}^{i,g-i}/\nu _{V}^{g-i,1}=\rho ^{i}$ and $\nu _{V^{\vee
}}^{i,1}=r_{\wedge ^{g}V}\binom{g}{i}^{-1}\binom{r-i}{i}i$, but now $\rho
_{V^{\vee }}^{g-i,i}=r_{\wedge ^{g}V}\binom{g}{i}^{-1}\binom{r-i}{i}i\rho
^{i}$, justifying the change of sign in the second commutative diagram of $%
\left( 2\right) $\ with respect to that of Theorem \ref{Dirac Alternating T1}%
.

$\left( 3\right) $ The proof is identical to the proof of Theorem \ref{Dirac
Alternating T1}, noticing that here we need to assume $r_{\wedge ^{i}V}>0$
in order to apply Lemma \ref{Dirac P3} because now $i_{\wedge ^{g}V}\circ
\varphi _{i,i}$ is a perfect symmetric pairing.
\end{proof}

\section{Laplace and Dirac operators for the symmetric algebras}

In this section we assume that we are given an object $V\in \mathcal{C}$ such that $\vee ^{g}V$ is invertible. If $X$ is an object we set $%
r_{X}:=\mathrm{rank}\left( X\right) $, so that $r_{\vee ^{g}V}\in \left\{
\pm 1\right\} $, and we use the shorthand $r:=r_{V}$.

\subsection{\label{Subsection LD Symmetric Lemmas}Preliminary lemmas}

We define%
\begin{equation*}
\psi _{i,1}^{V}:\vee ^{i}V\otimes V\overset{\varphi _{i,1}}{\rightarrow }%
\vee ^{i+1}V\overset{D^{i+1,g}}{\rightarrow }\vee ^{g-i-1}V^{\vee }\otimes
\vee ^{g}V^{\vee \vee }\text{,}
\end{equation*}%
and%
\begin{eqnarray*}
\overline{\psi }_{g-i,g-i-1}^{V} &:&\vee ^{g-i}V\otimes \vee ^{g-i-1}V^{\vee
}\overset{D^{g-i,g}\otimes 1_{\vee ^{g-i-1}V^{\vee }}}{\rightarrow }\vee
^{i}V^{\vee }\otimes \vee ^{g}V^{\vee \vee }\otimes \vee ^{g-i-1}V^{\vee }%
\overset{\varphi _{i,g-i-1}^{13}}{\rightarrow }\vee ^{g-1}V^{\vee }\otimes
\vee ^{g}V^{\vee \vee } \\
&&\overset{D_{g-1,g}\otimes 1_{\vee ^{g}V^{\vee \vee }}}{\rightarrow }%
V\otimes \vee ^{g}V^{\vee }\otimes \vee ^{g}V^{\vee \vee }\overset{%
1_{V}\otimes ev_{V^{\vee },a}^{g,\tau }}{\rightarrow }V\text{.}
\end{eqnarray*}

We may also consider%
\begin{equation*}
\psi _{g-i,1}^{V}:\vee ^{g-i}V\otimes V\overset{\varphi _{g-i,1}}{%
\rightarrow }\vee ^{g-i+1}V\overset{D^{g-i+1,g}}{\rightarrow }\vee
^{i-1}V^{\vee }\otimes \vee ^{g}V^{\vee \vee }\text{,}
\end{equation*}%
and%
\begin{eqnarray*}
\overline{\psi }_{i,i-1}^{V} &:&\vee ^{i}V\otimes \vee ^{i-1}V^{\vee }%
\overset{D^{i,g}\otimes 1_{\vee ^{i-1}V^{\vee }}}{\rightarrow }\vee
^{g-i}V^{\vee }\otimes \vee ^{g}V^{\vee \vee }\otimes \vee ^{i-1}V^{\vee }%
\overset{\varphi _{g-i,i-1}^{13}}{\rightarrow }\vee ^{g-1}V^{\vee }\otimes
\vee ^{g}V^{\vee \vee } \\
&&\overset{D_{g-1,g}\otimes 1_{\vee ^{g}V^{\vee \vee }}}{\rightarrow }%
V\otimes \vee ^{g}V^{\vee }\otimes \vee ^{g}V^{\vee \vee }\overset{%
1_{V}\otimes ev_{V^{\vee },a}^{g,\tau }}{\rightarrow }V\text{.}
\end{eqnarray*}

\begin{lemma}
\label{Dirac Symmetric L1}Setting%
\begin{eqnarray*}
&&\rho _{V}^{i,g-i}:=r_{\vee ^{g}V}\binom{g}{g-1}^{-1}\binom{g}{g-i}^{-1}%
\binom{r+g-1}{g-1}\binom{r+g-1}{g-i}g\text{,} \\
&&\nu _{V}^{g-i,1}:=i\text{ and }\nu _{V}^{i,1}:=g-i
\end{eqnarray*}%
the following diagram is commutative:%
\begin{equation*}
\xymatrix@C=170pt{ \vee^{i}V\otimes\vee^{g-i}V\otimes V \ar[r]^-{\left(1_{\vee^{i}V}\otimes\psi_{g-i,1},\left(1_{\vee^{g-i}V}\otimes\psi_{i,1}\right)\circ\left(\tau_{\vee^{i}V,\vee^{g-i}V}\otimes1_{V}\right)\right)} \ar[d]|{\varphi_{i,g-i}\otimes1_{V}} & \vee^{i}V\otimes\vee^{i-1}V^{\vee}\otimes\vee^{g}V^{\vee\vee}\oplus\vee^{g-i}V\otimes\vee^{g-i-1}V^{\vee}\otimes\vee^{g}V^{\vee\vee} \ar[d]|{\nu_{V}^{g-i,1}\cdot\overline{\psi}_{i,i-1}\otimes1_{\vee^{g}V^{\vee\vee}}\oplus\nu_{V}^{i,1}\cdot\overline{\psi}_{g-i,g-i-1}\otimes1_{\vee^{g}V^{\vee\vee}}} \\ \vee^{g}V\otimes V \ar[r]^-{\rho_{V}^{i,g-i}\cdot\tau_{\vee^{g}V^{\vee\vee},V}\circ\left(i_{\vee^{g}V}\otimes1_{V}\right)} & V\otimes\vee^{g}V^{\vee\vee}\text{.} }
\end{equation*}
\end{lemma}

\begin{proof}
The proof is just a copy of that of Lemma \ref{Dirac Alternating L1},
replacing the use of Theorem \ref{Alternating algebras T} (resp. Proposition %
\ref{Alternating algebras P2}) with Theorem \ref{Symmetric algebras T}
(resp. Proposition \ref{Symmetric algebras P2}).
\end{proof}

\bigskip

We now consider the following morphisms. We have%
\begin{equation*}
\psi _{i,g-i-1}^{V}:\vee ^{i}V\otimes \vee ^{g-i-1}V\overset{\varphi
_{i,g-i-1}}{\rightarrow }\vee ^{g-1}V\overset{D^{g-1,g}}{\rightarrow }%
V^{\vee }\otimes \vee ^{g}V^{\vee \vee }
\end{equation*}%
and%
\begin{eqnarray*}
\overline{\psi }_{g-i,1}^{V} &:&\vee ^{g-i}V\otimes V^{\vee }\overset{%
D^{g-i,g}\otimes 1_{V^{\vee }}}{\rightarrow }\vee ^{i}V^{\vee }\otimes \vee
^{g}V^{\vee \vee }\otimes V^{\vee }\overset{\varphi _{i,1}^{13}}{\rightarrow 
}\vee ^{i+1}V^{\vee }\otimes \vee ^{g}V^{\vee \vee } \\
&&\overset{D_{i+1,g}\otimes 1_{\vee ^{g}V^{\vee \vee }}}{\rightarrow }\vee
^{g-i-1}V\otimes \vee ^{g}V^{\vee }\otimes \vee ^{g}V^{\vee \vee }\overset{%
1_{\vee ^{g-i-1}V}\otimes ev_{V^{\vee },a}^{g,\tau }}{\rightarrow }\vee
^{g-i-1}V\text{.}
\end{eqnarray*}

On the other hand we have%
\begin{equation*}
\psi _{g-i,i-1}^{V}:\vee ^{g-i}V\otimes \vee ^{i-1}V\overset{\varphi
_{g-i,i-1}}{\rightarrow }\vee ^{g-1}V\overset{D^{g-1,g}}{\rightarrow }%
V^{\vee }\otimes \vee ^{g}V^{\vee \vee }
\end{equation*}%
and%
\begin{eqnarray*}
\overline{\psi }_{i,1}^{V} &:&\vee ^{i}V\otimes V^{\vee }\overset{%
D^{i,g}\otimes 1_{V^{\vee }}}{\rightarrow }\vee ^{g-i}V^{\vee }\otimes \vee
^{g}V^{\vee \vee }\otimes V^{\vee }\overset{\varphi _{g-i,1}^{13}}{%
\rightarrow }\vee ^{g-i+1}V^{\vee }\otimes \vee ^{g}V^{\vee \vee } \\
&&\overset{D_{g-i+1,g}\otimes 1_{\vee ^{g}V^{\vee \vee }}}{\rightarrow }\vee
^{i-1}V\otimes \vee ^{g}V^{\vee }\otimes \vee ^{g}V^{\vee \vee }\overset{%
1_{\vee ^{i-1}V}\otimes ev_{V^{\vee },a}^{g,\tau }}{\rightarrow }\vee ^{i-1}V%
\text{.}
\end{eqnarray*}

\begin{lemma}
\label{Dirac Symmetric L2}Setting%
\begin{eqnarray*}
&&\rho _{V^{\vee }}^{g-i,i}:=\binom{g}{g-1}^{-1}\binom{g}{g-i}^{-1}\binom{g}{%
i}^{-1}\binom{r+g-1}{g-1}\binom{r+g-1}{g-i}\binom{r+g-1}{i}g\text{,} \\
&&\nu _{V^{\vee }}^{i,1}:=r_{\vee ^{g}V}\binom{g}{i}^{-1}\binom{r+g-1}{i}i%
\text{ and} \\
&&\nu _{V^{\vee }}^{g-i,1}:=\binom{g}{g-i}^{-1}\binom{r+g-1}{g-i}\left(
g-i\right) \text{,}
\end{eqnarray*}%
the following diagram is commutative:%
\begin{equation*}
\xymatrix@C=180pt{ \vee^{g-i}V\otimes\vee^{i}V\otimes V^{\vee} \ar[r]^-{\left(1_{\vee^{g-i}V}\otimes\overline{\psi}_{i,1},\left(1_{\vee^{i}V}\otimes\overline{\psi}_{g-i,1}\right)\circ\left(\tau_{\vee^{g-i}V,\vee^{i}V}\otimes1_{V^{\vee}}\right)\right)} \ar[d]|{\varphi_{g-i,i}\otimes1_{V^{\vee}}} & \vee^{g-i}V\otimes\vee^{i-1}V\oplus\vee^{i}V\otimes\vee^{g-i-1}V \ar[d]^{\nu_{V^{\vee}}^{i,1}\cdot\psi_{g-i,i-1}\oplus\nu_{V^{\vee}}^{g-i,1}\cdot\psi_{i,g-i-1}} \\ \vee^{g}V\otimes V^{\vee} \ar[r]^-{\rho_{V^{\vee}}^{g-i,i}\cdot\left(1_{V^{\vee}}\otimes i_{\vee^{g}V}\right)\circ\tau_{\vee^{g}V,V^{\vee}}} & V^{\vee}\otimes\vee^{g}V^{\vee\vee}\text{.} }
\end{equation*}
\end{lemma}

\begin{proof}
Again the proof is a copy of that of Lemma \ref{Dirac Alternating L2}.
\end{proof}

\subsection{\label{Subsection LD Symmetric}Laplace and Dirac operators}

We now specialize the above discussion to the case $g=2i$, i.e. $i=g-i$, and
we simply write $L$ for the invertible object $\vee ^{g}V^{\vee \vee }$
and set $L^{-1}:=\vee ^{g}V^{\vee }$. We write \textrm{Alt}$^{n}\left(
M\right) :=\wedge ^{n}M$ and \textrm{Sym}$^{n}\left( M\right) :=\vee ^{n}M$
when $M$ is a symmetric power of $V$. Attached to
the multiplication map $\vee ^{i}V\otimes \vee ^{i}V\overset{\varphi _{i,i}}{%
\rightarrow }\vee ^{g}V\overset{i_{\vee ^{g}V}}{\rightarrow }L$ there are
the Laplace operators%
\begin{eqnarray*}
\Delta _{i_{\vee ^{g}V}\circ \varphi _{i,i},a}^{n} &:&\mathrm{Alt}^{n}\left(
\vee ^{i}V\right) \rightarrow \mathrm{Alt}^{n-2}\left( \vee ^{i}V\right)
\otimes L\text{,} \\
\Delta _{i_{\vee ^{g}V}\circ \varphi _{i,i},s}^{n} &:&\mathrm{Sym}^{n}\left(
\vee ^{i}V\right) \rightarrow \mathrm{Sym}^{n-2}\left( \vee ^{i}V\right)
\otimes L
\end{eqnarray*}%
and, since $\varphi _{i,i}\circ \tau _{\vee ^{i}V,\vee ^{i}V}=\varphi _{i,i}$%
, by Lemma \ref{Dirac L1}\ we have $\Delta _{i_{\vee ^{g}V}\circ \varphi
_{i,i},a}^{n}=0$. Hence we will only consider $\Delta ^{n}:=\Delta _{i_{\vee
^{g}V}\circ \varphi _{i,i},s}^{n}$.

Looking at the pairings defined before Lemma \ref{Dirac Symmetric L1}, we
note that we have $\psi _{i,1}^{V}=\psi _{g-i,1}^{V}$ and $\overline{\psi }%
_{g-i,g-i-1}^{V}=\overline{\psi }_{i,i-1}^{V}$, while looking at the
pairings defined before \ref{Dirac Symmetric L2}, we remark the equalities $%
\psi _{i,g-i-1}^{V}=\psi _{g-i,i-1}^{V}$ and $\overline{\psi }_{g-i,1}^{V}=%
\overline{\psi }_{i,1}^{V}$.

\bigskip

Then we define the following Dirac operators, for every integer $n\geq 1$:%
\begin{equation*}
\begin{array}{llll}
\partial _{1}^{n}:=\partial _{\psi _{g-i,1}^{V},s}^{n}: & \mathrm{Sym}%
^{n}\left( \vee ^{i}V\right) \otimes V & \rightarrow & \mathrm{Sym}%
^{n-1}\left( \vee ^{i}V\right) \otimes \vee ^{i-1}V^{\vee }\otimes L\text{,}
\\ 
\overline{\partial }_{i-1}^{n}:=\partial _{\overline{\psi }%
_{i,i-1}^{V},s}^{n}: & \mathrm{Sym}^{n}\left( \vee ^{i}V\right) \otimes \vee
^{i-1}V^{\vee } & \rightarrow & \mathrm{Sym}^{n-1}\left( \vee ^{i}V\right)
\otimes V\text{,} \\ 
\partial _{i-1}^{n}:=\partial _{\psi _{g-i,i-1}^{V},s}^{n}: & \mathrm{Sym}%
^{n}\left( \vee ^{i}V\right) \otimes \vee ^{i-1}V & \rightarrow & \mathrm{Sym%
}^{n-1}\left( \vee ^{i}V\right) \otimes V^{\vee }\otimes L\text{,} \\ 
\overline{\partial }_{1}^{n}:=\partial _{\overline{\psi }_{i,1}^{V},s}^{n}:
& \mathrm{Sym}^{n}\left( \vee ^{i}V\right) \otimes V^{\vee } & \rightarrow & 
\mathrm{Sym}^{n-1}\left( \vee ^{i}V\right) \otimes \vee ^{i-1}V\text{.}%
\end{array}%
\end{equation*}%
In the same way as we have deduced Theorem \ref{Dirac Alternating T1} from
Lemmas \ref{Dirac Alternating L1} and \ref{Dirac Alternating L2}, the
following result can be deduced from Lemmas \ref{Dirac Symmetric L1} and \ref%
{Dirac Symmetric L2}.

\begin{theorem}
\label{Dirac Symmetric T1}Set%
\begin{equation*}
\rho ^{i}:=r_{\vee ^{g}V}\binom{g}{g-1}^{-1}\binom{g}{i}^{-1}\binom{r+g-1}{%
g-1}\binom{r+g-1}{i}\frac{g}{i}\text{.}
\end{equation*}

\begin{itemize}
\item[$\left( 1\right) $] The following diagram is commutative:%
\begin{equation*}
\xymatrix{ \mathrm{Sym}^{n}\left(\vee^{i}V\right)\otimes V \ar[r]^-{\partial_{1}^{n}} \ar[d]_{\Delta^{n}\otimes1_{V}} & \mathrm{Sym}^{n-1}\left(\vee^{i}V\right)\otimes\vee^{i-1}V^{\vee}\otimes L \ar[d]^{\overline{\partial}_{i-1}^{n-1}\otimes1_{L}} \\ \mathrm{Sym}^{n-2}\left(\vee^{i}V\right)\otimes L\otimes V \ar[r]_-{\frac{\rho^{i}}{2}\cdot1_{\mathrm{Sym}^{n-2}\left(\vee^{i}V\right)}\otimes\tau_{L,V}} & \mathrm{Sym}^{n-2}\left(\vee^{i}V\right)\otimes V\otimes L\text{.} }
\end{equation*}

\item[$\left( 2\right) $] When $r_{L}=1$, the first of the following
diagrams is commutative and it becomes equivalent to the second diagram when
we further assume that $\binom{r+g-1}{i}\in End\left( \mathbb{I}\right) $ is
a non-zero divisor:%
\begin{equation*}
\xymatrix{ \mathrm{Sym}^{n}\left(\vee^{i}V\right)\otimes V^{\vee} \ar[r]^-{\overline{\partial}_{1}^{n}} \ar[d]|{\Delta^{n}\otimes1_{V^{\vee}}} & \mathrm{Sym}^{n-1}\left(\vee^{i}V\right)\otimes\vee^{i-1}V \ar[d]|{\binom{r+g-1}{i}\cdot\partial_{i-1}^{n-1}} & \mathrm{Sym}^{n}\left(\vee^{i}V\right)\otimes V^{\vee} \ar[r]^-{\overline{\partial}_{1}^{n}} \ar[d]|{\Delta^{n}\otimes1_{V^{\vee}}} & \mathrm{Sym}^{n-1}\left(\vee^{i}V\right)\otimes\vee^{i-1}V \ar[d]|{\partial_{i-1}^{n-1}} \\ \mathrm{Sym}^{n-2}\left(\vee^{i}V\right)\otimes L\otimes V^{\vee} \ar[r]_-{\binom{r+g-1}{i}\frac{\rho^{i}}{2}\cdot1_{\mathrm{Sym}^{n-2}\left(\vee^{i}V\right)}\otimes\tau_{L,V^{\vee}}} & \mathrm{Sym}^{n-2}\left(\vee^{i}V\right)\otimes V^{\vee}\otimes L\text{,} & \mathrm{Sym}^{n-2}\left(\vee^{i}V\right)\otimes L\otimes V^{\vee} \ar[r]_-{\frac{\rho^{i}}{2}\cdot1_{\mathrm{Sym}^{n-2}\left(\vee^{i}V\right)}\otimes\tau_{L,V^{\vee}}} & \mathrm{Sym}^{n-2}\left(\vee^{i}V\right)\otimes V^{\vee}\otimes L\text{.}}
\end{equation*}

\item[$\left( 3\right) $] Suppose that $L\simeq \mathbb{L}^{\otimes 2}$ for
some invertible object $\mathbb{L}$, that $r_{\vee ^{i}V}>0$ (see definition %
\ref{Dirac definition positive rank})\ and that $V$ has symmetric rank $g$.
Then there are morphisms%
\begin{eqnarray*}
s_{\Delta }^{n-2} &:&\mathrm{Sym}^{n-2}\left( \vee ^{i}V\right) \otimes
L\rightarrow \mathrm{Sym}^{n}\left( \vee ^{i}V\right) \text{ for }n\geq 2%
\text{,} \\
s_{\overline{\partial }_{i-1}}^{n-1} &:&\mathrm{Sym}^{n-1}\left( \vee
^{i}V\right) \otimes V\rightarrow \mathrm{Sym}^{n}\left( \vee ^{i}V\right)
\otimes \vee ^{i-1}V^{\vee }\text{ for }n\geq 1\text{,} \\
s_{\partial _{i-1}}^{n-1} &:&\mathrm{Sym}^{n-1}\left( \vee ^{i}V\right)
\otimes V^{\vee }\otimes L\rightarrow \mathrm{Sym}^{n}\left( \vee
^{i}V\right) \otimes \vee ^{i-1}V\text{ for }n\geq 1
\end{eqnarray*}%
such that%
\begin{eqnarray*}
&&\Delta ^{n}\circ s_{\Delta }^{n-2}=1_{\mathrm{Sym}^{n-2}\left( \vee
^{i}V\right) \otimes L}\text{, }\overline{\partial }_{i-1}^{n}\circ s_{%
\overline{\partial }_{i-1}}^{n-1}=1_{\mathrm{Sym}^{n-1}\left( \vee
^{i}V\right) \otimes V} \\
&&\text{and }\partial _{i-1}^{n}\circ s_{\partial _{i-1}}^{n-1}=1_{\mathrm{%
Sym}^{n-1}\left( \vee ^{i}V\right) \otimes V^{\vee }\otimes L}.
\end{eqnarray*}%
In particular, the following objects exist:%
\begin{eqnarray*}
&&\ker \left( \Delta ^{n}\right) \subset \mathrm{Sym}^{n}\left( \vee
^{i}V\right) \text{,} \\
&&\ker \left( \overline{\partial }_{i-1}^{n}\right) \subset \mathrm{Sym}%
^{n}\left( \vee ^{i}V\right) \otimes \vee ^{i-1}V^{\vee }\text{,} \\
&&\ker \left( \partial _{i-1}^{n}\right) \subset \mathrm{Sym}^{n}\left( \vee
^{i}V\right) \otimes \vee ^{i-1}V\text{.}
\end{eqnarray*}
\end{itemize}
\end{theorem}

\section{Some remarks about the functoriality of the Dirac operators}

We will assume, from now on, that we are given an object $V\in \mathcal{C}$
and that $\mathcal{C}\ $and $\mathcal{D}$\ are $\mathbb{Q}$-linear rigid and
pseudo-abelian $ACU$ tensor categories. Once again, if $X$ is an object, we
set $r_{X}:=\mathrm{rank}\left( X\right) $ and we use the shorthand $r:=r_{V}
$. As usual, we write $e_{X,?}^{n}$, $i_{X,?}^{n}$ and $p_{X,?}^{n}$ for the
idempotent $e_{X,?}^{n}$ in $End\left( \otimes ^{n}X\right) $ giving rise to 
$\wedge ^{n}X$ when $?=a$ and $\vee ^{n}X$ when $?=s$ and the associated
canonical injective and surjective morphisms. We denote by $D_{V,?}^{i,j}$
and $D_{i,j}^{V,?}$ the Poincare duality morphisms in the algebra $\otimes
^{\cdot }V$ when $?=t$, $\wedge ^{\cdot }V$ when $?=a$ and $\vee ^{\cdot }V$
when $?=s$. Then it easily follows from \cite[Lemma 2.3, \S 5 and \S 6]{MS}
that, for every $g\geq i$ and $?=a$ or $s$,%
\begin{equation}
D_{V,?}^{i,g}:A_{i}\overset{i_{V,?}^{i}}{\rightarrow }\otimes ^{i}V\overset{%
D_{t}^{i,g}}{\rightarrow }\left( \otimes ^{g-i}V^{\vee }\right) \otimes
\left( \otimes ^{g}V^{\vee \vee }\right) \overset{p_{V^{\vee
},?}^{g-i}\otimes p_{V^{\vee \vee },?}^{g}}{\rightarrow }A_{g-i}^{\vee
}\otimes A_{g}^{\vee \vee }  \label{Poincare and Dirac functoriality F4}
\end{equation}%
and%
\begin{equation}
D_{i,g}^{V,?}:A_{i}^{\vee }\overset{i_{V,?}^{i}}{\rightarrow }\otimes
^{i}V^{\vee }\overset{D_{i,g}^{t}}{\rightarrow }\left( \otimes
^{g-i}V\right) \otimes \left( \otimes ^{g}V^{\vee }\right) \overset{%
p_{V,?}^{g-i}\otimes p_{V^{\vee },?}^{g}}{\rightarrow }A_{g-i}\otimes
A_{g}^{\vee }  \label{Poincare and Dirac functoriality F5}
\end{equation}

Suppose that we are given a (covariant) additive $AU$ tensor functor $F:%
\mathcal{C}\rightarrow \mathcal{D}$; it preserves internal $\hom $s and
dualities. We suppose that $F$ has the following further properties:

\begin{itemize}
\item $F\left( \tau _{V,V}\right) =\varepsilon \cdot \tau _{F\left( V\right)
,F\left( V\right) }$ and $F\left( \tau _{V^{\vee },V^{\vee }}\right)
=\varepsilon \cdot \tau _{F\left( V\right) ,F\left( V\right) }$, where $%
\varepsilon \in \left\{ \pm 1\right\} $;

\item $F\left( \tau _{V^{\vee },V}\right) =\eta \cdot \tau _{F\left(
V\right) ^{\vee },F\left( V\right) }$ (so that $F\left( \tau _{V,V^{\vee
}}\right) =\eta \cdot \tau _{F\left( V\right) ,F\left( V\right) ^{\vee }}$),
where $\eta \in \left\{ \pm 1\right\} $.
\end{itemize}

We remark that, if $\varepsilon =1$ (resp. $\varepsilon =-1$) and $X\in
\left\{ V,V^{\vee },V^{\vee \vee }\right\} $, we have $F\left(
e_{X,a}^{n}\right) =e_{F\left( X\right) ,a}^{n}$ (resp. $F\left(
e_{X,a}^{n}\right) =e_{F\left( X\right) ,s}^{n}$), $F\left(
e_{X,s}^{n}\right) =e_{F\left( X\right) ,s}^{n}$ (resp. $F\left(
e_{X,s}^{n}\right) =e_{F\left( X\right) ,a}^{n}$)\ and the same for the
associated injective and surjective morphisms. The following result is now
an easy consequence of this remark, $\left( \text{\ref{Poincare and Dirac
functoriality F4}}\right) $, $\left( \text{\ref{Poincare and Dirac
functoriality F5}}\right) $ and an explicit computation showing that $%
F\left( D_{V,t}^{i,g}\right) =\eta ^{\frac{i^{2}+i}{2}}\cdot D_{F\left(
V\right) ,t}^{i,g}$ and $F\left( D_{i,g}^{V,t}\right) =\eta ^{\frac{i^{2}-i}{%
2}}\cdot D_{i,g}^{F\left( V\right) ,t}$ (see \cite[\S 5]{MS} for the
explicit of $D_{V,t}^{i,j}$ and $D_{i,j}^{V,t}$).

\begin{lemma}
\label{Poincare and Dirac functoriality L1}Suppose that we are given a
(covariant) additive $AU$ tensor functor $F:\mathcal{C}\rightarrow \mathcal{D%
}$ as above.

\begin{itemize}
\item[$\left( 1\right) $] If $\varepsilon =1$ then we have%
\begin{eqnarray*}
F\left( D_{V,a}^{i,g}\right) &=&\eta ^{\frac{i^{2}+i}{2}}\cdot D_{F\left(
V\right) ,a}^{i,g}\text{, }F\left( D_{V,s}^{i,g}\right) =\eta ^{\frac{i^{2}+i%
}{2}}\cdot D_{F\left( V\right) ,s}^{i,g}\text{,} \\
F\left( D_{i,g}^{V,a}\right) &=&\eta ^{\frac{i^{2}-i}{2}}\cdot
D_{i,g}^{F\left( V\right) ,a}\text{, }F\left( D_{i,g}^{V,s}\right) =\eta ^{%
\frac{i^{2}-i}{2}}\cdot D_{i,g}^{F\left( V\right) ,s}\text{.}
\end{eqnarray*}

\item[$\left( 2\right) $] If $\varepsilon =-1$ then the same formulas hold after
  swapping the symbols $s$ and $a$ in the right-hand side.
\end{itemize}
\end{lemma}

\bigskip

Fix $g\geq i$ such that $g=2i$ and set $L_{a}:=\wedge ^{g}V^{\vee \vee }$
and $L_{s}:=\vee ^{g}V^{\vee \vee }$. Write $\psi _{i,1}^{V,a}=\psi
_{g-i,1}^{V,a}$ and $\overline{\psi }_{g-i,g-i-1}^{V,a}=\overline{\psi }%
_{i,i-1}^{V,a}$ for the pairings defined before Lemma \ref{Dirac Alternating
L1}, $\psi _{i,g-i-1}^{V,a}=\psi _{g-i,i-1}^{V,a}$ and $\overline{\psi }%
_{g-i,1}^{V,a}=\overline{\psi }_{i,1}^{V,a}$ for those defined before \ref%
{Dirac Alternating L2}, $\psi _{i,1}^{V,s}=\psi _{g-i,1}^{V,s}$ and $%
\overline{\psi }_{g-i,g-i-1}^{V,s}=\overline{\psi }_{i,i-1}^{V,s}$ for the
ones considered before Lemma \ref{Dirac Symmetric L1} and $\psi
_{i,g-i-1}^{V,s}=\psi _{g-i,i-1}^{V,s}$ and $\overline{\psi }_{g-i,1}^{V,s}=%
\overline{\psi }_{i,1}^{V,s}$ for those defined before \ref{Dirac Symmetric
L2}.

Consider the following operators from \S \S \ref{Subsection LD Alternating}%
\footnote{%
Of course some of them will be zero, but it will be convenient to consider
all of them, in order to state the result in a symmetric way.}:%
\begin{equation*}
\begin{array}{llll}
\Delta ^{\mathrm{Alt}^{n}\left( \wedge ^{i}V\right) }:=\Delta _{i_{\wedge
^{g}V}\circ \varphi _{i,i},a}^{n}: & \mathrm{Alt}^{n}\left( \wedge
^{i}V\right) & \rightarrow & \mathrm{Alt}^{n-2}\left( \wedge ^{i}V\right)
\otimes L_{a}\text{,} \\ 
\overline{\partial }_{i-1}^{\mathrm{Alt}^{n}\left( \wedge ^{i}V\right)
}:=\partial _{\overline{\psi }_{i,i-1}^{V,a},a}^{n}: & \mathrm{Alt}%
^{n}\left( \wedge ^{i}V\right) \otimes \wedge ^{i-1}V^{\vee } & \rightarrow
& \mathrm{Alt}^{n-1}\left( \wedge ^{i}V\right) \otimes V\text{,} \\ 
\partial _{i-1}^{\mathrm{Alt}^{n}\left( \wedge ^{i}V\right) }:=\partial
_{\psi _{g-i,i-1}^{V,a},a}^{n}: & \mathrm{Alt}^{n}\left( \wedge ^{i}V\right)
\otimes \wedge ^{i-1}V & \rightarrow & \mathrm{Alt}^{n-1}\left( \wedge
^{i}V\right) \otimes V^{\vee }\otimes L_{a}\text{,} \\ 
\end{array}%
\end{equation*}%
and similar for $\Delta^{\mathrm{Sym^n}(\wedge^i V)} :=\Delta _{i_{\wedge
^{g}V}\circ \varphi _{i,i},a}^{n}$, $\bar\partial_{i-1}^{\mathrm{Sym^n}(\wedge^i V)} :=\partial _{\overline{\psi }_{i,i-1}^{V,a},s}^{n}$ and $\partial_{i-1}^{\mathrm{Sym^n}(\wedge^i V)} :=\partial
 _{\psi _{g-i,i-1}^{V,a},s}^{n}$ where one swaps
the symbols $\mathrm{Alt}$ with $\mathrm{Sym}$.

Similarly, in order to symmetrically state the results, we will need to
consider the operators from \S \S \ref{Subsection LD Symmetric} together
with the analogous operators induced on the alternating powers:%
\begin{equation*}
\begin{array}{llll}
\Delta ^{\mathrm{Alt}^{n}\left( \vee ^{i}V\right) }:=\Delta _{i_{\vee
^{g}V}\circ \varphi _{i,i},a}^{n}: & \mathrm{Alt}^{n}\left( \vee ^{i}V\right)
& \rightarrow & \mathrm{Alt}^{n-2}\left( \vee ^{i}V\right) \otimes L_{s}%
\text{,} \\ 
\overline{\partial }_{i-1}^{\mathrm{Alt}^{n}\left( \vee ^{i}V\right)
}:=\partial _{\overline{\psi }_{i,i-1}^{V,s},a}^{n}: & \mathrm{Alt}%
^{n}\left( \vee ^{i}V\right) \otimes \vee ^{i-1}V^{\vee } & \rightarrow & 
\mathrm{Alt}^{n-1}\left( \vee ^{i}V\right) \otimes V\text{,} \\ 
\partial _{i-1}^{\mathrm{Alt}^{n}\left( \vee ^{i}V\right) }:=\partial _{\psi
_{g-i,i-1}^{V,s},a}^{n}: & \mathrm{Alt}^{n}\left( \vee ^{i}V\right) \otimes
\vee ^{i-1}V & \rightarrow & \mathrm{Alt}^{n-1}\left( \vee ^{i}V\right)
\otimes V^{\vee }\otimes L_{s}\text{,} \\ 
\end{array}%
\end{equation*}
and similar for the remaining three operators with $\textrm{Alt}$ and $\textrm{Sym}$ swapped.
The following result, whose proof is left to the reader, follows from Lemma %
\ref{Poincare and Dirac functoriality L1} and a small computation.

\begin{proposition}
\label{Poincare and Dirac functoriality P1}Suppose that we are given a
(covariant) additive $AU$ tensor functor $F:\mathcal{C}\rightarrow \mathcal{D%
}$ as above.

\begin{itemize}
\item[$\left( 1\right) $] If $\varepsilon =1$ then we have%
\begin{equation*}
\begin{array}{ll}
F\left( \Delta ^{\mathrm{Alt}^{n}\left( \wedge ^{i}V\right) }\right) =\Delta
^{\mathrm{Alt}^{n}\left( \wedge ^{i}F\left( V\right) \right) }\text{,} & 
F\left( \Delta ^{\mathrm{Alt}^{n}\left( \vee ^{i}V\right) }\right) =\Delta ^{%
\mathrm{Alt}^{n}\left( \vee ^{i}F\left( V\right) \right) }\text{,} \\ 
F\left( \overline{\partial }_{i-1}^{\mathrm{Alt}^{n}\left( \wedge
^{i}V\right) }\right) =\eta ^{\frac{i\left( i+1\right) }{2}+1}\cdot 
\overline{\partial }_{i-1}^{\mathrm{Alt}^{n}\left( \wedge ^{i}F\left(
V\right) \right) }\text{,} & F\left( \overline{\partial }_{i-1}^{\mathrm{Alt}%
^{n}\left( \vee ^{i}V\right) }\right) =\eta ^{\frac{i\left( i+1\right) }{2}%
+1}\cdot \overline{\partial }_{i-1}^{\mathrm{Alt}^{n}\left( \vee ^{i}F\left(
V\right) \right) }\text{,} \\ 
F\left( \partial _{i-1}^{\mathrm{Alt}^{n}\left( \wedge ^{i}V\right) }\right)
=\eta ^{\frac{g\left( g-1\right) }{2}}\cdot \partial _{i-1}^{\mathrm{Alt}%
^{n}\left( \wedge ^{i}F\left( V\right) \right) }\text{,} & F\left( \partial
_{i-1}^{\mathrm{Alt}^{n}\left( \vee ^{i}V\right) }\right) =\eta ^{\frac{%
g\left( g-1\right) }{2}}\cdot \partial _{i-1}^{\mathrm{Alt}^{n}\left( \vee
^{i}F\left( V\right) \right) }\text{.}
\end{array}%
\end{equation*}
Six more formulas hold where the symbols $\operatorname{Alt}$ and $\operatorname{Sym}$ are swapped.

\item[$\left( 2\right) $] If $\varepsilon =-1$ and $i$ is even, then
  similar twelve formulas hold where in the right-hand side the symbols $\wedge$ and $\vee$
  must be swapped. If $\varepsilon = -1$ and $i$ is odd, in addition one must swap
  the symbols $\operatorname{Alt}$
  and $\operatorname{Sym}$.

\end{itemize}
\end{proposition}

\subsection{Application to quaternionic objects}

We will now focus on the case $i=2$ and $g=2i=4$ and we let $B$ be a
quaternion $\mathbb{Q}$-algebra, whose main involution we denote by $%
b\mapsto b^{\iota }$. An alternating (resp. symmetric) quaternionic object
in $\mathcal{C}$\ is a couple $\left( V,\theta \right) $\ where $V$ has
alternating (resp. symmetric) rank $4$ and $\theta :B\rightarrow End\left(
V\right) $ is a unitary ring homomorphism. We will assume that such a $%
\left( V,\theta \right) $ has been given in the following discussion.

We have $\vee ^{2}B\subset B\otimes B$, the $\mathbb{Q}$-vector space
generated by the elements $b_{1}\vee b_{2}=\frac{1}{2}\left( b_{1}\otimes
b_{2}+b_{2}\otimes b_{1}\right) $. Noticing that%
\begin{equation*}
\left( b_{1}+b_{2}\right) \otimes \left( b_{1}+b_{2}\right) =b_{1}\otimes
b_{1}+b_{2}\otimes b_{2}+b_{1}\otimes b_{2}+b_{2}\otimes b_{1}
\end{equation*}%
and that $b\vee b=b\otimes b$, we see that%
\begin{equation*}
b_{1}\vee b_{2}=\frac{\left( b_{1}+b_{2}\right) \vee \left(
b_{1}+b_{2}\right) }{2}-\frac{b_{1}\vee b_{1}}{2}-\frac{b_{2}\vee b_{2}}{2}%
\text{,}
\end{equation*}%
so that $\vee ^{2}B$ is the $\mathbb{Q}$-vector space generated by the
elements $b\vee b$. Considering $B\otimes B$ as a $\mathbb{Q}$-algebra in
the natural way, it follows that $\vee ^{2}B$ is a subalgebra, because the
product of elements of the form $b\vee b$ is again of this form. Let \textrm{%
Tr}$:B\rightarrow \mathbb{Q}$ and \textrm{Nr}$:B\rightarrow \mathbb{Q}$ be
the reduced trace and norm and set $B_{0}:=\ker \left( \mathrm{Tr}\right) $.

Write $W$ for the $\mathbb{Q}$-vector space $B$, endowed with the action of $%
B\otimes B$ defined by the rule $b_{1}\otimes b_{2}\cdot
x:=b_{1}xb_{2}^{\iota }$. It gives rise to a unitary ring homomorphism $%
f:B\otimes B\rightarrow End_{\mathbb{Q}}\left( W\right) \simeq \mathbf{M}%
_{4}\left( \mathbb{Q}\right) $ which is injective because $B\otimes B$ is
simple, hence an isomorphism by counting dimensions. As a $\vee ^{2}B$%
-module $W=B_{0}\oplus \mathbb{Q}$ and it easily follows that the resulting
homomorphism%
\begin{equation*}
\vee ^{2}B\rightarrow End_{\mathbb{Q}}\left( B_{0}\right) \oplus End_{%
\mathbb{Q}}\left( \mathbb{Q}\right) \simeq \mathbf{M}_{3}\left( \mathbb{Q}%
\right) \oplus \mathbb{Q}
\end{equation*}%
is an isomorphism: it is injective because $End_{\mathbb{Q}}\left(
B_{0}\right) \oplus End_{\mathbb{Q}}\left( \mathbb{Q}\right) \subset End_{%
\mathbb{Q}}\left( W\right) $ and $f$ is injective, hence an isomorphism
again by counting dimensions. Furthermore, the action of $\vee ^{2}B$ on $%
\mathbb{Q}$ is given by the $\mathbb{Q}$-algebra homomorphism%
\begin{equation*}
\chi :\vee ^{2}B\rightarrow \mathbb{Q}\text{, }\chi \left( b_{1}\vee
b_{2}\right) =\frac{\mathrm{Tr}\left( b_{1}^{\iota }b_{2}\right) }{2}\text{.}
\end{equation*}%
It follows that there is an idempotent $e_{-}\in \vee ^{2}B$ characterized
by $se_{-}=\chi \left( s\right) e_{-}$ for every $s\in \vee ^{2}B$.

We have a natural $B\otimes B$-action on $V\otimes V$ by $\theta ^{\otimes
2}:=\theta \otimes \theta $\ and, since $b\otimes b\circ
e_{V,?}^{2}=e_{V,?}^{2}\circ b\otimes b$ for $?\in \left\{ a,s\right\} $, $%
B\subset B\otimes B$ (diagonally)\ operates on $\wedge ^{2}V$ and $\vee ^{2}V
$. But $\vee ^{2}B$ is generated by the elements of the form $b\otimes b$ as
a $\mathbb{Q}$-algebra (and indeed as a $\mathbb{Q}$-vector space, as
already noticed): hence $\vee ^{2}B\subset B\otimes B$ operates on $\wedge
^{2}V$ and $\vee ^{2}V$. The above discussion shows that we may write%
\begin{equation}
\wedge ^{2}V=\left( \wedge ^{2}V\right) _{+}\oplus \left( \wedge
^{2}V\right) _{-}\text{ and }\vee ^{2}V=\left( \vee ^{2}V\right) _{+}\oplus
\left( \vee ^{2}V\right) _{-}  \label{Poincare and Dirac functoriality F6}
\end{equation}%
where $\left( \wedge ^{2}V\right) _{-}:=\func{Im}\left( \theta ^{\otimes
2}\left( e_{-}\right) \right) $ and $\left( \vee ^{2}V\right) _{-}:=\func{Im}%
\left( \theta ^{\otimes 2}\left( e_{-}\right) \right) $ are characterized by
the property that $\vee ^{2}B$ acts on them via $\chi $. Indeed we remark
that, since $\vee ^{2}B$ is generated by the diagonal image of $B\subset
B\otimes B$ and $\chi \left( b\otimes b\right) =$\textrm{Nr}$\left( b\right) 
$, $\left( \wedge ^{2}V\right) _{-}$ (resp. $\left( \vee ^{2}V\right) _{-}$)
is the unique maximal subobject of $\wedge ^{2}V$ (resp. $\vee ^{2}V$) on
which $B$ acts via the reduced norm.

Associated with $\left( V,\theta \right) $\ is the dual quaternionic object $%
\left( V^{\vee },\theta ^{\vee }\right) $ where $\theta ^{\vee }\left(
b\right) :=\theta \left( b^\iota\right) ^{\vee }$. We will simply write $\left(
\wedge ^{2}V^{\vee }\right) _{\pm }$ (resp. $\left( \vee ^{2}V^{\vee
}\right) _{\pm }$) for the $\pm $ components attached to $\left( V^{\vee
},\theta ^{\vee }\right) $ obtained in this way. Since by definition $\theta
^{\vee \otimes 2}\left( e_{-}\right) :=\left( \theta \left( e_{-}\right)
^{\vee }\right) ^{\otimes 2}=\left( \theta \left( e_{-}\right) ^{\otimes
2}\right) ^{\vee }$, we have $\left( \wedge ^{2}V^{\vee }\right) _{\pm
}=\left( \wedge ^{2}V_{\pm }\right) ^{\vee }$ (resp. $\left( \vee
^{2}V^{\vee }\right) _{\pm }=\left( \vee ^{2}V_{\pm }\right) ^{\vee }$)

We summarize the above discussion in the first part of following lemma,
while the second follows from the remark before Lemma \ref{Poincare and
Dirac functoriality L1}.

\begin{lemma}
\label{Poincare and Dirac functoriality L2}If $\left( V,\theta \right) $ is
an alternating (resp. symmetric) quaternionic object in $\mathcal{C}$, there
is a canonical decomposition $\left( \text{\ref{Poincare and Dirac
functoriality F6}}\right) $\ (in the category of quaternionic objects), where 
$\left( \wedge ^{2}V\right) _{-}$ (resp. $\left( \vee ^{2}V\right) _{-}$) is
characterized by the fact that it is the unique maximal subobject $X$ of $%
\wedge ^{2}V$ (resp. $\vee ^{2}V$) such that the action of $B$ acting
diagonally on $\wedge ^{2}V$ (resp. $\vee ^{2}V$) is given by the reduced
norm on $X$. We have $\left( \wedge ^{2}V^{\vee }\right) _{\pm }=\left(
\wedge ^{2}V_{\pm }\right) ^{\vee }$ (resp. $\left( \vee ^{2}V^{\vee
}\right) _{\pm }=\left( \vee ^{2}V_{\pm }\right) ^{\vee }$).

Suppose that we are given a (covariant) additive $AU$ tensor functor $F:%
\mathcal{C}\rightarrow \mathcal{D}$ as above and define $F\left( \theta
\right) \left( b\right) :=F\left( \theta \left( b\right) \right) $. Then $%
\left( F\left( V\right) ,F\left( \theta \right) \right) $ is an alternating
(resp. symmetric) quaternionic object in $\mathcal{D}$ when $\varepsilon =1$%
, $\left( F\left( V\right) ,F\left( \theta \right) \right) $ is a symmetric
(resp. alternating) quaternionic object in $\mathcal{D}$ when $\varepsilon
=-1$ and we have%
\begin{equation*}
F\left( \left( \wedge ^{2}V\right) _{\pm }\right) =\left( \wedge ^{2}F\left(
V\right) \right) _{\pm }\text{ (resp. }F\left( \left( \vee ^{2}V\right)
_{\pm }\right) =\left( \vee ^{2}F\left( V\right) \right) _{\pm })\text{ when 
}\varepsilon =1
\end{equation*}%
and%
\begin{equation*}
F\left( \left( \wedge ^{2}V\right) _{\pm }\right) =\left( \vee ^{2}F\left(
V\right) \right) _{\pm }\text{ (resp. }F\left( \left( \vee ^{2}V\right)
_{\pm }\right) =\left( \wedge ^{2}F\left( V\right) \right) _{\pm })\text{
when }\varepsilon =-1\text{.}
\end{equation*}
\end{lemma}

\bigskip 

Since $i=2$, we have $1=i-1$ and it follows that$\ \psi _{i,1}^{V,?}=\psi
_{g-i,1}^{V,?}=\psi _{i,g-i-1}^{V,?}=\psi _{g-i,i-1}^{V,?}$ and $\overline{%
\psi }_{g-i,g-i-1}^{V,?}=\overline{\psi }_{i,i-1}^{V,?}=\overline{\psi }%
_{g-i,1}^{V,?}=\overline{\psi }_{i,1}^{V,?}$. Hence, our discussion on Dirac
operators is confined to the two pairings $\psi ^{V,?}:=\psi
_{i,1}^{V,?}$ and $\overline{\psi }_{g-i,g-i-1}^{V,?}:=\overline{\psi }%
_{i,1}^{V,?}$. Together with the multiplication maps they induce%
\begin{equation*}
\begin{array}{llll}
\Delta ^{\mathrm{Sym}^{n}\left( \wedge ^{2}V\right) }:=\Delta _{i_{\wedge
^{g}V}\circ \varphi _{2,2},a}^{n} & \mathrm{Sym}^{n}\left( \wedge
^{2}V\right)  & \rightarrow  & \mathrm{Sym}^{n-2}\left( \wedge ^{2}V\right)
\otimes L_{a}\text{,} \\ 
\overline{\partial }^{\mathrm{Sym}^{n}\left( \wedge ^{2}V\right) }:=\partial
_{\overline{\psi }^{V,a},s}^{n}: & \mathrm{Sym}^{n}\left( \wedge
^{2}V\right) \otimes V^{\vee } & \rightarrow  & \mathrm{Sym}^{n-1}\left(
V\right) \otimes V\text{,} \\ 
\partial ^{\mathrm{Sym}^{n}\left( \wedge ^{2}V\right) }:=\partial _{\psi
^{V,a},s}^{n}: & \mathrm{Sym}^{n}\left( \wedge ^{2}V\right) \otimes V & 
\rightarrow  & \mathrm{Sym}^{n-1}\left( V\right) \otimes V^{\vee }\otimes
L_{a}%
\end{array}%
\end{equation*}%
and their analogous where $\wedge ^{2}V$ is replaced by $\vee ^{2}V$ in the
notation and the sources and the targets. On the other hand, let $\psi
_{-}^{V,?}:=\psi ^{V,?}\circ \left( i_{-}\otimes 1_{V}\right) $, $\overline{%
\psi }_{-}^{V,?}:=\overline{\psi }^{V,?}\circ \left( i_{-}\otimes 1_{V^{\vee
}}\right) $ and $\varphi _{2,2,-}:=\varphi _{2,2}\circ \left( i_{-}\otimes
i_{-}\right) $ be the restrictions of these pairings, where $i_{-}$ is the
injection associated to $e_{-}$. Then we have operators $\Delta _{-}^{?}$, $%
\overline{\partial }_{-}^{?}$ and $\partial _{-}^{?}$ with $?\in \left\{ 
\mathrm{Alt}^{n}\left( \wedge ^{2}V\right) ,\mathrm{Sym}^{n}\left( \wedge
^{2}V\right) ,\mathrm{Alt}^{n}\left( \vee ^{2}V\right) ,\mathrm{Sym}%
^{n}\left( \vee ^{2}V\right) \right\} $ induced by these pairings.

\begin{corollary}
\label{Poincare and Dirac functoriality C1}Suppose that $\left( V,\theta
\right) $ is an alternating (resp. symmetric) quaternionic object in $%
\mathcal{C}$ such that $L_{a}\simeq \mathbb{L}^{\otimes 2}$ (resp. $%
L_{s}\simeq \mathbb{L}^{\otimes 2}$) for some invertible object $\mathbb{L}$
and that we have $r_{\wedge ^{2}V}>0$ (resp. $r_{\vee ^{2}V}>0$)\footnote{%
As explained in the introduction, this latter condition on the rank of the $2
$-powers is always fulfilled when $V$ is Kimura positive (resp. negative).}.
Then the Laplace and the Dirac operators induced by these $\varphi _{2,2,-}$%
, $\psi _{-}^{V,?}$ and $\overline{\psi }_{-}^{V,?}$ satisfies the
conclusion of Theorem \ref{Dirac Alternating T2} (resp. Theorem \ref{Dirac
Symmetric T1}).

Furthermore, if we are given a (covariant) additive $AU$ tensor functor $F:%
\mathcal{C}\rightarrow \mathcal{D}$ as above, then the conclusion of
Proposition \ref{Poincare and Dirac functoriality P1} holds true with the
Laplace and the Dirac operators induced by these $\varphi _{2,2,-}$, $\psi
_{-}^{V,?}$ and $\overline{\psi }_{-}^{V,?}$.
\end{corollary}

\begin{proof}
Suppose that we are given $\psi :X\otimes Y\rightarrow Z$, $\psi ^{\prime
}:X^{\prime }\otimes Y^{\prime }\rightarrow Z^{\prime }$ and morphisms $%
f:X\rightarrow X^{\prime }$, $g:Y\rightarrow Y^{\prime }$ and $%
h:Z\rightarrow Z^{\prime }$ such that $\psi ^{\prime }\circ \left( f\otimes
g\right) =h\circ \psi $. Then it is clear from \S \ref{Dirac and Laplace
operators} that we have $\partial _{\psi ^{\prime },a}^{n}\circ \left(
\left( \wedge ^{n}f\right) \otimes g\right) =\left( \left( \wedge
^{n-1}f\right) \otimes h\right) \circ \partial _{\psi ,a}^{n}$ and $\partial
_{\psi ^{\prime },s}^{n}\circ \left( \left( \vee ^{n}f\right) \otimes
g\right) =\left( \left( \vee ^{n-1}f\right) \otimes h\right) \circ \partial
_{\psi ,s}^{n}$. Similarly, when $f=g$, we deduce $\Delta _{\psi ^{\prime
},a}^{n}\circ \left( \wedge ^{n}f\right) =\left( \left( \wedge
^{n-2}f\right) \otimes h\right) \circ \Delta _{\psi ,a}^{n}$ and $\Delta
_{\psi ^{\prime },s}^{n}\circ \left( \left( \vee ^{n}f\right) \right)
=\left( \left( \vee ^{n-2}f\right) \otimes h\right) \circ \Delta _{\psi
,s}^{n}$. We write $\psi \rightarrow _{\left( f,g,h\right) }\psi ^{\prime }$
in this case. Then, setting $A_{g}:=\wedge ^{g}V$ or $\vee ^{g}V$, we have
by definition $i_{A_{g}}\circ \varphi _{2,2,-}\rightarrow _{\left(
i_{-},i_{-},1_{L_{?}}\right) }i_{A_{g}}\circ \varphi _{2,2}$, $\psi
_{-}^{V,?}\rightarrow _{\left( i_{-},1_{V},1_{L_{?}}\right) }\psi ^{V,?}$
and $\overline{\psi }_{-}^{V,?}\rightarrow _{\left( i_{-},1_{V^{\vee
}},1_{L_{?}}\right) }\overline{\psi }^{V,?}$. It follows that, with $\psi $
one of these pairings and $\psi _{-}$ the corresponding pairing obtained by
restriction, the induced Laplace or Dirac operators commutes with the
canonical injections induced by $\wedge ^{k}i_{-}$ (resp. $\vee ^{k}i_{-}$)
in the sources and the targets. We simply write $\Delta ^{n}$, $\overline{%
\partial }^{n}$ and $\partial ^{n}$ and $\Delta _{-}^{n}$, $\overline{%
\partial }_{-}^{n}$ and $\partial _{-}^{n}$ for one of these operators.
Hence, if $\left( V,\theta \right) $ is alternating (resp. symmetric), we
may apply Theorem \ref{Dirac Alternating T2} (resp. Theorem \ref{Dirac
Symmetric T1}) to deduce that $\overline{\partial }_{-}^{n-1}\circ \partial
_{-}^{n}$ or $\partial _{-}^{n-1}\circ \overline{\partial }_{-}^{n}$ equals $%
\Delta _{-}^{n}\otimes 1_{V}$ or $\Delta _{-}^{n}\otimes 1_{V^{\vee }}$ up
to the isomorphism provided by this theorem. Let $e$ be the idempotent which
gives the kernel of one of the Laplace or Dirac operators $\Delta ^{n}$, $%
\overline{\partial }^{n}$ and $\partial ^{n}$; on the other hand we have an
idempotent $e^{\prime }$ of the form $e^{\prime }:=\mathrm{Alt}^{n}\left(
e_{-}\right) \otimes 1_{Z}$ or $\mathrm{Sym}^{n}\left( e_{-}\right) \otimes
1_{Z}$ which corresponds to the injections induced by $\wedge ^{k}i_{-}$
(resp. $\vee ^{k}i_{-}$) in the sources of these operators $\Delta _{-}^{n}$%
, $\overline{\partial }_{-}^{n}$ and $\partial _{-}^{n}$. Then $ee^{\prime }$
gives the kernel of the Laplace or Dirac operators $\Delta _{-}^{n}$, $%
\overline{\partial }_{-}^{n}$ and $\partial _{-}^{n}$ (once again because
the operators induced by $\psi $ or $\psi _{-}$ are related by a commutative
diagram involving injections). Hence the analogue of Theorem \ref{Dirac
Alternating T2} (resp. Theorem \ref{Dirac Symmetric T1}) $\left( 3\right) $
is true. Finally, the statement about $F$ follows from Lemma \ref{Poincare
and Dirac functoriality L2} and Proposition \ref{Poincare and Dirac
functoriality P1}.
\end{proof}

\bigskip 

\begin{definition}
\label{Def quaternionic}If we are given an alternating (resp. symmetric)
quaternionic object in $\mathcal{C}$, we define $M_{2}\left( V,\theta
\right) :=\left( \wedge ^{2}V\right) _{-}$ (resp. $M_{2}\left( V,\theta
\right) :=\left( \vee ^{2}V\right) _{-}$),%
\begin{eqnarray*}
&&\text{ }M_{2n}\left( V,\theta \right) :=\ker \left( \Delta _{-}^{\mathrm{%
Sym}^{n}\left( \wedge ^{2}V\right) }\right) \subset \mathrm{Sym}^{n}\left(
\left( \wedge ^{2}V\right) _{-}\right) \text{, where }n\geq 2 \\
&&\text{ (resp. }M_{2n}\left( V,\theta \right) :=\ker \left( \Delta _{-}^{%
\mathrm{Sym}^{n}\left( \vee ^{2}V\right) }\right) \subset \mathrm{Sym}%
^{n}\left( \left( \vee ^{2}V\right) _{-}\right) \text{)}
\end{eqnarray*}%
and $M_{1}\left( V,\theta \right) :=V$ (resp. $M_{1}\left( V,\theta \right)
:=V$)%
\begin{eqnarray*}
&&\text{ }M_{2n+1}\left( V,\theta \right) :=\ker \left( \partial _{-}^{%
\mathrm{Sym}^{n}\left( \wedge ^{2}V\right) }\right) \subset \mathrm{Sym}%
^{n}\left( \left( \wedge ^{2}V\right) _{-}\right) \otimes V\text{, where }%
n\geq 1 \\
&&\text{ (resp. }M_{2n+1}\left( V,\theta \right) :=\ker \left( \partial
_{-}^{\mathrm{Sym}^{n}\left( \vee ^{2}V\right) }\right) \subset \mathrm{Sym}%
^{n}\left( \left( \vee ^{2}V\right) _{-}\right) \otimes V\text{)}
\end{eqnarray*}
\end{definition}

It follows from Lemma \ref{Poincare and Dirac functoriality L2}\ and
Corollary \ref{Poincare and Dirac functoriality C1} that these objects are
canonical in the category of quaternionic objects and that, if we are given
a (covariant) $AU$ tensor functor $F:\mathcal{C}\rightarrow \mathcal{D}$ as
above, then $F\left( M_{2n}\left( V,\theta \right) \right) =M_{2n}\left(
F\left( V\right) ,F\left( \theta \right) \right) $ and $F\left(
M_{2n+1}^{\left( \oplus 2\right) }\left( V,\theta \right) \right)
=M_{2n+1}^{\left( \oplus 2\right) }\left( F\left( V\right) ,F\left( \theta
\right) \right) $.

\section{Realizations}

\subsection{The quaternionic Poincar\'{e} upper half plane}

We write%
\begin{equation*}
\mathcal{P}:=\mathbf{P}_{\mathbb{C}}^{1}-\mathbf{P}_{\mathbb{R}}^{1}:=%
\mathcal{H}^{+}\sqcup \mathcal{H}^{-}\text{ and }\mathcal{H}:=\mathcal{H}^{+}%
\text{,}
\end{equation*}%
where $\mathcal{H}^{\pm }$ is the connected component of $\mathcal{P}$ such
that $\pm i\in \mathcal{H}^{\pm }$. We recall that $\mathcal{P}$ has a
natural moduli interpretation in the category of analytic spaces\footnote{%
  See~\cite[I.1.5]{GR} for a treatment of analytic spaces, but beware that they are called ``complex spaces'' in this book.}  as follows. Set $L_{1}:=\mathbb{Z}^{2}$, $%
P_{1}:=\left( \mathbb{Z}^{2}\right) ^{\vee }$ (the $\mathbb{Z}$-dual)\ and,
for a positive integer $k$, $L_{k}:=S_{\mathbb{Z}}^{k}\left( L_{1}\right) $
(the $k$-symmetric power of $L_{1}$) and $P_{k}:=S_{\mathbb{Z}}^{k}\left(
P_{1}\right) =L_{k}^{\vee }$, the space of homogeneous polynomials of degree 
$k$ in two variables $\left( X,Y\right) $. Then we have $\mathcal{O}_{%
\mathbf{P}_{\mathbb{C}}^{1}}\left( k\right) \left( \mathbf{P}_{\mathbb{C}%
}^{1}\right) =P_{k,\mathbb{C}}$\footnote{%
If $M$ is an $A$-module over a ring $A\subset B$, we write $M_{B}:=B\otimes
_{A}M$.}. To give an $\mathcal{S}$-point $x:\mathcal{S}\rightarrow \mathbf{P}%
_{\mathbb{C}}^{1}$ from an analytic space $\mathcal{S}$ is to give an
epimorphism $\mathcal{O}_{\mathcal{S}}\left( P_{1}\right) \twoheadrightarrow 
\mathcal{L}$\footnote{%
If $M$ is an $A$-module over a ring $A\subset \mathbb{C}$ and $\mathcal{S}$
is an analytic space, we write $M_{\mathcal{S}}$ for the associated sheaf of
locally constant $M$-valued functions on $\mathcal{S}$ and $\mathcal{O}_{%
\mathcal{S}}\left( M\right) :=\mathcal{O}_{\mathcal{S}}\otimes _{A_{\mathcal{%
S}}}M_{\mathcal{S}}$.} up to isomorphism, where $\mathcal{L}$ is an
invertible sheaf on $\mathcal{S}$ and, taking $x=1_{\mathbf{P}_{\mathbb{C}%
}^{1}}$, gives the universal quotient%
\begin{equation*}
\mathcal{O}_{\mathbf{P}_{\mathbb{C}}^{1}}\left( P_{1}\right)
\twoheadrightarrow \mathcal{O}_{\mathbf{P}_{\mathbb{C}}^{1}}\left( 1\right)
\end{equation*}%
mapping the global sections $1\otimes X,1\otimes Y\in \mathcal{O}_{\mathbf{P}%
_{\mathbb{C}}^{1}}\left( P_{1,\mathbb{C}}\right) \left( \mathbf{P}_{\mathbb{C%
}}^{1}\right) $ respectively to the global sections $X,Y\in \mathcal{O}_{%
\mathbf{P}_{\mathbb{C}}^{1}}\left( 1\right) \left( \mathbf{P}_{\mathbb{C}%
}^{1}\right) $. Taking duals we see that to give $x:\mathcal{S}\rightarrow 
\mathbf{P}_{\mathbb{C}}^{1}$ is the same as to give a monomorphism $\mathcal{%
L}\hookrightarrow \mathcal{O}_{\mathcal{S}}\left( L_{1,\mathbb{C}}\right) $
up to isomorphism, where $\mathcal{L}$ is an invertible sheaf on $\mathcal{S}
$ and the cokernel of the inclusion is locally free too; taking $x=1_{%
\mathbf{P}_{\mathbb{C}}^{1}}$ gives the universal object%
\begin{equation*}
\mathcal{O}_{\mathbf{P}_{\mathbb{C}}^{1}}\left( -1\right) \hookrightarrow 
\mathcal{O}_{\mathbf{P}_{\mathbb{C}}^{1}}\left( L_{1}\right) \text{.}
\end{equation*}%
Although not needed, let us remark that we have indeed $\mathcal{O}_{\mathbf{%
P}_{\mathbb{C}}^{1}}\left( L_{1}\right) =\mathcal{O}_{\mathbf{P}_{\mathbb{C}%
}^{1}}\left( P_{1}\right) $ and the above universal epimorphism and
monomorphism are part of usual canonical short exact sequence. It follows
that, setting $F_{x}^{-1}\left( \mathcal{O}_{\mathbf{P}_{\mathbb{C}%
}^{1}}\left( L_{1}\right) \right) :=\mathcal{O}_{\mathbf{P}_{\mathbb{C}%
}^{1}}\left( L_{1}\right) $ and $F_{x}^{0}\left( \mathcal{O}_{\mathbf{P}_{%
\mathbb{C}}^{1}}\left( L_{1}\right) \right) :=\func{Im}\left( \mathcal{L}%
\hookrightarrow \mathcal{O}_{\mathcal{S}}\left( L_{1}\right) \right) $, the
space $\mathbf{P}_{\mathbb{C}}^{1}$ classifies all the possible filtrations
on $\mathcal{O}_{\mathcal{S}}\left( L_{1}\right) $ by an invertible $%
\mathcal{O}_{\mathcal{S}}$-module having locally free cokernel.

It is easy to realize that a necessary and sufficient condition for a point $%
x:\mathcal{S}\rightarrow \mathbf{P}_{\mathbb{C}}^{1}$ to factor through $%
\mathcal{P}$ is that the filtration $F_{x}^{\cdot }\left( \mathcal{O}_{%
\mathbf{P}_{\mathbb{C}}^{1}}\left( L_{1}\right) \right) $ on $\mathcal{O}_{%
\mathbf{P}_{\mathbb{C}}^{1}}\left( L_{1}\right) =\mathcal{O}_{\mathbf{P}_{%
\mathbb{C}}^{1}}\left( L_{1,\mathbb{R}}\right) $ gives $L_{1,\mathbb{R},%
\mathcal{P}}$\ the structure of a variation of Hodge structures of type $%
\left\{ \left( -1,0\right) ,\left( 0,-1\right) \right\} $. Hence $\mathcal{P}
$ classifies variations of Hodge structures on $\mathcal{S}$ of Hodge type $%
\left\{ \left( -1,0\right) ,\left( 0,-1\right) \right\} $ with fibers in the
constant sheaf $L_{1,\mathbb{R},\mathcal{P}}$. The universal object is%
\begin{equation*}
\mathcal{L}_{1}:=\left( L_{1,\mathbb{R}},\mathcal{O}_{\mathbf{P}_{\mathbb{C}%
}^{1}}\left( -1\right) _{\mid \mathcal{P}}\hookrightarrow \mathcal{O}_{%
\mathbf{P}_{\mathbb{C}}^{1}}\left( L_{1}\right) _{\mid \mathcal{P}}\right) 
\text{.}
\end{equation*}

\bigskip

Let us fix $B$, an indefinite quaternion algebra, an identification $%
B_{\infty }\simeq \mathbf{M}_{2}\left( \mathbb{R}\right) $ and a lattice $%
I\subset B$ with right (resp. left)\ order $R\left( I\right) $ (resp. $%
E\left( I\right) $). Then $I,R\left( I\right) \subset \mathbf{M}_{2}\left( 
\mathbb{R}\right) $, $\mathbb{R}\otimes _{\mathbb{Z}}R\left( I\right) \simeq 
\mathbf{M}_{2}\left( \mathbb{R}\right) $ are identified as $\mathbb{R}$%
-algebras, $\mathbb{R}\otimes _{\mathbb{Z}}I\simeq \mathbf{M}_{2}\left( 
\mathbb{R}\right) $ as right $\mathbb{R}\otimes _{\mathbb{Z}}R\left(
I\right) \simeq \mathbf{M}_{2}\left( \mathbb{R}\right) $-modules and we also
have $\mathcal{O}_{\mathcal{S}}\left( I\right) \simeq \mathcal{O}_{\mathcal{S%
}}\left( \mathbf{M}_{2}\left( \mathbb{R}\right) \right) $ for every analytic
space $\mathcal{S}$. We will mainly regard $E\left( I\right) $ as the
endomorphism group of $I$ as a right $R\left( I\right) $-module\footnote{%
\bigskip An endomorphism $\varphi $ of $I$ as a right $R\left( I\right) $%
-module induces and endomorphism of $W=B$ as a right $B$-module. Recalling
the identification $f:B\otimes B\rightarrow End_{\mathbb{Q}}\left( W\right) $
obtained when discussing quaternionic objects (defined by $b_{1}\otimes
b_{2}\cdot x:=b_{1}xb_{2}^{\iota }$), it is easy to see that $\varphi $ is
induced by left multiplication by some element of $B$, which then belongs to 
$E\left( I\right) $.}. We will always with $R$ to denote a maximal order.

\begin{definition}
A quaternionic variation of Hodge structures on $\mathcal{S}$\ ($qVHS$ in
short)\ is a variation of Hodge structures of type $\left\{ \left(
-1,0\right) ,\left( 0,-1\right) \right\} $ with fibers in the constant
coherent sheaf $\mathbf{M}_{2}\left( \mathbb{R}\right) _{\mathcal{S}}$ such
that the action of $\mathbf{M}_{2}\left( \mathbb{R}\right) $ induced on $%
\mathcal{O}_{\mathcal{S}}\left( \mathbf{M}_{2}\left( \mathbb{R}\right)
\right) $ by right multiplication preserves the filtration $F^{\cdot }\left( 
\mathcal{O}_{\mathcal{S}}\left( \mathbf{M}_{2}\left( \mathbb{R}\right)
\right) \right) $ on $\mathcal{O}_{\mathcal{S}}\left( \mathbf{M}_{2}\left( 
\mathbb{R}\right) \right) $.

An $I$-rigidified quaternionic variation of Hodge structures on $\mathcal{S}$%
\ ($IqVHS$ in short)\ is a variation of Hodge structures with fibers in the
sheaf $I_{\mathcal{S}}$ such that $\left( \mathbb{R}\otimes _{\mathbb{Z}%
}I\right) _{\mathcal{S}}\simeq \mathbf{M}_{2}\left( \mathbb{R}\right) _{%
\mathcal{S}}$ gives by transport to the right hand side the structure of a
quaternionic variation of Hodge structures on $\mathcal{S}$.
\end{definition}

We note that, having fixed $I\subset \mathbf{M}_{2}\left( \mathbb{R}\right) $%
, to give a $qVHS$ is the same thing as to give an $IqVHS$.

\bigskip

As above, let $I\subset B$ be a lattice and let $O\subset B$ be any order.

\begin{definition}
\label{Def AFEC}A fake (analytic) $O$-elliptic curve over $\mathcal{S}$ ($%
f_{O}EC$ in short)\ is $\left( \mathcal{A}/\mathcal{S},i\right) $ where $%
\mathcal{A}/\mathcal{S}$ is an analytic abelian surface over $\mathcal{S}$%
\footnote{%
By analytic abelian surface over $\mathcal{S}$\ we mean a proper and flat
morphism $\pi :\mathcal{A}\rightarrow \mathcal{S}$ of analytic spaces of
relative dimension $2$, with a section $e:\mathcal{S}\rightarrow \mathcal{A}$
and a morphism $\mathcal{A}\times _{\mathcal{S}}\mathcal{A}\rightarrow 
\mathcal{A}$ satisfying the usual group constraints making the morphism $e$
a unit section.} and $i:O\rightarrow End_{\mathcal{S}}\left( \mathcal{A}%
\right) $ is a ring morphism (acting from the right on $\mathcal{A}/\mathcal{%
S}$).

An $I$-rigidified fake (analytic) elliptic curve over $\mathcal{S}$ ($IfEC$%
)\ is a $\left( \mathcal{A}/\mathcal{S},i,\rho \right) $ where $\left( 
\mathcal{A}/\mathcal{S},i\right) $ is an $R\left( I\right) $-fake (analytic)
elliptic curve over $\mathcal{S}$ and $\rho :R^{1}\pi _{\ast }\mathbb{Z}_{%
\mathcal{A}}^{\vee }\overset{\sim }{\rightarrow }I_{\mathcal{S}}$ is an
isomorphism as right sheaves of $R\left( I\right) $-modules (the action on
the left hand side is by functoriality).
\end{definition}

\begin{remark}
\label{Realizations R1}Although not needed in the sequel, we remark that the
proof of \cite[Ch. III, Proposition $\left( 1.5\right) $]{BC} applies in
this analytic setting, implying that a fake (analytic) $R$-elliptic curve $%
\mathcal{A}/\mathcal{S}$ always has a canonical principal polarization (see
loc.cit. for the precise conditions making the polarization canonical). In
particular, it is algebrizable when $\mathcal{S}$ is algebrizable.

Also, we remark that, if $I\simeq \mathbb{Z}^{4}$ is an $R$-module (left or
right), then $I\simeq R$: indeed, $\mathbb{Q}\otimes _{\mathbb{Z}}I\simeq B$
because $B$ is simple, so that we may view $I\subset B$ as an $R$-module;
but the class number of $B$ is one (by strong approximation), implying $%
I\simeq R$. For a fake (analytic) $R$-elliptic curve $\left( \mathcal{A}/%
\mathcal{S},i\right) $, we have $\left( R^{1}\pi _{\ast }\mathbb{Z}_{%
\mathcal{A}}^{\vee }\right) _{s}\simeq \mathbb{Z}^{4}$ for every $s\in 
\mathcal{S}$ by the (topological) proper base change theorem. Because the
left hand side is naturally a right $R$-module, we see that $R^{1}\pi _{\ast
}\mathbb{Z}_{\mathcal{A}}^{\vee }\simeq R_{\mathcal{S}}$ when $\mathcal{S}$
is simply connected.
\end{remark}

\bigskip

If we are given an $I$-rigidified fake (analytic) elliptic curve $\left(
\pi :\mathcal{A}\rightarrow \mathcal{S},i,\rho \right) $\ over $\mathcal{S}$%
, the exponential map gives an exact sequence of sheaves on $\mathcal{S}$,%
\begin{equation}
0\rightarrow R^{1}\pi _{\ast }\mathbb{Z}_{\mathcal{A}}^{\vee }\rightarrow T_{%
\mathcal{A}/\mathcal{S}}\rightarrow \mathcal{A}\rightarrow 0\text{.}
\label{Realizations D1}
\end{equation}%
Then we may define $F^{0}\left( \mathcal{O}_{\mathcal{S}}\left( R^{1}\pi
_{\ast }\mathbb{Z}_{\mathcal{A}}^{\vee }\right) \right) $ by means of the
exact sequence%
\begin{equation}
0\rightarrow F^{0}\left( \mathcal{O}_{\mathcal{S}}\left( R^{1}\pi _{\ast }%
\mathbb{Z}_{\mathcal{A}}^{\vee }\right) \right) \rightarrow \mathcal{O}_{%
\mathcal{S}}\left( R^{1}\pi _{\ast }\mathbb{Z}_{\mathcal{A}}^{\vee }\right)
\rightarrow T_{\mathcal{A}/\mathcal{S}}\rightarrow 0\text{.}
\label{Realizations D2}
\end{equation}%
By means of $i$ the ring $R\left( I\right) $ acts on this sequence (say from
the right). The rigidification yields $\rho :R^{1}\pi _{\ast }\mathbb{Z}_{%
\mathcal{A}}^{\vee }\simeq I_{\mathcal{S}}$ compatible with the right action
of $R\left( I\right) $ on $I_{\mathcal{S}}$. It follows that $F^{0}\left( 
\mathcal{O}_{\mathcal{S}}\left( R^{1}\pi _{\ast }\mathbb{Z}_{\mathcal{A}%
}^{\vee }\right) \right) \simeq F^{0}\left( \mathcal{O}_{\mathcal{S}}\left(
I\right) \right) \simeq F^{0}\left( \mathcal{O}_{\mathcal{S}}\left( \mathbf{M%
}_{2}\left( \mathbb{R}\right) \right) \right) $ (the right hand sides
defined by transport)\ gives $I_{\mathcal{S}}$ a rigidified quaternionic
variation of Hodge structures on $\mathcal{S}$ that we denote $R^{1}\pi
_{\ast }\mathbb{Z}_{\mathcal{A}}^{\vee }$. The correspondence is indeed an
equivalence of categories: when $\mathcal{S}=S^{an}$ for a complex algebraic
variety $S$, this is an application of \cite[Theorem 7.13]{Mi} or \cite[4.4.3%
]{De2}; the reference \cite[Proposition $\left( 2.2\right) $ $\left(
ii\right) $]{De1} contains (without proof) the analogous statement in the
case of analytic elliptic curves $\mathcal{E}$ over analytic spaces and our
case $B=\mathbf{M}_{2}\left( \mathbb{Q}\right) $ and $R\left( I\right) =%
\mathbf{M}_{2}\left( \mathbb{Z}\right) $ follows writing $\mathcal{A}=%
\mathcal{E}^{2}$ and splitting the rigidified quaternionic variation of
Hodge structures in a similar way. The above general result is shown in the
proof of Proposition \ref{Realizations P1} below (and a posteriori
equivalent to it). See also \cite[Theorem 3]{Sh} for related results.

\bigskip

The following result yields a quaternionic moduli description of $\mathcal{P}
$, which depends on the fixed identification $B_{\infty }\simeq \mathbf{M}%
_{2}\left( \mathbb{R}\right) $ and the choice of a lattice $I\subset B$.

\begin{proposition}
\label{Realizations P1}The analytic space $\mathcal{P}$ classifies $I$%
-rigidified fake (analytic) elliptic curve over $\mathcal{S}$. If $\left(
\pi _{I}:\mathcal{A}_{I}\rightarrow \mathcal{P},i_{I},\rho _{I}\right) $ is
the universal fake elliptic curve, we have that $R^{1}\pi _{I\ast }\mathbb{Z}%
_{\mathcal{A}_{I}}^{\vee }=I_{\mathcal{P}}$ and the associated variation of
Hodge structures on $\mathcal{O}_{\mathcal{S}}\left( I\right) \simeq 
\mathcal{O}_{\mathcal{S}}\left( \mathbf{M}_{2}\left( \mathbb{R}\right)
\right) $ is given by $\mathcal{L}_{1}\oplus \mathcal{L}_{1}$.
\end{proposition}

\begin{proof}
Let us remark that, if we may apply \cite[Theorem 7.13]{Mi} or \cite[4.4.3]%
{De2}, the above discussion would imply that we have just to classify $I$%
-rigidified quaternionic variation of Hodge structures, rather than $I$%
-rigidified fake (analytic) elliptic curves. Rather, we begin in the
opposite direction.

\emph{Step 1: the analytic space }$\mathcal{P}$\emph{\ classifies }$IqVHS$%
\emph{, with universal object as described above.} As remarked above, the
choice $B_{\infty }\simeq \mathbf{M}_{2}\left( \mathbb{R}\right) $
determines $I\subset \mathbf{M}_{2}\left( \mathbb{R}\right) $ identifying
the data of $IqVHS$ and $qVHS$. Hence we classify $qVHS$. But it is easy to
see that the association $\mathcal{L\mapsto L}\oplus \mathcal{L}$ realizes
an identification between variations of Hodge structures on $\mathcal{S}$ of
Hodge type $\left\{ \left( -1,0\right) ,\left( 0,-1\right) \right\} $ with
fibers in the constant coherent sheaf $L_{1,\mathbb{R}}$ and $qVHS$. The
claim follows from the above description of $\mathcal{P}$ and its universal
object.

\emph{Step 2: there is an }$IfEC$\emph{\ with prescribed associated }$IqVHS$%
\emph{. }The proof will be given in \S \ref{Subsection Universal fake
elliptic curve} point $\left( 2\right) $\ below.

\emph{Step 3: the analytic space }$\mathcal{P}$\emph{\ classifies rigidified
fake (analytic) elliptic curves.} First, we remark that the association $%
\left( \pi :\mathcal{A}\rightarrow \mathcal{S}\right) \leadsto R^{1}\pi
_{\ast }\mathbb{Z}_{\mathcal{A}}^{\vee }$ is fully faithful thanks to the
above exact sequences $\left( \text{\ref{Realizations D1}}\right) $ and $%
\left( \text{\ref{Realizations D2}}\right) $, as remarked in the proof of 
\cite[Theorem 7.13]{Mi}. As explained above, the added data $i$ (resp. $\rho 
$) corresponds to giving the structure of an $R\left( I\right) $-module
object (resp. an $I$-rigidification). It follows that $\left( \pi :\mathcal{A%
}\rightarrow \mathcal{S},i,\rho \right) \leadsto R^{1}\pi _{\ast }\mathbb{Z}%
_{\mathcal{A}}^{\vee }$ is fully faithful, valued in $IqVHS$ on $\mathcal{S}$%
. Let us now show that this latter functor is essentially surjective; in
view of Step 1, this is equivalent to our claim (and the proof will directly
show the versal part of the statement about $\mathcal{P}$). Hence, suppose
we are given $\mathbf{x}$, an $IqVHS$ on $\mathcal{S}$; thanks to Step 1,
it gives rise to a morphism of analytic spaces $x:\mathcal{S}\rightarrow 
\mathcal{P}$ and we have $\mathbf{x}=x^{\ast }\mathbf{u}$, where $\mathbf{u}$
is the universal $IqVHS$ on $\mathcal{P}$ and $x^{\ast }$ denotes the
pull-back of variations of Hodge structures. According to Step 2, we have $%
\mathbf{u}=R^{1}\pi _{\ast }\mathbb{Z}_{\mathcal{A}}^{\vee }$. Let us remark
that, by the (topological) proper base change theorem, we have $x^{\ast
}R^{1}\pi _{\ast }\mathbb{Z}_{\mathcal{A}}^{\vee }=R^{1}\pi _{\ast }\mathbb{Z%
}_{\mathcal{A}\times _{\mathcal{P}}\mathcal{S}}^{\vee }$ as variations of
Hodge structures, i.e. base change commutes with the formation of the
variation of Hodge structure $R^{1}\pi _{\ast }\mathbb{Z}_{\mathcal{A}%
}^{\vee }$: indeed, the formation of the tangent spaces $T_{\mathcal{A}/%
\mathcal{S}}$\ commutes with base changes, as it follows from $\left( \text{%
\ref{Realizations D1}}\right) $ and the fact that the formation of $R^{1}\pi
_{\ast }\mathbb{Z}_{\mathcal{A}}^{\vee }$ and $\mathcal{A}$ commute with
base changes; then, because the formation of $T_{\mathcal{A}/\mathcal{S}}$
and $\mathcal{O}_{\mathcal{S}}\left( R^{1}\pi _{\ast }\mathbb{Z}_{\mathcal{A}%
}^{\vee }\right) $ commute with base changes (the latter by definition of
pull-back of variations of Hodge structures, once again because the
formation of $R^{1}\pi _{\ast }\mathbb{Z}_{\mathcal{A}}^{\vee }$ commutes
with base changes), we see that the formation of the whole $\left( \text{\ref%
{Realizations D2}}\right) $ commutes with base changes. Summarizing, $%
\mathbf{x}=R^{1}\pi _{\ast }\mathbb{Z}_{\mathcal{A}\times _{\mathcal{P}}%
\mathcal{S}}^{\vee }$ as variations of Hodge structures, as wanted.
\end{proof}

\begin{remark}
\label{Realizations R2}Having fixed $B_{\infty }\simeq \mathbf{M}_{2}\left( 
\mathbb{R}\right) $, Proposition \ref{Realizations P1} implies that we have
that $R^{1}\pi _{I\ast }\mathbb{Q}_{\mathcal{A}_{I}}^{\vee }=B_{\mathcal{P}}$
with associated variation of Hodge structures on $\mathcal{O}_{\mathcal{S}%
}\left( B\right) \simeq \mathcal{O}_{\mathcal{S}}\left( \mathbf{M}_{2}\left( 
\mathbb{R}\right) \right) $ given by $\mathcal{L}_{1}\oplus \mathcal{L}_{1}$%
. Hence the underlying rational variation of Hodge structures $R^{1}\pi
_{\ast }\mathbb{Q}_{\mathcal{A}}^{\vee }:=R^{1}\pi _{I\ast }\mathbb{Q}_{%
\mathcal{A}_{I}}^{\vee }$ does not depend on the choice of the lattice $%
I\subset B$: it is endowed with a natural left $B$-action, extending the
left $E\left( I\right) $-action and acting on the quaternionic structure
(given by right multiplication).
\end{remark}

\subsection{Linear algebra in the category of $B^{\times }$-representations}

We write $x\mapsto x^{\iota }$ to denote the main involution of $B$, so that 
$x+x^{\iota }=\limfunc{Tr}\left( x\right) $ and $xx^{\iota }=\limfunc{Nr}%
\left( x\right) $. We let $B^{\times }$ acts on $B$ by left multiplication,
while we write $B^{\iota }$ to denote $B$ on which $B^{\times }$ acts from
the left by the rule $b\cdot x:=bxb^{\iota }$. We write $B^{0}:=\ker \left( 
\limfunc{Tr}\right) $ to denote the trace zero elements, viewed as a $%
B^{\times }$-subrepresentation of $B^{\iota }$ (indeed $\limfunc{Tr}\left(
bxb^{\iota }\right) =\limfunc{Nr}\left( b\right) \limfunc{Tr}\left( x\right) 
$). If $V\in \mathrm{R}$\textrm{ep}$\left( B^{\times }\right) $ and $r\in 
\mathbb{Z}$, we let $V\left( r\right) $ be $V$ on which $B^{\times }$ acts
by $b\cdot _{r}v=\limfunc{Nr}\nolimits^{-r}\left( b\right) bv$, so that $%
V\left( r\right) =V\otimes \mathbb{Q}\left( r\right) $ (canonically). We let 
$B_{+}^{\times }\subset B^{\times }$ be the subgroup of elements having
positive norm.

In \cite{JL} certain Laplace and Dirac operators has been defined with
source and target those of the subsequent Lemma \ref{Realizations L1}. While
their definition is completely explicit, it is only the definition of the
Laplace operator that readily generalizes to arbitrary tensor categories; on
the other hand, the definition of the Dirac operator requires the theory we
have developed in order to provide good models for their kernels which have
a general meaning for tensor categories. Indeed, we have the following key
remark, that allows us to replace the Jordan-Livn\'{e} models with ours,
whose proof is left to the reader.

\begin{lemma}
\label{Realizations L1}Let%
\begin{equation*}
f_{n}:\limfunc{Sym}\nolimits^{n}\left( B_{0}\right) \rightarrow \limfunc{Sym}%
\nolimits^{n-2}\left( B_{0}\right) \left( -2\right) \text{ and }g_{n}:%
\limfunc{Sym}\nolimits^{n}\left( B_{0}\right) \otimes B\rightarrow \limfunc{%
Sym}\nolimits^{n-1}\left( B_{0}\right) \otimes B\left( -1\right)
\end{equation*}%
be any epimorphism in $\limfunc{Rep}\nolimits_{\mathbb{Q}}\left( B^{\times
}\right) $. Once we fix $B\otimes \mathbb{F}\simeq \mathbf{M}_{2}\left( 
\mathbb{F}\right) $, where $\mathbb{F}$ is any splitting field of $B$, there
are canonical isomorphisms%
\begin{equation*}
\ker \left( f_{n}\right) \otimes \mathbb{F}\simeq L_{2n}\otimes \mathbb{F}%
\text{, }\ker \left( g_{n}\right) \otimes \mathbb{F}\simeq
L_{2n+1}^{2}\otimes \mathbb{F}
\end{equation*}%
which are compatible with the $\left( B\otimes \mathbb{F}\right) ^{\times }$%
-action on the left hand side, the $\mathbf{GL}_{2}\left( \mathbb{F}\right) $%
-action on the right side and the induced identification $\left( B\otimes 
\mathbb{F}\right) ^{\times }\simeq \mathbf{GL}_{2}\left( \mathbb{F}\right) $.
\end{lemma}

The following Lemma will be useful. Recall that, if $M$ is an object in a
pseudo abelian $\mathbb{Q}$-linear category on which $B$-acts, we may write $%
\wedge ^{2}M=\left( \wedge ^{2}M\right) _{+}\oplus \left( \wedge
^{2}M\right) _{-}$ canonically, where $B$ operates on $\left( \wedge
^{2}M\right) _{-}$ via the reduced norm. The following

\begin{lemma}
\label{Realizations L2}Let $\theta :B\overset{\sim }{\rightarrow }End_{%
\mathrm{R}\text{\textrm{ep}}\left( B^{\times }\right) }\left( B\right) $ be
the isomorphism provided by the right multiplication. Then we have, in $%
\limfunc{Rep}\nolimits_{\mathbb{Q}}\left( B^{\times }\right) $,%
\begin{equation*}
\wedge ^{2}B=\left( \wedge ^{2}B\right) _{+}\oplus \left( \wedge
^{2}B\right) _{-}\text{ with }\left( \wedge ^{2}B\right) _{+}\simeq \mathbb{Q%
}\left( -1\right) ^{3}\text{ and }\left( \wedge ^{2}B\right) _{-}\simeq B_{0}%
\text{.}
\end{equation*}
\end{lemma}

Since $\left( B,\theta \right) $ is an \emph{alternating} quaternionic
object, we may define%
\begin{equation*}
L_{2n}^{B}:=M_{2n}\left( B,\theta \right) \text{ (for }n\geq 1\text{)\ and }%
L_{2n+1}^{B\left( 2\right) }:=M_{2n+1}\left( B,\theta \right) \text{ (for }%
n\geq 0\text{)}
\end{equation*}%
and it is a consequence of Lemmas \ref{Realizations L1} and \ref%
{Realizations L2} that, when $B\otimes \mathbb{F}\simeq \mathbf{M}_{2}\left( 
\mathbb{F}\right) $,%
\begin{equation}
L_{2n,\mathbb{F}}^{B}\simeq L_{2n,\mathbb{F}}\text{ and }L_{2n+1,\mathbb{F}%
}^{B\left( 2\right) }\simeq L_{2r+1,\mathbb{F}}^{2}\text{.}
\label{Realizations F splitting}
\end{equation}

\subsection{Variations of Hodge structures attached to $B^{\times }$%
-representations}

In this subsection we define a $\mathbb{Q}$-additive $ACU$ tensor functor,
depending on the choice of an identification $B_{\infty }\simeq \mathbf{M}%
_{2}\left( \mathbb{R}\right) $,%
\begin{equation*}
\mathcal{L}:\mathrm{R}\text{\textrm{ep}}\left( B^{\times }\right)
\rightarrow \mathbf{VHS}_{\mathcal{P}}\left( \mathbb{Q}\right) \text{,}
\end{equation*}%
where $\mathbf{VHS}_{\mathcal{S}}\left( \mathbb{F}\right) $ denotes the
category of variations of Hodge structures on $\mathcal{S}$ with
coefficients in the field $\mathbb{F}\subset \mathbb{R}$. The identification 
$B_{\infty }\simeq \mathbf{M}_{2}\left( \mathbb{R}\right) $ induces $\mathrm{%
R}$\textrm{ep}$\left( B_{\infty }^{\times }\right) \simeq \mathrm{R}$\textrm{%
ep}$\left( \mathbf{GL}_{2,\mathbb{R}}\right) $ and it follows from \cite[%
Corollary 3.2 and its proof for uniqueness]{Ha} that we may define a (unique
up to equivalence) faithful and exact $\mathbb{Q}$-additive $ACU$ tensor
functor%
\begin{equation*}
\mathcal{L}_{\mathbb{R}}:\mathrm{R}\text{\textrm{ep}}\left( B_{\infty
}^{\times }\right) \rightarrow \mathbf{VHS}_{\mathcal{P}}\left( \mathbb{R}%
\right)
\end{equation*}%
requiring $\mathcal{L}_{\mathbb{R}}\left( L_{1,\mathbb{R}}\right) :=\mathcal{%
L}_{1}$. Since $\mathcal{O}_{\mathcal{P}}\left( V\right) =\mathcal{O}_{%
\mathcal{P}}\left( V_{\mathbb{R}}\right) $ for every $V\in \mathrm{R}$%
\textrm{ep}$\left( B^{\times }\right) $, we deduce that the restriction of $%
\mathcal{L}_{\mathbb{R}}$ to $\mathrm{R}$\textrm{ep}$\left( B^{\times
}\right) \rightarrow \mathrm{R}$\textrm{ep}$\left( B_{\infty }^{\times
}\right) $ via $V\mapsto V_{\mathbb{R}}\ $factors through $\mathbf{VHS}_{%
\mathcal{P}}\left( \mathbb{Q}\right) \rightarrow \mathbf{VHS}_{\mathcal{P}%
}\left( \mathbb{R}\right) $ (again via scalar extension)\ and gives our $%
\mathcal{L}$. It follows from this description and Proposition \ref%
{Realizations P1} (see Remark \ref{Realizations R2})\ that we have%
\begin{equation}
\mathcal{L}\left( B\right) =R^{1}\pi _{\ast }\mathbb{Q}_{\mathcal{A}}^{\vee }%
\text{,}  \label{Realizations F key}
\end{equation}%
if $B$ denotes the left $B^{\times }$-representation whose underlying vector
space is $B$ with the action given by left multiplication and quaternionic
structure induced by the right multiplication $\theta $.

Since $\left( \mathcal{L}\left( B\right) ,\mathcal{L}\left( \theta \right)
\right) $ is an \emph{alternating} quaternionic object, we may define%
\begin{equation*}
\mathcal{L}_{2n}^{B}:=M_{2n}\left( \mathcal{L}\left( B\right) ,\mathcal{L}%
\left( \theta \right) \right) \text{ (for }n\geq 1\text{)\ and }\mathcal{L}%
_{2n+1}^{B\left( 2\right) }:=M_{2n+1}\left( \mathcal{L}\left( B\right) ,%
\mathcal{L}\left( \theta \right) \right) \text{ (for }n\geq 0\text{).}
\end{equation*}

\bigskip

If $K\subset \widehat{B}^{\times }$ is an open and compact subgroup, we may
consider the Shimura curve%
\begin{equation*}
S_{K}\left( \mathbb{C}\right) :=B^{\times }\backslash \left( \mathcal{P}%
\times \widehat{B}^{\times }\right) /K=B_{+}^{\times }\backslash \left( 
\mathcal{H}\times \widehat{B}^{\times }\right) /K
\end{equation*}%
where:

\begin{itemize}
\item $B^{\times }$ acts diagonally on $\mathcal{P}\times \widehat{B}%
^{\times }$ (via $B^{\times }\subset B_{\infty }^{\times }$ and the diagonal
embedding $B^{\times }\subset \widehat{B}^{\times }$ on the second component)

\item The action of $K$ is trivial on $\mathcal{P}$ and by right
multiplication on $\widehat{B}^{\times }$.
\end{itemize}

When $B\neq \mathbf{M}_{2}\left( \mathbb{Q}\right) $, $X_{K}\left( \mathbb{C}%
\right) :=S_{K}\left( \mathbb{C}\right) $ is compact and otherwise we set $%
X_{K}\left( \mathbb{C}\right) :=\overline{S_{K}\left( \mathbb{C}\right) }$,
compactified by "adding cusps". Then $\mathcal{L}\left( V\right) $ (for any $%
V\in \mathrm{R}$\textrm{ep}$\left( B^{\times }\right) $)\ descend to a
variation of Hodge structures $\mathcal{L}_{K}\left( V\right) $\ on $%
S_{K}\left( \mathbb{C}\right) $\footnote{%
Suffices indeed to check that $\mathcal{L}_{1}$ descend to $S_{K}\left( 
\mathbb{C}\right) $ in order to get a functor $\mathcal{L}_{\mathbb{R},K}$
(from \cite{Ha}) valued in $\mathbf{VHS}_{S_{K}\left( \mathbb{C}\right)
}\left( \mathbb{R}\right) $ and then appeal to the uniqueness to deduce that 
$\mathcal{L}_{\mathbb{R},K}\left( V\right) $ is obtained from $\mathcal{L}_{%
\mathbb{R}}\left( V\right) $ by descend for every $V$. Then one can promote
the restriction of $\mathcal{L}_{\mathbb{R},K}$ to $\mathrm{R}$\textrm{ep}$%
\left( B^{\times }\right) $ to take values in $\mathbf{VHS}_{S_{K}\left( 
\mathbb{C}\right) }\left( \mathbb{Q}\right) $, exactly as above.}.

Setting $\pi _{0}\left( S_{K}\left( \mathbb{C}\right) \right) :=B^{\times
}\backslash \widehat{B}^{\times }/K$ we have%
\begin{equation*}
\pi _{0}:S_{K}\left( \mathbb{C}\right) \twoheadrightarrow \pi _{0}\left(
S_{K}\left( \mathbb{C}\right) \right) \text{ and }\pi _{K}:\widehat{B}%
^{\times }\rightarrow \pi _{0}\left( S_{K}\left( \mathbb{C}\right) \right)
\end{equation*}%
If $x\in \widehat{B}^{\times }$ define $\Gamma _{K}\left( x\right)
:=xKx^{-1}\cap B^{\times }$ (resp. $\Gamma _{K}\left( x\right)
_{+}:=xKx^{-1}\cap B_{+}^{\times }$), where $B^{\times }\subset \widehat{B}%
^{\times }$ is diagonally embedded, that we view as a subgroup $\Gamma
_{K}\left( x\right) \subset B^{\times }\subset B_{\infty }^{\times }=\mathbf{%
GL}_{2}\left( \mathbb{R}\right) $. We have the mutually inverse bijections%
\begin{equation*}
p_{x}:\pi _{0}^{-1}\left( \pi _{K}\left( x\right) \right) =B^{\times
}\backslash B^{\times }\left( \mathcal{P}\times xK/K\right) \overset{\sim }{%
\rightarrow }\Gamma _{K}\left( x\right) \backslash \mathcal{P}\text{ and }%
\iota _{x}:\Gamma _{K}\left( x\right) \backslash \mathcal{P}\overset{\sim }{%
\rightarrow }\pi _{0}^{-1}\left( \pi _{K}\left( x\right) \right)
\end{equation*}%
defined by the rules%
\begin{equation*}
\left[ \tau ,xk\right] \mapsto \left[ \tau \right] \text{ and }\left[ \tau %
\right] \mapsto \left[ \tau ,x\right]
\end{equation*}%
Then%
\begin{equation}
S_{K}\left( \mathbb{C}\right) =\tbigsqcup\nolimits_{\pi _{K}\left( x\right)
\in \pi _{0}\left( X_{K}\left( \mathbb{C}\right) \right) }\pi
_{0}^{-1}\left( \pi _{K}\left( x\right) \right) \simeq
\tbigsqcup\nolimits_{\pi _{K}\left( x\right) \in \pi _{0}\left( X_{K}\left( 
\mathbb{C}\right) \right) }\Gamma _{K}\left( x\right) \backslash \mathcal{P}%
\text{ and }\Gamma _{K}\left( x\right) _{+}\backslash \mathcal{H}\overset{%
\sim }{\rightarrow }\Gamma _{K}\left( x\right) \backslash \mathcal{P}
\label{Realizations F decomposition}
\end{equation}%
It follows from the Eichler-Shimura isomorphism (see \cite[Ch. 6]{Hi} and 
\cite[\S 3.2]{GSS} or \cite[\S 2.4]{RS} for the statement in the
quaternionic setting) and $\left( \text{\ref{Realizations F splitting}}%
\right) $ that the cohomology groups (let $\left( ?\right) =\phi $ when $k$
is even and $\left( ?\right) =\left( 2\right) $ when $k$ is odd)%
\begin{equation}
H^{1}\left( S_{K}\left( \mathbb{C}\right) ,\mathcal{L}_{k,K}^{B\left(
?\right) }\right) \simeq \tbigoplus\nolimits_{\pi _{K}\left( x\right) \in
\pi _{0}\left( X_{K}\left( \mathbb{C}\right) \right) }H^{1}\left( \Gamma
_{K}\left( x\right) ,L_{k}^{B\left( ?\right) }\right) \text{,}
\label{Realizations F cohomology}
\end{equation}%
afford weight $k+2$ modular forms of level $K$ when $k$ is \emph{even} and
two copies of them when $k$ is \emph{odd}. Indeed, it is not difficult to
define Hecke operators on the family $\left\{ \mathcal{L}_{k,K}^{B\left(
?\right) }\right\} _{K}$ by means of correspondences, which are given by
double cosets on the right hand side; we also remark that the left hand side
has a natural Hodge structure endowed with Hecke multiplication. We remark
that, when $B\neq \mathbf{M}_{2}\left( \mathbb{Q}\right) $ the Hecke action
on $\left( \text{\ref{Realizations F cohomology}}\right) $ is purely
cuspidal, whereas in case $B=\mathbf{M}_{2}\left( \mathbb{Q}\right) $ the
Hecke action factors $\left( \text{\ref{Realizations F cohomology}}\right) $
as the direct sum of its cuspidal and Eisenstin part. It is a non-trivial
task to lift this decomposition at the motivic level in order to single out
a motive of cuspidal modular forms: this is done in \cite{Sc} and we will
not touch this problem.

\subsection{The motives of quaternionic modular forms and their realizations}

Let $K\subset \widehat{B}^{\times }$ be an open and compact subgroup which
is small enough so that $S_{K}$ is a fine moduli space and let $\pi
_{K}:A_{K}\rightarrow S_{K}$ be the universal level $K$ fake elliptic curve
over $S_{K}$. Consider the relative motive $h(A_{K})$ as an object of $%
\mathbf{Mot}_{+}^{0}(S_{K},\mathbb{F})$, where $h$ is the contravariant
functor 
\begin{equation*}
h\colon \mathbf{Sch}(S_{K})\rightarrow \mathbf{Mot}_{+}^{0}(S_{K},\mathbb{F})
\end{equation*}%
from the category of smooth and projective schemes over $S$ to the category
of Chow motives with coefficients in a field $\mathbb{F}$. By functoriality
of the motivic decomposition, there is $\theta :B\rightarrow End\left(
h^{1}(A_{K})\right) $ making $\left( h^{1}(A_{K}),\theta \right) $ a \emph{%
symmetric} quaternionic object, and we may define%
\begin{equation*}
M_{2n,K}^{B}:=M_{2n}\left( h^{1}(A_{K}),\theta \right) \text{ (for }n\geq 1%
\text{)\ and }M_{2n+1,K}^{B\left( 2\right) }:=M_{2n+1}\left(
h^{1}(A_{K}),\theta \right) \text{ (for }n\geq 0\text{)}
\end{equation*}%
There is a realization functor%
\begin{equation*}
R_{S_{K}}:\mathbf{Mot}_{+}^{0}(S_{K},\mathbb{F})\rightarrow D^{b}\left( 
\mathbf{VMHS}\left( S_{K},\mathbb{F}\right) \right)
\end{equation*}%
extending the correspondence mapping $\pi :X\rightarrow S_{K}$ to $R\pi
_{\ast }\mathbb{F}_{X}^{\vee }$. Here $\mathbf{VMHS}(S_K,\mathbb{F})$
denotes the abelian category of variations of mixed Hodge structures over $%
S_K$ with coefficients in $\mathbb{F}$. See~\cite[14.4]{PS} for details.

\begin{theorem}
  \label{thm:realizations}
Taking $F=\mathbb{Q}$ we have the following realizations.

\begin{itemize}
\item[$\left( 1\right) $] Suppose that $2n\geq 2$ is even. Then:%
\begin{equation*}
R(M_{2n,K}^{B})=\mathcal{L}_{2n,K}^{B}[-2n]\text{.}
\end{equation*}

\item[$\left( 2\right) $] Suppose that $2n+1\geq 3$ is odd. Then:%
\begin{equation*}
R(M_{2n+1,K}^{B\left( 2\right) })=\mathcal{L}_{2n+1,K}^{B\left( 2\right)
}[-(2n+1)].
\end{equation*}
\end{itemize}
\end{theorem}

\begin{proof}
As in \cite[Remarks 2) after Corollary 3.2]{DM} one has $R\left(
h^{1}(A_{K})\right) =R^{1}\pi _{K\ast }\mathbb{Q}_{A_{K}}^{\vee }\left[ -1%
\right] $ and $R^{1}\pi _{K\ast }\mathbb{Q}_{A_{K}}^{\vee }$ is obtained by
descending the sheaf $R^{1}\pi _{\ast }\mathbb{Q}_{\mathcal{A}}^{\vee }$
(see Proposition \ref{Realizations UfEC L1} below for details), so that $\left( 
\text{\ref{Realizations F key}}\right) $ implies $R\left(
h^{1}(A_{K})\right) =\mathcal{L}_{K}\left( B\right) \left[ -1\right] $.
Since $R_{S_{K}}$ is a $AU$ tensor functor (indeed anti-commutative, see 
\cite[Remark $\left( 2.6.1\right) $]{Ku}), we deduce from the remark after
Definition \ref{Def quaternionic} that (let $\left( ?\right) =\phi $ when $k$
is even and $\left( ?\right) =\left( 2\right) $ when $k$ is odd)%
\begin{equation*}
R(M_{k,K}^{B\left( ?\right) })=M_{k}\left( \mathcal{L}_{K}\left( B\right) %
\left[ -1\right] ,\mathcal{L}_{K}\left( \theta \right) \right) =M_{k}\left( 
\mathcal{L}_{K}\left( B\right) ,\mathcal{L}_{K}\left( \theta \right) \right) %
\left[ -k\right] =\mathcal{L}_{k,K}^{B\left( ?\right) }[-k]\text{.}
\end{equation*}
\end{proof}

Together with $\left( \text{\ref{Realizations F cohomology}}\right) $ and
recalling that the group cohomology of $L_{k}$ is concentrated in degree $1$%
, we deduce the following result.

\begin{corollary}
Let $H$ be the Betti realization functor, valued in $\mathbf{VHS}\left( 
\mathbb{Q}\right) $, and let view $M_{2n,K}^{B}$ and $M_{2n+1,K}^{B\left(
2\right) }$ as motives defined over $\mathbb{Q}$.

\begin{itemize}
\item[$\left( 1\right) $] Suppose that $2n\geq 2$ is even. Then:%
\begin{equation*}
H^{i}(M_{2n,K}^{B})=\left\{ 
\begin{array}{ll}
H^{1}\left( S_{K}\left( \mathbb{C}\right) ,\mathcal{L}_{k,K}^{B\left(
?\right) }\right) & \text{if }i=2n+1 \\ 
0 & \text{otherwise.}%
\end{array}%
\right.
\end{equation*}

\item[$\left( 2\right) $] Suppose that $2n+1\geq 3$ is odd. Then:%
\begin{equation*}
H^{i}(M_{2n+1,K}^{B\left( 2\right) })=\left\{ 
\begin{array}{ll}
H^{1}\left( S_{K}\left( \mathbb{C}\right) ,\mathcal{L}_{k,K}^{B\left(
?\right) }\right) & \text{if }i=2n+2 \\ 
0 & \text{otherwise.}%
\end{array}%
\right.
\end{equation*}
\end{itemize}
\end{corollary}

As explained after $\left( \text{\ref{Realizations F cohomology}}\right) $,
this motivates our designation of $M_{2n,K}^{B}$ (resp. $M_{2n+1,K}^{B\left(
2\right) }$) as the motive of (resp. two copies of)\ level $K$ and weight $%
2n+2$ (resp. $2n+3$) modular forms: indeed the functoriality of our
construction implies that the Hecke correspondences induces a Hecke
multiplication on $M_{2n,K}^{B}$ (resp. $M_{2n+1,K}^{B\left( 2\right) }$),
which is compatible with that on the realizations. As remarked after $\left( 
\text{\ref{Realizations F cohomology}}\right) $, in case $B=\mathbf{M}%
_{2}\left( \mathbb{Q}\right) $ Scholl has been able to single out a motive
of cuspidal modular forms $M_{m,K}^{\mathrm{cusp}}$. The concrete
construction of its motive is actually different, as we always start the
game with two copies of the universal elliptic curve. Its method is finer
even on the open modular curve: it gives $M_{m,K}$ such that $M_{m,K}$ has
the same realization of $M_{2n,K}^{B}$ when $m=2n$ and realize one of the
two copies of the realization of $M_{2n+1,K}^{B\left( 2\right) }$ in case $%
m=2n+1$. The abstract approach employed here for computing the realizations,
inspired by \cite{IS}, easily adapts to the other realizations: one has only
to appropriately replace $\left( \text{\ref{Realizations F key}}\right) $
(which is, for example, the deeper \cite[Lemma 5.10]{IS} in the $p$-adic
realm considered there).

\subsection{\label{Subsection Universal fake elliptic curve}Final remarks}

In this \S , we collect basic facts that are surely well-known to experts
about analytic families of fake elliptic curves, mainly due to Shimura,
following the point of view of \cite{De1} in the case of modular curves (see
also \cite[\S 6.1]{C}). This will allow us to finish the proof of Proposition~\ref{Realizations P1},
as well as of Theorem~\ref{thm:realizations} thanks to the following result, whose
proof will be given at the very end of the section.

\begin{proposition}
\label{Realizations UfEC L1}If $K$ is small enough, then $R^{1}\pi _{K\ast }%
\mathbb{Q}_{A_{K}}^{\vee }$ is obtained by descending the sheaf $R^{1}\pi
_{\ast }\mathbb{Q}_{\mathcal{A}}^{\vee }$.
\end{proposition}

Recall our fixed identification $B_{\infty }\simeq 
\mathbf{M}_{2}\left( \mathbb{R}\right) $ and note that, for every $\tau \in 
\mathcal{P}$, we may identify (see \cite[Prop. 14]{Sh} or \cite[Lemma 1.6]{P}):
\begin{equation*}
\Phi _{\tau }:B_{\infty }\simeq \mathbf{M}_{2}\left( \mathbb{R}\right) 
\overset{\sim }{\rightarrow }\mathbb{C}^{2}\text{, via }w\mapsto w^{\iota
}\left( 
\begin{array}{c}
\tau \\ 
1%
\end{array}%
\right) \text{,}
\end{equation*}%
thus getting a $\mathcal{C}^{\infty }$-morphism
\begin{equation*}
\Phi :\mathcal{P}\times B_{\infty }\overset{\sim }{\rightarrow }\mathcal{P}%
\times \mathbb{C}^{2}\text{, via }\left( \tau ,w\right) \mapsto \left( \tau
,\Phi _{\tau }\left( w\right) \right) \text{.}
\end{equation*}%
If $\beta =\left( 
\begin{array}{cc}
a & b \\ 
c & d%
\end{array}%
\right) $, then we define $j\left( \beta ,\tau \right) =c\tau +d$. We remark
the formula
\begin{equation}
j\left( \beta ^{\iota },\tau \right) \Phi _{\beta ^{\iota }\tau }\left(
w\right) =\Phi _{\tau }\left( \beta w\right)  \label{Realizations UfEC F1}
\end{equation}

Suppose that we are given a left $B_{\infty }^{\times }$-representation $V$
and let us consider $B_{\infty }^{\times }\ltimes V$, the multiplication
being defined by the rule $\left( \beta _{1},v_{1}\right) \left( \beta
_{2},v_{2}\right) :=\left( \beta _{1}\beta _{2},\beta _{1}v_{2}+v_{1}\right) 
$. We define a left action of $B_{\infty }^{\times }\ltimes V$ on $\mathcal{P%
}\times V$ by the rule%
\begin{eqnarray}
&&\left( B_{\infty }^{\times }\ltimes V\right) \times \left( \mathcal{P}%
\times V\right) \rightarrow \mathcal{P}\times V  \notag \\
&&\left( \beta ,b\right) \cdot \left( \tau ,w\right) :=\left( \beta \tau
,\beta w+b\right) \text{.}  \label{Realizations UfEC F2}
\end{eqnarray}%
On the other hand, we define a left action of $B_{\infty }^{\times }\ltimes
B_{\infty }$ on $\mathcal{P}\times \mathbb{C}^{2}$ by the rule%
\begin{eqnarray}
&&\left( B_{\infty }^{\times }\ltimes B_{\infty }\right) \times \left( 
\mathcal{P}\times \mathbb{C}^{2}\right) \rightarrow \mathcal{P}\times 
\mathbb{C}^{2}  \notag \\
&&\left( \beta ,b\right) \cdot \left( \tau ,w\right) :=\left( \beta \tau ,%
\frac{\det \left( \beta \right) }{j\left( \beta ,\tau \right) }\left(
w+b^{\iota }\beta ^{-\iota }\left( 
\begin{array}{c}
\tau \\ 
1%
\end{array}%
\right) \right) \right)  \label{Realizations UfEC F3}
\end{eqnarray}%
It is easy to see, using $\left( \text{\ref{Realizations UfEC F1}}\right) $,
that $\Phi $ is $B_{\infty }^{\times }\ltimes B_{\infty }$-equivariant, thus
making the actions $\left( \text{\ref{Realizations UfEC F2}}\right) $ and $%
\left( \text{\ref{Realizations UfEC F3}}\right) $ correspond to each other:
%
%
\begin{equation}
\Phi \left( \left( \beta ,b\right) \cdot \left( \tau ,w\right) \right)
=\left( \beta ,b\right) \cdot \Phi \left( \left( \tau ,w\right) \right) 
\text{.}  \label{Realizations UfEC F4}
\end{equation}

\bigskip

Let us now give some constructions showing the existence of a $I$-rigidified
fake (analytic) elliptic curve as required in Step 2 of the proof of
Proposition \ref{Realizations P1}, explaining the relationship with the
level $K$ fake elliptic curves and making explicit the sheaves of locally
constant functions underlying the $\mathcal{L}\left( V\right) $'s and the
data obtained from the open and compact subgroups of $\widehat{B}^{\times }$.

\noindent\textbf{(Step 1)} Let $\Gamma \subset E\left( I\right) ^{\times }$
be a subgroup acting without fixed points on $\mathcal{P}$ and consider%
\begin{equation*}
\pi _{\Gamma ,I}:\mathcal{A}_{\Gamma ,I}:=\left( \Gamma \ltimes I\right)
\backslash \left( \mathcal{P}\times B_{\infty }\right) \rightarrow \Gamma
\backslash \mathcal{P}\text{,}
\end{equation*}%
where $\pi _{\Gamma ,I}$ is induced by the first projection. This is a morphism
of analytic spaces, since the fact that $\Gamma$ acts without fixed points on $\mathcal{P}$
makes the induced action of $\Gamma \ltimes I$ properly discontinuous.
This also implies that for any other $\Gamma ^{\prime }\subset \Gamma $
 the following diagram is cartesian
\begin{equation}
\begin{array}{ccccc}
\mathcal{A}_{\Gamma ^{\prime },I} & \rightarrow & \mathcal{A}_{\Gamma ,I} & 
&  \\ 
\pi _{\Gamma ^{\prime },I}\downarrow \text{ \ \ \ } &  & \text{ \ \ \ \ }%
\downarrow \pi _{\Gamma ,I} & \text{and} & \pi _{\Gamma ,I}=\Gamma
\backslash \pi _{\Gamma ^{\prime },I}\text{.} \\ 
\Gamma ^{\prime }\backslash \mathcal{P} & \rightarrow & \Gamma \backslash 
\mathcal{P} &  & 
\end{array}
\label{Realizations UfEC D1}
\end{equation}%
Note that, by construction, the second projection $\mathcal{P}\times \mathbb{%
C}^{2}\rightarrow \mathbb{C}^{2}$ induces%
\begin{equation}
\pi _{1,I}^{-1}\left( \tau \right) \overset{\sim }{\rightarrow }\Phi _{\tau
}\left( I\right) \backslash \mathbb{C}^{2}\text{.}
\label{Realizations UfEC D2}
\end{equation}%
More generally, because $\left( \text{\ref{Realizations UfEC D1}}\right) $
is cartesian, we see that $\left( \text{\ref{Realizations UfEC D2}}\right) $%
\ implies that we have a non-canonical bijection $\pi _{\Gamma
,I}^{-1}\left( \Gamma \tau \right) \simeq \Phi _{\tau }\left( I\right)
\backslash \mathbb{C}^{2}$. Note that, when $\Gamma =\left\{ 1\right\} $, we
have an evident morphism $\mathcal{A}_{1,I}\times _{\mathcal{P}}\mathcal{A}%
_{1,I}\rightarrow \mathcal{A}_{1,I}$ and a section making $\pi _{1,I}$ an
analytic abelian surface over $\mathcal{P}$. Indeed, writing $\mathcal{A}%
_{\Gamma ,I}=\Gamma \backslash \mathcal{A}_{1,I}$, we see that $\pi _{\Gamma
,I}$ has a natural structure of fake (analytic) $R\left( I\right) $-elliptic
curve over $\mathcal{P}$, with right $R\left( I\right) $-multiplication $%
i_{\Gamma ,I}:R\left( I\right) \rightarrow End_{\Gamma \backslash \mathcal{P}%
}\left( \mathcal{A}_{\Gamma ,I}\right) $\ induced by $\left( \tau ,w\right)
\mapsto \left( \tau ,wr\right) $.

\noindent\textbf{(Step 2)} When $\Gamma =\left\{ 1\right\} $, we remove the
subscript $\Gamma $ from the notation and the above construction yields%
\begin{equation*}
\pi _{I}:\mathcal{A}_{I}:=\left( 1\ltimes I\right) \backslash \left( 
\mathcal{P}\times B_{\infty }\right) \rightarrow \mathcal{P}
\end{equation*}%
and $i_{I}:R\left( I\right) \rightarrow End_{\mathcal{P}}\left( \mathcal{A}%
_{I}\right) $ which is further endowed with an $I$-rigidification given by
the identity $\rho _{I}:R^{1}\pi _{\ast }\mathbb{Z}_{\mathcal{A}}^{\vee }=I_{%
\mathcal{P}}$. It is easy to see that we have $R^{1}\pi _{\ast }\mathbb{%
Z}_{\mathcal{A}}^{\vee }=I_{\mathcal{P}}$ as rigidified quaternionic
variations of Hodge structures, i.e. the associated variation of Hodge
structures on $\mathcal{O}_{\mathcal{S}}\left( I\right) \simeq \mathcal{O}_{%
\mathcal{S}}\left( \mathbf{M}_{2}\left( \mathbb{R}\right) \right) $ is $%
\mathcal{L}_{1}\oplus \mathcal{L}_{1}$: indeed, suffices to look at the
complex structure on the fiber over $\tau $, which in view of $\left( \text{%
\ref{Realizations UfEC D2}}\right) $\ is obtained identifying $\mathbb{R}%
\otimes _{\mathbb{Q}}I\simeq \mathbf{M}_{2}\left( \mathbb{R}\right) \overset{%
\sim }{\rightarrow }\mathbb{C}^{2}$, via $b\mapsto b^{\iota }\left( 
\begin{array}{c}
\tau \\ 
1%
\end{array}%
\right) $; this is two copies of the complex structure obtained identifying $%
\mathbb{R}^{2}\overset{\sim }{\rightarrow }\mathbb{C}$ via $\left(
x,y\right) \mapsto x\tau +y$, i.e. it is two copies of the fiber of $%
\mathcal{L}_{1}$ over $\tau $. This completes the proof of Proposition \ref%
{Realizations P1} and then we know that $\left( \pi _{I},i_{I},\rho
_{I}\right) $ is the universal $I$-rigidified fake (analytic) elliptic curve.

We remark that the rule%
\begin{equation}
\beta ^{\iota }\left( \pi :\mathcal{A}\rightarrow \mathcal{S},i,\rho \right)
=\left( \pi :\mathcal{A}\rightarrow \mathcal{S},i,\rho \right) \beta
:=\left( \pi :\mathcal{A}\rightarrow \mathcal{S},i,\beta \circ \rho \right) 
\text{,}  \label{Realizations UfEC F6}
\end{equation}%
where $\beta \in E\left( I\right) ^{\times }$ denotes the $R\left( I\right) $%
-linear morphism given by left multiplication by $\beta $, induces a left $%
E\left( I\right) ^{\iota \times }$-action on $IqVHS$.
Applying this
definition to the universal family $\left( \pi _{I}:\mathcal{A}%
_{I}\rightarrow \mathcal{P},i_{I},\rho _{I}\right) $ gives a unique morphism 
$\left[ \beta ^{\iota }\right] :\mathcal{P}\rightarrow \mathcal{P}$ such
that $\left[ \beta ^{\iota }\right] ^{\ast }\left( \pi _{I}:\mathcal{A}%
_{I}\rightarrow \mathcal{P},i_{I},\rho _{I}\right) =\left( \pi _{I}:\mathcal{%
A}_{I}\rightarrow \mathcal{P},i_{I},\beta \circ \rho _{I}\right) $. The
multiplication by $j\left( \beta ^{\iota },\tau \right) $ induces an
isomorphism $\Phi _{\beta ^{\iota }\tau }\left( I\right) \backslash \mathbb{C%
}^{2}\simeq j\left( \beta ^{\iota },\tau \right) \Phi _{\beta ^{\iota }\tau
}\left( I\right) \backslash \mathbb{C}^{2}$ and we have $j\left( \beta
^{\iota },\tau \right) \Phi _{\beta ^{\iota }\tau }\left( I\right)
\backslash \mathbb{C}^{2}\simeq \Phi _{\tau }\left( \beta I\right)
\backslash \mathbb{C}^{2}=\Phi _{\tau }\left( I\right) \backslash \mathbb{C}%
^{2}$ in view of $\left( \text{\ref{Realizations UfEC F1}}\right) $; also,
it follows from $\left( \text{\ref{Realizations UfEC F1}}\right) $\ that,
under this isomorphism, the identification identity $H_{1}\left( \Phi
_{\beta ^{\iota }\tau }\left( I\right) \backslash \mathbb{C}^{2},\mathbb{Z}%
\right) =I$ corresponds to $H_{1}\left( \Phi _{\tau }\left( I\right)
\backslash \mathbb{C}^{2},\mathbb{Z}\right) =I\overset{\beta \cdot }{%
\rightarrow }I$. Hence $\left[ \beta ^{\iota }\right] \left( \tau \right)
=\beta ^{\iota }\tau $ and, by uniqueness, we see that $\left( \text{\ref%
{Realizations UfEC F2}}\right) $ (and $\left( \text{\ref{Realizations UfEC
F3}}\right) $) describes the resulting cartesian diagram $\mathcal{A}_{I}/%
\mathcal{P}\rightarrow \mathcal{A}_{I}/\mathcal{P}$. In particular, we see
that%
\begin{equation}
\pi _{\Gamma ,I}=\Gamma \backslash \pi _{I}\text{ with }\Gamma \subset
E\left( I\right) ^{\times }\text{ acting via }\left( \text{\ref{Realizations
UfEC F6}}\right) =\left( \text{\ref{Realizations UfEC F2}}\right) \text{.}
\label{Realizations UfEC F7}
\end{equation}%
Furthermore, if for an integer $N\geq 1$ we define%
$\Gamma _{I,N}:=\left\{ \gamma \in E\left( I\right) ^{\times }:\gamma \equiv 1%
\text{ \textrm{mod }}IN\right\}$,
then the rigidification $\rho _{I}:R^{1}\pi _{\ast }\mathbb{Z}_{\mathcal{A}%
}^{\vee }=I_{\mathcal{P}}$ yields a natural isomorphism $\rho _{I,N}:\mathcal{A}_{\Gamma _{I,N},I}\left[ N\right] =I\backslash N^{-1}I$.

\noindent\textbf{(Step 3)} Fix a maximal order $R$ and, for an integer $N\geq
1$, let $K_{N}\subset \widehat{R}^{\times }$ be the normal subgroup of
elements that are congruent to $1$ modulo $N$. Because $B$ has class number
one (by strong approximation), we have $\widehat{B}^{\times }=B^{\times }%
\widehat{R}^{\times }$ and we see that, for every $K\subset \widehat{R}%
^{\times }$, we have 
\begin{equation}
\pi _{0}\left( S_{K}\left( \mathbb{C}\right) \right) =B^{\times }\backslash
B^{\times }\widehat{R}^{\times }/K\overset{\sim }{\leftarrow }R^{\times
}\backslash \widehat{R}^{\times }/K\text{ and }\Gamma _{K}\left( x\right)
\subset R^{\times }  \label{Realizations UfEC F dec1}
\end{equation}%
choosing $x\in \widehat{R}^{\times }$ in the definition of $\Gamma
_{K}\left( x\right) $\ that appears in the decomposition $\left( \text{\ref%
{Realizations F decomposition}}\right) $. We may therefore apply the above
considerations \textbf{(Step 1)} and \textbf{(Step 2)} to $\pi _{\Gamma
_{K}\left( x\right) ,R}:\mathcal{A}_{\Gamma _{K}\left( x\right)
,R}\rightarrow \Gamma _{K}\left( x\right) \backslash \mathcal{P}$ assuming $K
$ is so small that $\Gamma _{K}\left( x\right) $ acts without fixed points
on $\mathcal{P}$. We define, for every $K\subset \widehat{R}^{\times }$,%
\begin{equation*}
\pi _{K,R}:\mathcal{A}_{K,R}:=\left( R^{\times }\ltimes R\right) \backslash
\left( \mathcal{P}\times B_{\infty }\times \widehat{R}^{\times }\right)
/K\rightarrow B^{\times }\backslash \left( \mathcal{P}\times \widehat{B}%
^{\times }\right) /K=S_{K}\left( \mathbb{C}\right) \text{,}
\end{equation*}%
where $\pi _{K}$ is induced by the first and the third projection, $%
R^{\times }\ltimes R$ acts via $\left( \beta ,b\right) \cdot \left( \tau
,w,x\right) =\left( \left( \beta ,b\right) \cdot \left( \tau ,w\right)
,\beta x\right) =\left( \beta \tau ,\beta w+b,\beta x\right) $ and the
action of $K$ is trivial on $\mathcal{P}\times B_{\infty }$ and by right
multiplication on $\widehat{R}^{\times }$. Choosing $x\in \widehat{R}%
^{\times }$ as in $\left( \text{\ref{Realizations UfEC F dec1}}\right) $, we
see that we have the mutually inverse bijections%
\begin{eqnarray}
&&p_{x}:\pi _{K.R}^{-1}\left( \pi _{0}^{-1}\left( \pi _{K}\left( x\right)
\right) \right) =\left( R^{\times }\ltimes R\right) \backslash \left(
R^{\times }\ltimes R\right) \left( \mathcal{P}\times B_{\infty }\times
xK/K\right) \overset{\sim }{\rightarrow }\Gamma _{K}\left( x\right)
\backslash \mathcal{A}_{R}=\mathcal{A}_{\Gamma _{K}\left( x\right) ,R} 
\notag \\
&&\iota _{x}:\mathcal{A}_{\Gamma _{K}\left( x\right) ,R}=\Gamma _{K}\left(
x\right) \backslash \mathcal{A}_{R}\overset{\sim }{\rightarrow }\pi
_{K}^{-1}\left( \pi _{0}^{-1}\left( \pi _{K}\left( x\right) \right) \right) 
\label{Realizations UfEC F decUfEC}
\end{eqnarray}%
defined by the rules 
\begin{equation*}
\left[ \tau ,w,xk\right] \mapsto \left[ \tau ,w\right] \text{ and }\left[
\tau ,w\right] \mapsto \left[ \tau ,w,x\right] \text{.}
\end{equation*}

\noindent\textbf{(Step 4)} Again assuming that $K\subset \widehat{R}^{\times
} $, let us show that $\pi _{K,R}:\mathcal{A}_{K,R}\rightarrow S_{K}\left( 
\mathbb{C}\right) $ is canonically identified with $\pi _{K}:A_{K}\left( 
\mathbb{C}\right) \rightarrow S_{K}\left( \mathbb{C}\right) $ for $K$ small
enough . Indeed, the general result follows from the case $K=K_{N}$ (because 
$\left\{ K_{N}\right\} $ is cofinal, hence we may choose $K_{N}\subset K$
and take $K$-invariants from $\pi _{K_{N},R}=\pi _{K_{N}}$ to get $\pi
_{K,R}=\pi _{K}$). In particular, we have 
\begin{equation*}
\Gamma _{K_{N}}\left( x\right) :=xK_{N}x^{-1}\cap B^{\times }=K_{N}\cap
B^{\times }=:\Gamma _{R,N}\text{,}
\end{equation*}%
the subgroup of the elements of $\Gamma _{R}:=\widehat{R}^{\times }\cap
B^{\times }$ that are congruent to one modulo $N$. As explained in \cite[Ch.
III, Th\'{e}or\`{e}me $\left( 1.1\right) $]{BC}, $S_{K_{N}}$ classifies fake 
$R$-elliptic curves $\left( \pi :A\rightarrow S,i\right) $\ together with an
isomorphism of right $R$-modules $\rho _{N}:A\left[ N\right] \simeq \frac{%
N^{-1}R}{R}$: we identify $S_{K_{N}}\left( S\right) $ with the set of
isomorphism classes of these triples and, abusively, for an analytic space $%
\mathcal{S}$ we write $S_{K_{N}}\left( \mathcal{S}\right) $ for the
corresponding moduli problem in the category of analytic spaces. Let us
consider a subfunctor $\mathcal{S}_{K_{N}}^{\circ }\subset S_{K_{N}}$
defined as follows: it is characterized by the fact that, if $\mathcal{S}$ is
connected, it classifies triples $\left( \pi :\mathcal{A}\rightarrow 
\mathcal{S},i,\rho _{N}\right) $ with the property that, writing $\widetilde{%
\mathcal{S}}\rightarrow \mathcal{S}$ for the universal cover and $\left( 
\widetilde{\pi }:\widetilde{\mathcal{A}}\rightarrow \widetilde{\mathcal{S}},%
\widetilde{i},\widetilde{\rho _{N}}\right) $ for the pull-back of $\left(
\pi :\mathcal{A}\rightarrow \mathcal{S},i,\rho _{N}\right) $ to $\widetilde{%
\mathcal{S}}$, there is a rigidification $\widetilde{\rho }:R^{1}\widetilde{%
\pi }_{\ast }\mathbb{Z}_{\widetilde{\mathcal{A}}}^{\vee }\overset{\sim }{%
\rightarrow }R_{\widetilde{\mathcal{S}}}$ which lifts $\widetilde{\rho _{N}}$.
%
We remark that $\left( \frac{R}{NR}\right) ^{\times }$ acts
simply transitively from the right on the set of isomorphism $\widetilde{%
\rho _{N}}:\widetilde{\mathcal{A}}\left[ N\right] \simeq \frac{N^{-1}R}{R}$
(resp. $\rho _{N}:\mathcal{A}\left[ N\right] \simeq \frac{N^{-1}R}{R}$)\
similarly as in $\left( \text{\ref{Realizations UfEC F6}}\right) $ and, by Remark \ref{Realizations R1}, there is always a
rigidification $R^{1}\widetilde{\pi }_{\ast }\mathbb{Z}_{\widetilde{\mathcal{%
A}}}^{\vee }\overset{\sim }{\rightarrow }R_{\widetilde{\mathcal{S}}}$, which
then induces $x\circ \widetilde{\rho _{N}}$ for some $x\in \left( \frac{R}{NR%
}\right) ^{\times }$: in other words, $S_{K_{N}}=\tbigcup\nolimits_{x\in
\left( \frac{R}{NR}\right) ^{\times }}\mathcal{S}_{K_{N}}^{\circ }x$.
Indeed, we can refine the union as follows. First, we remark that $\mathcal{S%
}_{K_{N}}^{\circ }r=\mathcal{S}_{K_{N}}^{\circ }$ for every $r$ coming from $%
R^{\times }$ because $R^{\times }$ acts on the rigidifications via $\left( 
\text{\ref{Realizations UfEC F6}}\right) $, implying that $%
S_{K_{N}}=\tbigcup\nolimits_{x\in R^{\times }\backslash \widehat{R}^{\times
}/K_{N}}\mathcal{S}_{K_{N}}^{\circ }x$. Suppose that $\left( \pi
:A\rightarrow S,i,\rho _{N}\right) $ is an $\mathcal{S}$ point in $\mathcal{S%
}_{K_{N}}^{\circ }x_{1}\cap \mathcal{S}_{K_{N}}^{\circ }x_{2}$, meaning that 
$x_{1}\circ \widetilde{\rho _{N}}$ and $x_{2}\circ \widetilde{\rho _{N}}$
lift to rigidifications $\widetilde{\rho _{1}},\widetilde{\rho _{2}}:R^{1}%
\widetilde{\pi }_{\ast }\mathbb{Z}_{\widetilde{\mathcal{A}}}^{\vee }\overset{%
\sim }{\rightarrow }R_{\widetilde{\mathcal{S}}}$. Because the $R^{\times }$%
-action $\left( \text{\ref{Realizations UfEC F6}}\right) $ on the set of
rigidification is simply transitive, we may write $\widetilde{\rho _{2}}=r\circ 
\widetilde{\rho _{1}}$ for some $r\in R^{\times }$ and then we see that $%
x_{2}\circ \widetilde{\rho _{N}}=r\circ x_{1}\circ \widetilde{\rho _{N}}$,
which implies $x_{2}=r\circ x_{1}$ in $\left( \frac{R}{NR}\right) ^{\times }=%
\widehat{R}^{\times }/K_{N}$. In other words,%
\begin{equation}
S_{K_{N}}=\tbigsqcup\nolimits_{x\in R^{\times }\backslash \widehat{R}%
^{\times }/K_{N}}\mathcal{S}_{K_{N}}^{\circ }x\text{.}
\label{Realizations UfEC F dec2}
\end{equation}

We would like to understand $\pi _{K_{N}}:A_{K_{N}}\left( \mathbb{C}\right)
\rightarrow S_{K_{N}}\left( \mathbb{C}\right) $ as a morphism of analytic
spaces. Comparing $\left( \text{\ref{Realizations F decomposition}}\right) $
with $\left( \text{\ref{Realizations UfEC F dec1}}\right) $ and $\left( 
\text{\ref{Realizations UfEC F dec2}}\right) $, suffices to understand $%
\mathcal{S}_{K_{N}}^{\circ }$. Let us identify $\Gamma _{R,N}\backslash 
\mathcal{P\simeq S}_{K_{N}}^{\circ }$ by showing that%
\begin{equation*}
\left( \pi _{\Gamma _{R,N},R}:\mathcal{A}_{\Gamma _{R,N},R}\rightarrow
\Gamma _{R,N}\backslash \mathcal{P},i_{\Gamma _{R,N},R},\rho _{R,N}\right) 
\end{equation*}%
is the universal object that classifies the triples $\left( \pi :\mathcal{A}%
\rightarrow \mathcal{S},i,\rho _{N}\right) \in \mathcal{S}_{K_{N}}^{\circ
}\left( \mathcal{S}\right) $ as above (see \cite[Theorem 6.1.10]{C} for the
analogue of this result in the modular case and \cite[Theorem 3 and
Proposition 15]{Sh} or \cite[Propositions 1.7, 1.11]{P}
for a description of the fibers of $\pi _{\Gamma _{N},R}$ as a
classifying space). Suppose we are given an $R$-fake (analytic) elliptic
curve over $\mathcal{S}$ with full level $N$ structure $\left( \pi :\mathcal{%
A}\rightarrow \mathcal{S},i,\rho _{N}\right) $, where $\mathcal{S}$ is
connected. Let $\widetilde{\mathcal{S}}\rightarrow \mathcal{S}$ be the
universal cover, write $\left( \widetilde{\pi }:\widetilde{\mathcal{A}}%
\rightarrow \widetilde{\mathcal{S}},\widetilde{i},\widetilde{\rho _{N}}%
\right) $ for the pull-back of $\left( \pi :\mathcal{A}\rightarrow \mathcal{S%
},i,\rho _{N}\right) $ to $\widetilde{\mathcal{S}}$ and choose a
rigidification $\widetilde{\rho }:R^{1}\widetilde{\pi }_{\ast }\mathbb{Z}_{%
\widetilde{\mathcal{A}}}^{\vee }\overset{\sim }{\rightarrow }R_{\widetilde{%
\mathcal{S}}}$ which lifts $\widetilde{\rho _{N}}$: two different lifts
being uniquely determined up to replacing $\widetilde{\rho }$ with $\beta
\circ \widetilde{\rho }$ for some $\beta \in \Gamma _{R,N}$. We remark that
the rigidification $\widetilde{\rho }$ yields a representation%
\begin{equation}
\pi _{1}\left( \mathcal{S}\right) \rightarrow \operatorname{Aut}\left( R^{1}\widetilde{\pi }%
_{\ast }\mathbb{Z}_{\widetilde{\mathcal{A}}}^{\vee }\right) \simeq \operatorname{Aut}\left(
R_{\widetilde{\mathcal{S}}}\right) =E\left( R\right) ^{\times }=R^{\times }%
\text{,}  \label{Realizations UfEC F8}
\end{equation}%
whose image is contained in $\Gamma _{R,N}$ because the elements of $\pi
_{1}\left( \mathcal{S}\right) $ act as the identity on $\widetilde{\rho _{N}}%
:\widetilde{\mathcal{A}}\left[ N\right] \simeq \frac{N^{-1}R}{R}$ which
comes from the constant sheaf $\rho _{N}:\mathcal{A}\left[ N\right] \simeq 
\frac{N^{-1}R}{R}$.
By Proposition %
\ref{Realizations P1}, there is a unique morphism $\widetilde{x}:\widetilde{%
\mathcal{S}}\rightarrow \mathcal{P}$ such that $\widetilde{x}^{\ast }\left(
\pi _{R}:\mathcal{A}_{R}\rightarrow \mathcal{P},i_{R},\rho _{R}\right)
=\left( \widetilde{\pi }:\widetilde{\mathcal{A}}\rightarrow \widetilde{%
\mathcal{S}},\widetilde{i},\widetilde{\rho }\right) $, which is $\pi
_{1}\left( \mathcal{S}\right) $-equivariant with respect to $\left( \text{%
\ref{Realizations UfEC F8}}\right) $ (because $\tilde x^*(\tilde \rho) = 
\rho_R$ and $\left( \text{\ref{Realizations
UfEC F6}}\right) =\left( \text{\ref{Realizations UfEC F2}}\right) $ in $%
\left( \text{\ref{Realizations UfEC F7}}\right) $). In particular, writing $%
\rho _{R,N}:\mathcal{A}_{R}\left[ N\right] =R\backslash N^{-1}R$ for the
identification induced by $\rho _{R}$ (same notation already in force for $%
\mathcal{A}_{\Gamma _{R,N},R}$), we have $\widetilde{x}^{\ast }\left( \rho
_{R,N}\right) =\widetilde{\rho _{N}}$. We see that we have constructed a
commutative diagram%
\begin{equation}
\begin{array}{ccc}
\left( \widetilde{\pi }:\widetilde{\mathcal{A}}\rightarrow \widetilde{%
\mathcal{S}},\widetilde{i},\widetilde{\rho _{N}}\right)  & \overset{%
\widetilde{x}}{\rightarrow } & \left( \pi _{R}:\mathcal{A}_{R}\rightarrow 
\mathcal{P},i_{R},\rho _{R,N}\right)  \\ 
\downarrow  &  & \downarrow  \\ 
\left( \pi :\mathcal{A}\rightarrow \mathcal{S},i,\rho _{N}\right)  & \overset%
{x}{\dashrightarrow } & \left( \pi _{\Gamma _{R,N},R}:\mathcal{A}_{\Gamma
_{R,N},R}\rightarrow \Gamma _{R,N}\backslash \mathcal{P},i_{\Gamma
_{R,N},R},\rho _{R,N}\right) 
\end{array}
\label{Realizations UfEC D3}
\end{equation}%
in which all the arrows are cartesian (the most right because $\left( \text{%
\ref{Realizations UfEC D1}}\right) $ is cartesian and by definition of the $%
\rho _{R,N}$'s) and we are looking for the dotted arrow $x$. First, we
remark that it exists because $\pi =\pi _{1}\left( \mathcal{S}\right)
\backslash \widetilde{\pi }$, the image of $\pi _{1}\left( \mathcal{S}%
\right) $ in $R^{\times }$ via $\left( \text{\ref{Realizations UfEC F8}}%
\right) $ is contained in $\Gamma _{R,N}$ and $\widetilde{x}$ is $\pi
_{1}\left( \mathcal{S}\right) $-equivariant with respect to $\left( \text{%
\ref{Realizations UfEC F8}}\right) $. It is a cartesian arrow because the
left vertical arrow is both surjective on the base $\widetilde{\mathcal{S}}%
\rightarrow \mathcal{S}$ and cartesian and the composition of it with $x$ is
cartesian 
(by the commutativity of $\left( \text{\ref%
{Realizations UfEC D3}}\right) $, because $\widetilde{x}$ and the right
vertical arrow are cartesian). Because $x$ is uniquely determined (since the
left vertical arrow is surjective)\ and, as remarked, two lifts of $%
\widetilde{\rho _{N}}$ differs by some $\beta \in \Gamma _{R,N}$ which
leaves $x$ unchanged, also the uniqueness of $x$ is proved.

\noindent\textbf{(Step 5)} If $V$ is a left $B^{\times }$-representation,
then we define $L(V) := \mathcal{P}\times V\to \mathcal{P}$.
If $K\subset \widehat{B}^{\times }$ is an open and compact subgroup, we may
form%
\begin{equation*}
L_{K}\left( V\right) :=\left( B^{\times }\ltimes 1\right) \backslash \left( 
\mathcal{P}\times B_{\infty }\times \widehat{B}^{\times }\right)
/K\rightarrow B^{\times }\backslash \left( \mathcal{P}\times \widehat{B}%
^{\times }\right) /K=S_{K}\left( \mathbb{C}\right) \text{.}
\end{equation*}%
We identify $L\left( V\right) $ and $L_{K}\left( V\right) $ with the
associated sheaf of sections. It is clear that we have $\mathcal{L}\left(
V\right) =L\left( V\right) $ and $\mathcal{L}_{K}\left( V\right)
=L_{K}\left( V\right) $, where we abusively write $\mathcal{L}\left(
V\right) $ and $\mathcal{L}_{K}\left( V\right) $ to denote the underlying
sheaves of locally constant functions\footnote{%
Indeed, we may also work with $B_{\infty }^{\times }$-representations, apply 
\cite[Corollary 3.2]{Ha} to construct $V\mapsto L\left( V\right) $ (resp. $%
L_{K}\left( V\right) $)\ by means of an $L_{\mathbb{R}}$\ as we did for $%
\mathcal{L}$ (resp. $\mathcal{L}_{K}$) and the uniqueness implies that we
have just to remark that $\mathcal{L}_{\mathbb{R}}\left( L_{1,\mathbb{R}%
}\right) =L_{\mathbb{R}}\left( L_{1,\mathbb{R}}\right) $.}. Taking $V=B$, it
is clear from $\mathcal{A}_{R}=\left( 1\ltimes R\right) \backslash \left( 
\mathcal{P}\times B_{\infty }\right) $ that we have $L\left( B\right)
=R^{1}\pi _{\ast }\mathbb{Q}_{\mathcal{A}}^{\vee }$, thus confirming $\left( 
\text{\ref{Realizations F key}}\right) $.

\bigskip

Finally, the proof of Proposition~\ref{Realizations UfEC L1} is straightforward. Given $K$ small enough,
we have $K\subset \widehat{R}^{\times }$ for some maximal order $R$. It then
 follows from \textbf{(Step 4)} that we have $R^{1}\pi _{K\ast }\mathbb{Q}%
_{A_{K}}^{\vee }=R^{1}\pi _{K,R\ast }\mathbb{Q}_{\mathcal{A}_{K,R}}^{\vee }$%
, where the right hand side is obtained descending the sheaves $R^{1}\pi
_{R\ast }\mathbb{Q}_{\mathcal{A}_{R}}^{\vee }$, in view of $\left( \text{\ref%
{Realizations UfEC F decUfEC}}\right) $ and the cartesian diagram $\left( 
\text{\ref{Realizations UfEC D1}}\right) $. This concludes the proof.

\end{document}